\documentclass{article}

\usepackage[  colorlinks=true,%
citecolor=black,%
filecolor=black,%
linkcolor=black,%
urlcolor=blue,%
pdffitwindow=false,%
pdfauthor={S. Congreve, J. Gedicke, and I. Perugia},%
pdftitle={Robust adaptive hp discontinuous Galerkin finite element methods for the Helmholtz equation},%
pdftex]{hyperref}

\usepackage[a4paper]{geometry}						  

\usepackage{amsfonts}
\usepackage{amssymb}
\usepackage{amsmath}
\usepackage{amsthm}
\usepackage{amsopn}
\usepackage{graphicx}
\graphicspath{{./pictures/}}
\usepackage{subfigure}
\usepackage{bm}
\usepackage{fancyhdr}
\usepackage{esint}
\usepackage{doi}
\usepackage{soul}
\usepackage{algorithm,algorithmicx,algpseudocode}

\DeclareMathOperator{\ddiv}{div}

\DeclareMathOperator{\osc}{osc}

\DeclareMathOperator*{\argmin}{arg\,min}
\DeclareMathOperator{\rot}{\mathbf{rot}}

\newcommand{\E}{\mathcal{E}}
\newcommand{\T}{\mathcal{T}}
\newcommand{\G}{\mathcal{G}}
\newcommand{\N}{\mathcal{N}}
\newcommand{\Pp}[1]{\mathbb{P}_{#1}}
\newcommand{\h}{\mathsf{h}}
\newcommand{\p}{\mathsf{p}}
\newcommand{\n}{\bm{n}}

\newcommand{\jump}[1]{\lbrack\!\lbrack #1 \rbrack\!\rbrack_N}

\newcommand{\average}[1]{\lbrace\!\!\lbrace #1 \rbrace\!\!\rbrace}

\theoremstyle{plain}
\newtheorem{theorem}{Theorem}[section]
\newtheorem{lemma}[theorem]{Lemma}

\theoremstyle{remark}
\newtheorem{remark}{Remark}[section]
\newtheorem{definition}{Definition}[section]

\numberwithin{theorem}{section}
\numberwithin{figure}{section}
\numberwithin{equation}{section}

\newcommand{\TheTitle}{Robust adaptive $hp$ discontinuous Galerkin
  finite element methods for the Helmholtz equation} 
\newcommand{\TheShortTitle}{Robust adaptive $hp$-DG for Helmholtz} 
\title{{\TheTitle}%
\thanks{
The work of the authors has been funded by the Austrian Science Fund (FWF) 
through the projects F~65 and P~29197-N32, and by the Vienna Science and Technology Fund (WWTF) through the project MA14-006.}
}

\newcommand{\TheAuthors}{S. Congreve, J. Gedicke, and I. Perugia}
\author{
Scott Congreve\thanks{University of Vienna, Faculty of Mathematics, Oskar-Morgenstern-Platz 1, 1090 Vienna, Austria, {\tt scott.congreve@univie.ac.at}} 
\and 
Joscha Gedicke\thanks{University of Vienna, Faculty of Mathematics, Oskar-Morgenstern-Platz 1, 1090 Vienna, Austria, {\tt joscha.gedicke@univie.ac.at}}
\and 
Ilaria Perugia\thanks{University of Vienna, Faculty of Mathematics, Oskar-Morgenstern-Platz 1, 1090 Vienna, Austria, {\tt ilaria.perugia@univie.ac.at}}
}

\date{}

\begin{document}
\thispagestyle{empty}

\maketitle

\begin{abstract}
This paper presents an $hp$ a posteriori error analysis for the 2D Helmholtz equation that is robust in the polynomial degree $p$ and the wave number $k$. 
For the discretization, we consider a discontinuous Galerkin formulation that is unconditionally well posed.
The a posteriori error analysis is based on the  technique of equilibrated fluxes applied to a shifted Poisson problem, with
the error due to the nonconformity of the discretization controlled by a potential reconstruction.
We prove that the error estimator is both reliable and efficient, under the condition that the initial mesh size and polynomial degree
is chosen such that the discontinuous Galerkin formulation converges, i.e., it is out of the regime of pollution.
We confirm the efficiency of an $hp$-adaptive refinement strategy based on the presented robust a posteriori error estimator
via several numerical examples.
\end{abstract}

{\small\noindent\textbf{Keywords}
a posteriori error analysis, $hp$ discontinuous Galerkin finite element method, equilibrated fluxes, potential reconstruction, Helmholtz problem

}

\section{Introduction}
\label{sec:Intro}
In this paper, we consider the following Helmholtz problem with impedance boundary
condition: Find a (complex) solution $u\in H^2(\Omega)$ such that
\begin{align}\label{problem}
\begin{split}
-\Delta u - k^2 u &= f\quad\textrm{in }\Omega,\\
\nabla u\cdot\n - iku &= g\quad\textrm{on }\partial\Omega,
\end{split}
\end{align}
where $\Omega\subset \mathbb{R}^2$ is a bounded, Lipschitz domain,
$\n$ denotes the outer unit normal on the boundary $\partial\Omega$,
$f\in L^2(\Omega)$, $g\in L^2(\partial\Omega)$, and $k>0$ is the (constant) wavenumber.
\par
The problem \eqref{problem} was shown to be well-posed in \cite{Melenk1995}.
A polynomial-based discontinuous Galerkin (DG) approximation was presented in \cite{MPS2013}, 
which uses the same numerical fluxes as in the ultra weak
variational formulation/plane wave DG methods
\cite{CD1998,CD2003,GHP2009}.
A DG discretization with stabilization terms also containing jumps in high order derivatives was presented in
\cite{FengWu}.
A residual-based a posteriori error estimator for the DG 
method of~\cite{MPS2013} is derived and analyzed in \cite{SZ2015}. 
\par
In this paper we will develop an a posteriori error
estimator based on a local reconstruction of equilibrated fluxes \cite{DEV2016,EV2015}.
Since \eqref{problem} is highly indefinite, it is not clear how to
localize the Helmholtz problem in order to obtain localized 
problems for the error approximation that are well posed.
However, as noted in \cite{BISG1997}, the error has two components,
the interpolation error and the pollution error. While the pollution error is global and
hence cannot be estimated with local error indicators, it is possible to derive
equilibrated a posteriori error estimators for the interpolation error.
\par
This analysis is based on considering a shifted Poisson problem with inhomogeneous Neumann
boundary conditions. Therefore, we can apply the unified framework for equilibrated fluxes \cite{EV2015}
to this auxiliary elliptic problem with an extension for the extra terms resulting
from the handling of the inhomogeneous Robin boundary condition by the DG method. Additionally,
an extra lifting operator is required due to the additional gradient stabilization terms in the 
DG formulation for Helmholtz.
In order to measure the nonconformity of the DG method we locally reconstruct
a conforming potential approximation.
\par
By construction, the a posteriori error estimator captures possible singularities of the solution
correctly, but
is only reliable up to an additional $L^2$ error which resembles
the pollution error.
Note that also the residual a posteriori error estimator for the DG
method in \cite{SZ2015} is only reliable up to the pollution
error, see \cite[Lemma 3.2]{SZ2015}.
\par
We will apply the theory of equilibrated flux and potential
reconstructions \cite{DEV2016,EV2015}
and
derive the a posteriori error estimator of the form
\begin{align*}
\eta_{hp}^2 &:= 
\sum_{T\in\T} \left(
 \|\G(u_{hp})+\bm{\sigma}_{hp}\|_{0,T} 
+ \frac{h_T}{j_{1,1}}\|f + k^2 u_{hp} -\ddiv\bm{\sigma}_{hp}\|_{0,T} \right.\\
&\quad
+ \left.C_{tr}\!\!\!\!\!\sum_{E\in\E(T)\cap \E(\partial\Omega)}\!\!\!\!\! h_E^{1/2}\|\bm{\sigma}_{hp}\cdot\bm{n} + g + iku_{hp} - \gamma k \frac{\h}{\p}(g - \nabla_h u_{hp} \cdot \n + iku_{hp})\|_{0,E}\right)^2\\
&\quad
+ \sum_{T\in\T}\|\G(u_{hp})-\nabla s_{hp}\|_{0,T}^2,
\end{align*}
where $\G(u_{hp})$ denotes a discrete gradient (which we call the \emph{DG gradient}), $\bm{\sigma}_{hp}$ an 
equilibrated flux reconstruction,
and $s_{hp}$ a  
potential reconstruction.
The parameter $\gamma$, as well as the mesh function $\h$ 
and the polynomial degree function $\p$ already enter the definition
of the DG methods (see~\eqref{discreteproblem} below), $h_T$ and $h_E$
are the diameter of the element $T$ of the mesh $\T$ and the edge
$E$ of $T$, respectively, $C_{tr}$ is a trace inequality constant, cf. Lemma~\ref{lem:tracebound}, and $j_{1,1}$ is the first positive root of the Bessel
function of the first kind.
We prove that the a posteriori error estimator is reliable 
and efficient, for suitably chosen functions $\bm\sigma_{hp}$ and $s_{hp}$,
up to generic constants which are independent of the wave number,
the polynomial degrees, and the element sizes.
\par
This paper is organized as follows.
In Section~\ref{sec:dg}, we will recall 
the DG method from \cite{MPS2013}.
In Section~\ref{sec:aposteriori}, we will present the a posteriori error estimator and prove its
reliability for any admissible flux and potential reconstructions. 
In Section~\ref{sec:efficiency},
we define specific local reconstructions of flux and potential functions,
such that the error estimator is efficient.
Finally, in Section~\ref{sec:numerics}, we present some numerical experiments.
\par
Throughout this paper, we employ the standard notation for  
(complex) Sobolev spaces $H^m(\omega)$ 
with norm $\|\cdot\|_{m,\omega}$ for (sub)-domains $\omega\subseteq\Omega$, and define $H(\ddiv;\Omega)=\{\bm{\tau}\in [L^2(\omega)]^2 : \ddiv \bm{\tau}\in L^2(\omega)\}$. 
We denote the (complex) $L^2$ inner product by $(\cdot,\cdot)_{\omega}$; 
if $\omega=\Omega$ we simply write $(\cdot,\cdot)$.
The (complex) $L^2$ inner product on the boundary is indicated by a subscript, e.g. $(\cdot,\cdot)_{\partial\omega}$.
By $\lesssim$, we abbreviate the inequality $x\leq C y$, with a generic constant $C$
independent of the wave number, the mesh size, and the polynomial degree,
but possibly dependent on the shape regularity of the mesh.

\section{The discontinuous Galerkin method}\label{sec:dg}
In this section, we discuss a numerical approximation to \eqref{problem}
based on employing an $hp$-version DG finite element method. We consider the same formulation as in \cite{MPS2013}.

The weak formulation of \eqref{problem} is defined as follows: Find $u\in H^1(\Omega)$ such that
\begin{align}\label{weak:formulation}
	a(u,v) = F(v)\quad\textrm{for all } v\in H^1(\Omega),
\end{align}
with the complex-valued sequilinear form $a(\cdot,\cdot)$ and linear form $F(\cdot)$ given by
\begin{align*}
a(u,v) := (\nabla u , \nabla v) - k^2 (u,v) - ik(u,v)_{\partial\Omega}
\quad\textrm{and}\quad
F(v) := (f,v) + (g, v)_{\partial\Omega}.
\end{align*}
\par
Let $\T$ be a triangulation of $\Omega$ with the set of nodes $\N$ and the set of edges $\E$.
For simplicity of the presentation we restrict ourselves to shape-regular conforming triangulations.
Let $\E(\Omega)$  and $\E(\partial\Omega)$
denote the subset of interior and boundary edges, respectively, and let $\E(T)$ denote the edges of the element
$T\in\T$. Let $\N(\partial\Omega)$ denote the
subset of nodes on the boundary of $\Omega$, $\N(T)$ denote the set of nodes of an element $T\in\T$, and $\N(E)$ the set of nodes of an edge $E\in\E$.
The subset of triangles that share a common node $z\in\N$ is denoted
by $\T(z)$, and the
subset of edges sharing the node $z$ by $\E(z)$.
For any node $z\in\N$, we denote by $\omega_z\subseteq\Omega$ the union of triangles that share the
node $z$. The set of triangles that share a common edge $E\in\E(\Omega)$ is denoted by $\T(E)$. For any
$E\in\E(\Omega)$, we denote by $\omega_E\subseteq\Omega$ the union of the two triangles $T_\pm\in\T$ that
share the edge $E$; we set $\omega_E=T$ for
$E\in\E(\partial\Omega)$. We denote by $h_T$ and $h_E$ the
diameter of $T$ and the length of $E$, respectively.
\par
We make use of the standard notation on averages and jumps of
scalar functions $v$ across edges $E\in\E(\Omega)$ with $E=\partial T_+\cap \partial T_-$
\begin{align}
\average{v} := \frac{1}{2}\left( v|_{T_+} + v|_{T_-} \right),
\qquad
\jump{v} := v|_{T_+}\bm{n}_{+} + v|_{T_-}\bm{n}_{-},
\end{align}
and, for vector-valued functions $\bm{\tau}$,
\begin{align*}
\average{\bm{\tau}} := \frac{1}{2}\left( \bm{\tau}|_{T_+} + \bm{\tau}|_{T_-} \right),
\qquad
\jump{\bm{\tau}} := \bm{\tau}|_{T_+}\cdot\bm{n}_{+} + \bm{\tau}|_{T_-}\cdot\bm{n}_{-},
\end{align*}
where $\bm{n}_{\pm}$ denotes the unit outer normal vector of $T_\pm$.
For any scalar function $v=v(x_1,x_2)$ we denote by $\rot v=[\frac{\partial v}{\partial x_2}, -\frac{\partial v}{\partial x_1}]^\top$ the \emph{rotation} of $v$, and
we denote the elementwise application of the gradient and rotation by $\nabla_h$ and $\rot_h$, respectively, i.e., $(\nabla_h \cdot)|_T = \nabla(\cdot)|_T$ and $(\rot_h \cdot)|_T = \rot(\cdot)|_T$
for all $T\in\T$.
\par
Let $V_{hp}$ denote the discontinuous finite element space of piecewise polynomial basis functions
\begin{align*}
V_{hp} := \{ v_{hp}\in L^2(\Omega)\,:\, v_{hp}|_T \in \Pp{p_T}(T) \text{ for all } T\in \T \},
\end{align*}
where $\Pp{p_T}(T)$ denotes the space of polynomials of degree less
than or equal to $p_T\geq 1$ on  
a triangle $T\in\T$.
Let us denote by $\h$ and $\p$ the piecewise constant 
mesh size function and polynomial degree function, respectively, 
defined on the mesh interfaces as follows:
 $\h|_E = \min(h_{T_+},h_{T_-})$ and $\p|_E = \max(p_{T_+},p_{T_-})$,
 if $E=\partial T_+\cap \partial T_-$, or 
$\h|_E =h_{T}$ and $\p|_E =p_{T}$, if $E=\partial
T\cap\partial\Omega$.
\par
The discrete problem then reads: 
Find $u_{hp}\in V_{hp}$ such that
\begin{align}\label{discreteproblem}
	a_{hp}(u_{hp},v_{hp}) = F_{hp}(v_{hp})\quad\textrm{for all } v_{hp} \in V_{hp},
\end{align}
where
\begin{align*}
	a_{hp}(u,v)  &:= (\nabla_h u, \nabla_h v) - k^2(u,v)\\
		&\quad - \sum_{E\in\E(\Omega)}(\jump{u},\average{\nabla_h v})_{E} - \sum_{E\in\E(\Omega)}(\average{\nabla_h u},\jump{v})_{E}\\
		&\quad -\left( \gamma k\frac{\h}{\p} u, \nabla_h v\cdot\n\right)_{\partial\Omega} 
		-\left( \gamma k\frac{\h}{\p}\nabla_h u\cdot\n,v\right)_{\partial\Omega}\\	
		&\quad -i  
\sum_{E\in\E(\Omega)}\left( \beta\frac{\h}{\p}\jump{\nabla_hu},\jump{\nabla_h v} \right)_{E}
		            - i\sum_{E\in\E(\Omega)}\left( \alpha\frac{\p^2}{\h}\jump{u},\jump{v}\right)_{E}\\	
		&\quad -i 
\left( \gamma \frac{\h}{\p}\nabla_h u\cdot \n,\nabla_h v\cdot \n\right)_{\partial\Omega}
		            - i\left( k(1-\gamma k\frac{\h}{\p})u,v\right)_{\partial\Omega},
\end{align*}
and
\begin{align*}
	F_{hp}(v) := (f,v) -i \left( \frac{\gamma\h}{\p} g, \nabla_h v\cdot\n\right)_{\partial\Omega}
	            + \left( (1-\gamma k\frac{\h}{\p})g,v\right)_{\partial\Omega}.
\end{align*}
The constants $\alpha>0$, $\beta>0$, and $0<\gamma<1/3$ are fixed constants.
Note that $\beta>0$ guarantees the unconditional well posedness of the discrete problem; cf. \cite{MPS2013}.
\par
In order to define the DG gradient, see Definition~\ref{def:gradient} below,
we need to introduce two lifting operators. 
For any $E\in\E(\Omega)$, let
\[
 \Pp0(\T(E))^2:=\{v_{hp} \in [L^2(\omega_E)]^2\,:\, v_{hp}|_T \in [\Pp0(T)]^2 \text{ for all } T\in \T(E)\};
\]
then, we define $\mathcal{L}_E^0\in \Pp0(\T(E))^2$ as
\begin{align*}
	\int_{\omega_E} \mathcal{L}_E^0(\jump{v_{hp}})\cdot\overline{\bm{\tau}}_{hp}\, dx 
	= \int_E\jump{v_{hp}}\cdot\average{\overline{\bm{\tau}}_{hp}}\, ds 
\end{align*}
for all $\bm{\tau}_{hp}\in \Pp0(\T(E))^2$,
and $\mathcal{L}_E^1\in \Pp0(\T(E))^2$ as
\begin{align*}
	\int_{\omega_E} \mathcal{L}_E^1(\jump{\nabla_h v_{hp}})\cdot\overline{\bm{\tau}}_{hp}\, dx 
	= i 
\beta\int_E\frac{\h}{\p}
	\jump{\nabla_h v_{hp}}\jump{\overline{\bm{\tau}}_{hp}}\, ds 
\end{align*}
for all $\bm{\tau}_{hp}\in \Pp0(\T(E))^2$.
\par
For a given integer $p\geq 0$, let $\Pi_E^p: L^2(E) \to \Pp{p}(E)$ denote the local $L^2$-orthogonal
projection onto the space of polynomials of degree at most $p$
along the edge $E\in\E$. 
Similarly we define $\Pi_T^p: L^2(T) \to \Pp{p}(T)$ to be the local
$L^2$-orthogonal projection onto the space of polynomials
of degree at most $p$ on a triangle $T\in\T$.
\par
We can derive the following stability estimates following the lines of
the proof of \cite[Proposition 4.2]{PS2003}.
\begin{lemma}\label{lemma:stabilitylifting}
The lifting operators $\mathcal{L}_E^0$ and $\mathcal{L}_E^1$ are stable in the sense that
\begin{align*}
\|\mathcal{L}_E^0(\jump{v_{hp}})\|_{0,T} &\lesssim h_E^{-1/2}\|\Pi_E^0(\jump{v_{hp}})\|_{0,E},\\
\|\mathcal{L}_E^1(\jump{\nabla v_{hp}})\|_{0,T} &\lesssim
  \beta h_E^{1/2}\|\p^{-1}\Pi_E^0(\jump{\nabla v_{hp}})\|_{0,E},
\end{align*}
for $T=T_\pm$, where $T_\pm$ are the two elements sharing the edge $E$.
\end{lemma}
\begin{proof}
For any $\bm{\tau}_{hp}\in \Pp0(\T(E))^2$, we have
that $h_E^{-1}\|\bm{\tau}_{hp}\|_{0,E}^2 =
|T|^{-1}\|\bm{\tau}_{hp}\|_{0,T}^2$, $T=T_\pm$.
Hence,
\begin{align*}
\|\mathcal{L}_E^0(\jump{v_{hp}})\|_{0,T}
&\le \|\mathcal{L}_E^0(\jump{v_{hp}})\|_{0,\omega_E}\\
&= \sup_{\bm{\tau}_{hp}\in \Pp0(\T(E))^2, \,\| \bm{\tau}_{hp} \|_{0,\omega_E}=1} \int_{\omega_E} \mathcal{L}_E^0(\jump{v_{hp}})\cdot\overline{\bm{\tau}}_{hp}\, dx\\
&= \sup_{\bm{\tau}_{hp}\in \Pp0(\T(E))^2, \,\| \bm{\tau}_{hp} \|_{0,\omega_E}=1} \int_E\Pi^0_E(\jump{v_{hp}})\cdot\average{\overline{\bm{\tau}}_{hp}}\, ds\\
&\leq C h_E^{-1/2}\|\Pi^0_E(\jump{v_{hp}})\|_{0,E},
\end{align*}
where $C = h_E\max\{|T_+|^{-1/2},|T_-|^{-1/2}\}$ is bounded by 
shape regularity.
The second bound follows similarly.

\end{proof}
\begin{definition}[DG gradient]\label{def:gradient}
We define the DG gradient by
\begin{align}\label{eq:gradient}
 \G(u_{hp}) := \nabla_h u_{hp} - \sum_{E\in\E(\Omega)} \mathcal{L}_E^0(\jump{u_{hp}}) - \sum_{E\in\E(\Omega)} \mathcal{L}_E^1(\jump{\nabla u_{hp}}).
\end{align}
\end{definition}

\begin{remark}
The lifting operators $\mathcal{L}_E^0$ arise already in \cite{EV2015}
for the DG discretization of the Poisson problem;
whereas the lifting operators $\mathcal{L}_E^1$ are required due
to the additional gradient stabilization terms in the formulation \ref{discreteproblem}.
\end{remark}%

\section{A posteriori error estimator and reliability}\label{sec:aposteriori}
In this section, we derive an equilibrated a posteriori error
estimator based on a shifted Poisson problem
and prove that it is reliable, up to additional $L^2$ and boundary errors.
The definition of this estimator involves flux and potential reconstructions, which will be defined below.
This approach is also related to the a posteriori error analysis
for the eigenvalue problem via equilibrated fluxes developed in \cite{CDMSV2017,CDMSV2018}.
\par

For simplicity of the presentation of the equilibrated flux technique, we
restrict ourselves to conforming meshes with no hanging nodes. For the
necessary modifications
to handle irregular
meshes we refer the reader to \cite{DEV2016}.
\par

We approach the a posteriori error estimation of the DG finite element approximation of the Helmholtz problem
by considering the following (shifted) Poisson problem with Neumann boundary conditions:
Find a (complex) function $w\in H^2(\Omega)$ such that
\begin{equation}\label{shifted_problem}
\begin{aligned}
	-\Delta w &= f + k^2 u_{hp} &&\textrm{in }\Omega,\\
	\nabla w \cdot \n &= g + iku_{hp} - \gamma k \frac{\h}{\p}(g - \nabla_h u_{hp} \cdot \n + iku_{hp})&&\textrm{on }\partial\Omega.
\end{aligned}
\end{equation}
Note that the boundary condition is chosen in such a way that the compatibility condition for the
pure Neumann problem is satisfied due to \eqref{discreteproblem}.
\par

\begin{definition}[Flux reconstruction]\label{flux_recon}
For a given $u_{hp}\in V_{hp}$, we define an \emph{equilibrated flux reconstruction} for $u_{hp}$ as
any function $\bm{\sigma}_{hp}\in H(\ddiv;\Omega)$ which satisfies
\begin{equation}\label{eq:fluxreconstruction}
\begin{aligned}
	\int_T\ddiv\bm{\sigma}_{hp} \, dx &= \int_T f + k^2 u_{hp} \, dx && \forall T\in\T,\\
	\int_E \bm{\sigma}_{hp}\cdot\bm{n} \, ds &= \int_E -(g+iku_{hp}) + \gamma k \frac{\h}{\p}(g - \nabla_h u_{hp} \cdot\n - iku_{hp}) \, ds
	&& \forall E\in\E(\partial\Omega).
\end{aligned}
\end{equation}
\end{definition}

Since the compatibility condition of the pure Neumann problem is satisfied,
the existence of such a function follows from the mixed
theory applied to the homogeneous Neumann problem in $H_0(\ddiv;\Omega)\times L^2_0(\Omega)$, where
\begin{align*}
H_0(\ddiv;\Omega) &:= \{ \tau\in H(\ddiv;\Omega)\,: \, \tau\cdot \bm{n} = 0\;\mbox{on}\;\partial\Omega\},\\
L^2_0(\Omega) &:= \left\{ v\in L^2(\Omega)\, : \, \int_\Omega v\, dx = 0\right\}. 
\end{align*}
By proceeding as in the Dirichlet case \cite[Example 4.2.1]{BBF2013},
this relies on the surjectivity of $\ddiv:H_0(\ddiv;\Omega)\to L^2_0(\Omega)$
(see, e.g., \cite[Equation (4.2.62)]{BBF2013}).

We point out that $\bm{\sigma}_{hp}$ is not necessarily a piecewise polynomial
  function; the subscript $hp$ simply indicates that it is associated
  with a piecewise polynomial
  function (namely $u_{hp}$).
\par

\begin{definition}[Potential]\label{potential_recon}
We define a \emph{potential} as any function
\begin{align*}
	s_{hp}\in H^1_*(\Omega):= \{v\in H^1(\Omega)\,:\, (v,1)=0\}.
\end{align*}
\end{definition}
As for $\bm{\sigma}_{hp}$, the subscript $hp$ indicates that $s_{hp}$ will be constructed from $u_{hp}$; see Section~\ref{sec:potential} below. For this reason, we will call $s_{hp}$ a \emph{potential reconstruction} for $u_{hp}$.
\par

For the proof of reliability of the error estimator (see \eqref{eq:errorestimator} below), the following
Poincar\'e and trace estimates, with explicit constants for triangles, are required.
\begin{lemma}[Poincar\'e inequality on triangles \cite{LS2010}]
For any $v\in H^1(T)$, where $T$ is  
a triangle, it holds that
\begin{align}\label{Poincare}
\| v - \Pi^0_T v\|_{0,T} \leq \frac{h_T}{j_{1,1}}\|\nabla v\|_{0,T},
\end{align}
where $j_{1,1}\approx 3.83170597020751$ denotes the first positive root of the Bessel function of the first kind. \qed
\end{lemma}
\begin{lemma}\label{lem:tracebound}
For any $v\in H^1(T)$, where $T$ is  
a triangle, we have the following trace estimate
for any edge $E$ of T, 
\begin{align*}
h_E^{-1/2}\| v - \Pi^0_T v\|_{0,E} \leq C_{tr} \|\nabla v\|_{0,T},
\end{align*}
where $C_{tr}^2 = (j_{1,1}^{-1}+j_{1,1}^{-2})\frac{h_T^2}{|T|}\leq 0.3291 \frac{h_T^2}{|T|}$.
\end{lemma}
\begin{proof}
The trace identity of \cite[Lemma 2.1]{CGR2012} leads to the inequality
\begin{align*}
h_E^{-1}\| v - \Pi^0_Tv\|_{0,E}^2 \leq \frac{h_T}{|T|}\|v - \Pi^0_Tv\|_{0,T} \|\nabla v\|_{0,T} + \frac{1}{|T|}\|v - \Pi^0_Tv\|_{0,T}^2.
\end{align*}
This, together with 
the Poincar\'e inequality \eqref{Poincare}, yields
\begin{align*}
h_E^{-1}\| v - \Pi^0_Tv\|_{0,E}^2 \leq \frac{h_T^2}{j_{1,1}|T|}\|\nabla v\|_{0,T}^2 + \frac{h_T^2}{j_{1,1}^2|T|}\|\nabla v\|_{0,T}^2.
\end{align*}
\end{proof}

\begin{remark}
For the adaptive meshes used in
Section~\ref{sec:numerics}, 
which consist only of right-angled 
triangles, it holds that $h_T^2/|T| = 4$ and, therefore, $C_{tr} \leq 1.14733$.
\end{remark}

We can now define the following error estimator:
\begin{equation}\label{eq:errorestimator}
\begin{split}
\eta_{hp}^2 &:= 
\sum_{T\in\T} \left(
 \|\G(u_{hp})+\bm{\sigma}_{hp}\|_{0,T} 
+ \frac{h_T}{j_{1,1}}\|f + k^2 u_{hp} -\ddiv\bm{\sigma}_{hp}\|_{0,T} \right.\\
&\quad
+ \left.C_{tr}\!\!\!\!\!\sum_{E\in\E(T)\cap \E(\partial\Omega)}\!\!\!\!\! h_E^{1/2}\|\bm{\sigma}_{hp}\cdot\bm{n} + g + iku_{hp} - \gamma k \frac{\h}{\p}(g - \nabla_h u_{hp} \cdot \n + iku_{hp})\|_{0,E}\right)^2\\
&\quad
+ \sum_{T\in\T}\|\G(u_{hp})-\nabla s_{hp}\|_{0,T}^2,
\end{split}
\end{equation}%
where $\bm{\sigma}_{hp}\in H(\ddiv;\Omega)$ is an equilibrated flux reconstruction
of $u_{hp}$ as in Definition~\ref{flux_recon}, 
and $s_{hp}\in H^1_*(\Omega)$ is a potential as in Definition~\ref{potential_recon}.
In the following theorem, we prove reliability of the estimator defined in \eqref{eq:errorestimator},
up to additional $L^2$ and boundary errors.
Notice that we are still in the abstract setting, where $\bm{\sigma}_{hp}$ and $s_{hp}$
are any admissible flux and potential reconstructions, according to Definitions \ref{flux_recon} and \ref{potential_recon} respectively. A specific choice of $\bm{\sigma}_{hp}$ and $s_{hp}$,
for which efficiency can also be proven, will be given in Section~\ref{sec:efficiency} below.

\begin{theorem}[Reliability]\label{th:reliability}
Let $u\in H^1(\Omega)$ be the weak solution of the Helmholtz problem \eqref{weak:formulation}, 
and $u_{hp}\in V_{hp}$ be the discrete solution of \eqref{discreteproblem}.
Then, for the
error estimator defined in~\eqref{eq:errorestimator}, we have that
\begin{equation}\label{eq:reliability}
\begin{split}
\|\nabla u-\G(u_{hp})\|_{0,\Omega} &\lesssim \eta_{hp} + k^2\|u-u_{hp}\|_{0,\Omega} + k\|u-u_{hp}\|_{0,\partial\Omega} \\
&\quad + \|\gamma k \frac{\h}{\p}(g - \nabla_h u_{hp}\cdot\n + iku_{hp})\|_{0,\partial\Omega}.
\end{split}
\end{equation}
\end{theorem}
\begin{proof}
We follow the general ideas of the proof of \cite[Theorem 3.3]{EV2015}.
However, in order to connect with the shifted Poisson problem \eqref{shifted_problem},
some extra terms need to be bounded.
We repeat the full proof for completeness.
\par
Let $s\in H^1_*(\Omega)$ be defined by the projection
\begin{align}\label{eq:orth}
	(\nabla s,\nabla v) = (\G(u_{hp}),\nabla v)\quad\textrm{for all } v\in H^1(\Omega).
\end{align}
Then, by orthogonality, we have that
\begin{align}\label{eq:rel:1}
	\|\nabla u - \G(u_{hp})\|^2_{0,\Omega} = \|\nabla(u-s)\|^2_{0,\Omega} + \|\nabla s-\G(u_{hp})\|^2_{0,\Omega}.
\end{align}
\par
Since $s\in H^1_*(\Omega)$ is the orthogonal projection, we have
\[
\|\nabla s-\G(u_{hp})\|_{0,\Omega} = \min_{v\in H^1_*(\Omega)} \|\nabla v-\G(u_{hp})\|_{0,\Omega}.
\]
Hence, for any     $s_{hp}\in H^1_*(\Omega)$, we get the following
bound for the second term in \eqref{eq:rel:1}
\begin{align}\label{eq:rel:1bis}
\|\nabla s-\G(u_{hp})\|^2_{0,\Omega} \leq \|\nabla s_{hp} -\G(u_{hp})\|^2_{0,\Omega}.
\end{align}
\par
The first term of \eqref{eq:rel:1} is estimated by the flux reconstruction as follows.
We have
\begin{align*}
\|\nabla(u-s)\|_{0,\Omega}
&= \sup_{v\in H^1_*(\Omega),\, \|\nabla v\|_{0,\Omega}=1} (\nabla(u-s),\nabla v)\\
&= \sup_{v\in H^1_*(\Omega),\, \|\nabla v\|_{0,\Omega}=1} (\nabla u - \G(u_{hp}),\nabla v),
\end{align*}
where the second identity follows from \eqref{eq:orth}.
Adding and subtracting an equilibrated  
flux reconstruction $\bm{\sigma}_{hp}\in H(\ddiv;\Omega)$ leads to
\begin{align}\label{eq:rel:2}
(\nabla u - \G(u_{hp}),\nabla v)
&= (\nabla u + \bm{\sigma}_{hp},\nabla v) - (\G(u_{hp}) + \bm{\sigma}_{hp},\nabla v).
\end{align}
Using the weak formulation \eqref{weak:formulation} and 
integrating by parts in the first term on the right-hand side of \eqref{eq:rel:2} yields, for any $v\in H^1_*(\Omega)$ with $\|\nabla v\|_{0,\Omega}=1$,
\begin{align}\label{eq:rel:3}
\begin{split}
(\nabla u + \bm{\sigma}_{hp},\nabla v) 
&=(\nabla u,\nabla v) + (\bm{\sigma}_{hp},\nabla v)\\
&= (f+k^2 u - \ddiv\bm{\sigma}_{hp}, v) + ( g+iku + \bm{\sigma}_{hp}\cdot \bm{n}, v)_{\partial\Omega}\\
&= (f+k^2 u_{hp} - \ddiv\bm{\sigma}_{hp}, v) + ( g+iku_{hp} + \bm{\sigma}_{hp}\cdot \bm{n} , v)_{\partial\Omega}\\
&\quad+k^2(u-u_{hp},v) +ik(u-u_{hp},v)_{\partial\Omega}.
\end{split}
\end{align}
Here, the proof differs from that of \cite[Theorem 3.3]{EV2015}, in that we 
introduce the last two extra terms.
From Definition~\ref{flux_recon} of the equilibrated
flux reconstruction $\bm{\sigma}_{hp}$ we get for the first term on the right-hand side of \eqref{eq:rel:3}, for each element $T\in\T$ that
\begin{align*}
(f+k^2 u_{hp} - \ddiv\bm{\sigma}_{hp}, v)_T
&= (f+k^2 u_{hp} - \ddiv\bm{\sigma}_{hp}, v-\Pi^0_Tv)_T\\
&\leq \|f+k^2 u_{hp} - \ddiv\bm{\sigma}_{hp}\|_{0,T} \|v-\Pi^0_Tv\|_{0,T}\\
&\leq \frac{h_T}{j_{1,1}}\|f+k^2 u_{hp} - \ddiv\bm{\sigma}_{hp}\|_{0,T} \|\nabla v\|_{0,T},
\end{align*}
where in the last step we have used the bound~\eqref{Poincare}.
For the second term on the right-hand side of
\eqref{eq:rel:3}, we write
\begin{align*}
&( g+iku_{hp} + \bm{\sigma}_{hp}\cdot \bm{n}  , v)_{\partial\Omega}\\
&\quad= ( \bm{\sigma}_{hp}\cdot \bm{n} + g+iku_{hp}  - \gamma k \frac{\h}{\p}(g - \nabla_h u_{hp}\cdot\n - iku_{hp}), v)_{\partial\Omega}\\
&\qquad  +(\gamma k \frac{\h}{\p}(g - \nabla_h u_{hp}\cdot\n - iku_{hp}), v)_{\partial\Omega},
\end{align*}
where we need to introduce the last term in order to connect to the shifted Poisson problem \eqref{shifted_problem}.
Again, from the definition of $\bm{\sigma}_{hp}$, for any
boundary edge $E$ belonging to the triangle $T$, we have that
\begin{align*}
&(\bm{\sigma}_{hp}\cdot \bm{n} + g+iku_{hp}  - \gamma k \frac{\h}{\p}(g - \nabla_h u_{hp}\cdot\n - iku_{hp}),v)_E\\
&\qquad= (\bm{\sigma}_{hp}\cdot \bm{n} + g+iku_{hp}  - \gamma k \frac{\h}{\p}(g - \nabla_h u_{hp}\cdot\n - iku_{hp}),v-\Pi^0_Tv)_E\\
&\qquad\leq \|\bm{\sigma}_{hp}\cdot \bm{n} + g+iku_{hp}  - \gamma k \frac{\h}{\p}(g - \nabla_h u_{hp}\cdot\n - iku_{hp})\|_{0,E} \|v-\Pi^0_Tv\|_{0,E}\\
&\qquad\leq C_{tr}h_E^{1/2} \|\bm{\sigma}_{hp}\cdot \bm{n} + g+iku_{hp}  - \gamma k \frac{\h}{\p}(g - \nabla_h u_{hp}\cdot\n - iku_{hp})\|_{0,E}
\|\nabla v\|_{0,T},
\end{align*}
where in the last step we have used the bound from
Lemma~\ref{lem:tracebound}.
\par
From the Cauchy-Schwarz inequality, the above estimates, the Poincar\'e and trace estimates, 
and
\begin{align*}
(\G(u_{hp}) + \bm{\sigma}_{hp},\nabla v)_T \leq \|\G(u_{hp}) + \bm{\sigma}_{hp}\|_{0,T} \|\nabla v\|_{0,T},
\end{align*}
noting that $\|\nabla v\|_{0,\Omega} = 1$, we deduce the bound
\begin{align*}
(\nabla u - \G(u_{hp}),\nabla v)
&\lesssim \eta_{hp} + k^2\|u-u_{hp}\|_{0,\Omega}+ k\| u-u_{hp}\|_{0,\partial\Omega}\\
&\quad+ \|\gamma k \frac{\h}{\p}(g - \nabla_h u_{hp}\cdot\n - iku_{hp})\|_{0,\partial\Omega},
\end{align*}
for \eqref{eq:rel:2}. Then, inserting this bound and \eqref{eq:rel:1bis} into \eqref{eq:rel:1} completes the proof.
\end{proof}

\begin{remark}
Assuming that the resolution conditions established
in~\cite{SZ2015} are satisfied, and that an appropriate mesh refinement near the domain corners is applied,
the $L^2$ error terms appearing on the right-hand side of the reliability bound
\eqref{eq:reliability} in Theorem~\ref{th:reliability}
are actually higher-order terms, compared to the left-hand side.
\end{remark}

\section{Efficiency of the error estimator}\label{sec:efficiency}
The result in the previous section holds for any equilibrated flux and potential reconstructions;
cf., Definitions~\ref{flux_recon} \&~\ref{potential_recon}, respectively.
In this section, we {\em locally} define equilibrated flux and potential reconstructions,
for which we can show that the error estimator \eqref{eq:errorestimator} is efficient.
\par
In Section~\ref{sec:equilibrated_flux_reconstruction}, we start by constructing equilibrated fluxes on nodal patches by
solving local mixed problems with Raviart-Thomas finite elements.
We then show that the sum of these local fluxes satisfy the definition of an admissible
flux reconstruction (see Definition~\ref{flux_recon}).
We prove efficiency of this flux reconstruction in Theorem~\ref{thm:flux_efficiency} below.
The technique of this proof involves two steps:
an estimate for the strong residual and the $p$-robustness of the mixed approximation
(see Lemmas \ref{lemma:efficiency} and \ref{lemma:stability}, respectively).
\par
In Section~\ref{sec:potential}, we construct local potentials by solving minimization
problems again on nodal patches.
We reformulate these minimization problems as coercive variational problems,
which can only be done in two dimensions.
By combining these local potentials, we construct an admissible potential reconstruction,
according to Definition~\ref{potential_recon}, and prove its efficiency
(see Theorem~\ref{thm:potential_efficiency}).
\par
In Section~\ref{sec:efficiency_summary}, we show efficiency of the remaining data terms
and state the final efficiency result.

\subsection{Localized equilibrated flux reconstruction}\label{sec:equilibrated_flux_reconstruction}
We first define a computable equilibrated flux reconstruction $\bm{\sigma}_{hp}$, such that the terms in the
error estimator \eqref{eq:errorestimator} containing this reconstruction are efficient.

Using the partition of unity property of the linear hat-functions,
we can localize the construction of $\bm{\sigma}_{hp}$ on nodal patches $\omega_z$
by solving local patch problems in mixed formulation. For a given node $z\in\N$, with given integer $p_z\geq 1$, we define the space
\begin{align*}
	\Sigma_{hp}(\omega_z) := \{ \bm{\tau}_{hp} \in H(\ddiv,\omega_z)\, : \, \bm{\tau}_{hp}|_T \in RT_{p_z}(T) \text{ for all } T\in\T(z) \}
\end{align*}
of Raviart-Thomas finite elements $RT_{p_z}(T) := \left\{[\Pp{p_z}(T)]^2 + \widetilde{\mathbb{P}}_{p_z}(T)[x_1,x_2]^t\right\}$,
where $\widetilde{\mathbb{P}}_{p_z}(T)$ is the space of homogeneous polynomials
of degree $p_z$, and the space
\[
Q_{hp}(\omega_z) = \{q_{hp}\in L^2(\omega_z) : q_{hp}|_T\in \Pp{p_z}(T) \text{ for all } T\in\T(z) \}.
\]

Let $\psi_z\in H^1(\Omega)$ denote the piecewise linear hat function for the vertex $z\in \N$ with patch $\omega_z$.
Inserting $\psi_z$ as test functions into the discrete weak formulation 
\eqref{discreteproblem}, we get, via straightforward calculations, the following
hat function orthogonality,
\begin{align}\label{eq:orthogonality}
\begin{split}
&( \G(u_{hp}) , \nabla \psi_z)_{\omega_z}
-(f+k^2u_{hp}, \psi_z)_{\omega_z} \\
&\qquad= 
\left( g+iku_{hp},\psi_z\right)_{\partial\omega_z\cap\partial\Omega}
- \left( \gamma k\frac{\h}{\p}(g-\nabla_h u_{hp}\cdot\n+iku_{hp}),\psi_z\right)_{\partial\omega_z\cap\partial\Omega}\\
&\qquad\quad
-i\gamma\frac{\h}{\p}\left(g-\nabla_h u_{hp}\cdot\n+iku_{hp},\nabla_h\psi_z\cdot\n \right)_{\partial\omega_z\cap\partial\Omega}.
\end{split}
\end{align}
\par
Define for  $z\in\N$, and a given function $g^z\in L^2(\partial\omega_z\cap\partial\Omega)$, the local mixed finite element spaces
\begin{align*}
        \Sigma_{g^z,hp}^z &:= \{ \bm{\tau}_{hp}\in \Sigma_{hp}(\omega_z)\, : \, \bm{\tau}_{hp}\cdot\bm{n} = 0\text{ on }\partial\omega_z\backslash\partial\Omega,\\
         &\qquad\qquad\qquad\qquad\quad\;\,\bm{\tau}_{hp}\cdot\bm{n}|_{E} = \Pi^{p_z}_{E}g^z\text{ for all }E\subset\partial\omega_z\cap\partial\Omega\},\\
	Q_{hp}^z &:= \{ q_{hp}\in Q_{hp}(\omega_z)\,:\, (q_{hp},1)_{\omega_z} = 0\}.  
\end{align*}
\par
For each node $z\in \N$, we solve the following local problem in mixed form:
Find an approximation $(\bm{\zeta}_{hp}^z, r_{hp}^z)\in \Sigma_{g^z,hp}^z\times Q_{hp}^z$ such that
\begin{equation}\label{local_mixed}
\begin{aligned}
(\bm{\zeta}_{hp}^z, \bm{\tau}_{hp})_{\omega_z} - (r_{hp}^z, \ddiv \bm{\tau}_{hp})_{\omega_z} &= -(\psi_z \G(u_{hp}),\bm{\tau}_{hp})_{\omega_z} &&\textrm{for all }\bm{\tau}_{hp}\in \Sigma_{0,hp}^z,\\
(\ddiv \bm{\zeta}_{hp}^z, q_{hp})_{\omega_z}  &= (f^z, q_{hp})_{\omega_z}&&\textrm{for all }q_{hp}\in Q_{hp}^z,
\end{aligned}
\end{equation}
where the function $f^z$ is given by
\begin{align}\label{eq:nodal_f}
 f^z := (f+k^2u_{hp}) \psi_z - \G(u_{hp}) \cdot \nabla \psi_z, 
\end{align}
and the function $g^z\in L^2(\partial\omega_z\cap\partial\Omega)$ in the definition of $\Sigma_{g^z, h}^z$ is given by
\begin{equation}\label{eq:nodal_g}
\begin{aligned}
g^z &:= -\left(g+iku_{hp} - \gamma k\frac{\h}{\p}(g-\nabla_h u_{hp}\cdot\n+iku_{hp})\right)\psi_z \\
&\qquad+i\gamma\frac{\h}{\p}\left(g-\nabla_h u_{hp}\cdot\n+iku_{hp}\right)\left(\nabla_h\psi_z\cdot\n\right).
\end{aligned}
\end{equation}
Actually, $f^z$ and $g^z$ are defined such that, from the hat function orthogonality \eqref{eq:orthogonality}, we get
\begin{align}\label{eqn:comp_cond}
(f^z, 1)_{\omega_z} =(g^z,1)_{\partial\omega_z\cap\partial\Omega},
\end{align}
which is the pure Neumann problem compatibility condition.
\begin{remark}
From integration by parts, the boundary condition on $\partial\omega_z$, and the compatibility condition \eqref{eqn:comp_cond}, we note that
\[
(\ddiv \bm\zeta_{hp}^z,1)_{\omega_z} = \int_{\partial\omega_z} \bm\zeta_{hp}^z\cdot\n\, dx
= \int_{\partial\omega_z\cap\partial\Omega} g^z\, dx = (f^z,1)_{\omega_z}.
\]
Hence, together with \eqref{local_mixed} we have that
\begin{equation}\label{eq:divzeta_f_equal}
(\ddiv \bm{\zeta}_{hp}^z, q_{hp})_{\omega_z}  = (f^z, q_{hp})_{\omega_z}\qquad\textrm{for all }q_{hp}\in Q_{hp}(\omega_z).
\end{equation}
\end{remark}

We can now define the equilibrated flux reconstruction $\bm{\sigma}_{hp}$ as
\begin{equation}\label{eq:equilibrated_flux}
\bm{\sigma}_{hp} := \sum_{z\in\N} \bm{\zeta}_{hp}^z,
\end{equation}
and prove that it satisfies Definition~\ref{flux_recon}.
\begin{lemma}\label{lemma:fluxreconstruction}
The flux approximation $\bm{\sigma}_{hp}$, defined in \eqref{eq:equilibrated_flux}, is an equilibrated flux reconstruction in $H(\ddiv;\Omega)$ which satisfies, for any $T\in \T$,
\begin{align*}
(f+k^2u_{hp} - \ddiv\bm{\sigma}_{hp}, q_{hp})_T &= 0
\end{align*}
for all $q_{hp}\in \bigcap_{z\in\N(T)} Q_{hp}(\omega_z)|_T$, and
for any $E\in \E(\partial\Omega)$,
\begin{align*}
(\bm{\sigma}_{hp}\cdot\bm{n} +g+iku_{hp}-\gamma k\frac{\h}{\p}(g-\nabla_h u_{hp}\cdot\n+iku_{hp}), q_{hp})_E &= 0
\end{align*}
for all $q_{hp}\in \bigcap_{z\in\N(E)} Q_{hp}(\omega_z)|_E$.
\end{lemma}
\begin{proof}
For all $z\in\N$, by extension of $\bm\zeta_{hp}^z$ by zero in $\Omega\setminus\omega_z$, we have that $\bm\zeta_{hp}^z\in H(\ddiv;\Omega)$; therefore, $\bm{\sigma}_{hp}\in H(\ddiv;\Omega)$ also holds.
For any $T\in\T$, by using the partition of unity property of $\psi_z$, the definition of
$\bm{\sigma}_{hp}$, and \eqref{eq:divzeta_f_equal}, it holds that
\begin{align*}
(f+k^2u_{hp} - \ddiv\bm{\sigma}_{hp}, q_{hp})_T &= \sum_{z\in\N(T)} (\psi_z(f+k^2u_{hp} ) - \ddiv \bm{\zeta}_{hp}^z,q_{hp})_T\\
& =  \sum_{z\in\N(T)}( \G(u_{hp}) \cdot \nabla \psi_z,q_{hp})_T\\
&=0,
\end{align*}
for all $q_{hp}\in Q_{hp}$, 
where in the last step we used the fact that $\sum_{z\in\mathcal{N}(T)}\nabla\psi_z=0$.
Using the partition of unity property of $\psi_z$ along the boundary edges,
the definition of $\bm{\sigma}_{hp}$, and the fact that $\bm\zeta_{hp}^z\in\Sigma_{g^z,hp}^z$,
we get for any $E\in \E(\partial\Omega)$, with associated element $T_E\in\T$, that
\begin{align*}
(\bm{\sigma}_{hp}\cdot\bm{n}, q_{hp})_E
&= 
\sum_{z\in\N(T_E)}(\bm{\zeta}_{hp}^z\cdot\bm{n},q_{hp})_E
=
\sum_{z\in\N(T_E)}(g^z,q_{hp})_E\\
&=
\left( -(g+iku_{hp})+\gamma k\frac{\h}{\p}(g-\nabla_h u_{hp}\cdot\n+iku_{hp}),q_{hp}\right)_E,
\end{align*}
for all $q_{hp}\in Q_{hp}$, where we use the fact that $\sum_{z\in\mathcal{N}(T_E)}\psi_z=1$ and
the fact that $\sum_{z\in\mathcal{N}(T_E)}\nabla\psi_z\cdot\n=0$ on $E$.
\end{proof}
\par
We proceed by showing that the flux reconstruction \eqref{eq:equilibrated_flux} is efficient.
In order to do that, we start by defining
the following data oscillation terms.
\begin{definition}[Data oscillations]
We define
\begin{align*}
\osc^2(f^z) &=  \sum_{T\in\T(z)} \frac{h_T^2}{j_{1,1}^2}\|f^z - \Pi^{p_z}_T f^z\|_{0,T}^2,\\
\osc^2(g^z) &=  \sum_{E\in\E(z)\cap\E(\partial\Omega)} 
C_{tr}^2h_E\| g^z - \Pi^{p_z}_E g^z\|_{0,E}^2,
\end{align*}
for all $z\in\N$, and
\begin{align*}
\osc^2(f) &= \sum_{z\in \N} \osc^2(f^z), \qquad \osc^2(g) = \sum_{z\in \N(\partial\Omega)}\osc^2(g^z).
\end{align*}
\end{definition}
\par

Now, we derive an estimate of
the strong residual.
The following lemma is based on the results in \cite{CF1999}.
\begin{lemma}[Continuous efficiency, flux reconstruction]\label{lemma:efficiency}
Let $w$ be the weak solution of the (shifted) Poisson problem \eqref{shifted_problem},
with $u_{hp}\in V_{hp}$ being the $hp$-DG approximation given by \eqref{discreteproblem}.
Furthermore, let $z\in\mathcal{N}$ and $r^z\in H^1_*(\omega_z):=\{ v\in H^1(\omega_z)\, : \, (v,1)_{\omega_z}=0\}$
be the solution to
the continuous problem
\begin{align}\label{eq:residual}
\begin{split}
	(\nabla r^z,\nabla v)_{\omega_z} &= -(\psi_z\G(u_{hp}),\nabla v)_{\omega_z} + \sum_{T\in\T(z)}(\Pi^{p_z}_T f^z,v)_T\\
	&\quad-\sum_{E\in\E(z)\cap\E(\partial\Omega)}(\Pi^{p_z}_{E}  g^z,v)_E
\end{split}
\end{align}
for all $v\in H^1(\omega_z)$, with the right hand side $f^z$  and the boundary function $g^z$
given in \eqref{eq:nodal_f} and \eqref{eq:nodal_g}, respectively.
Then, it holds that
\begin{align*}
 \|\nabla r^z\|_{0,\omega_z} &\lesssim \|\nabla w - \G(u_{hp})\|_{0,\omega_z}
 + \| i\gamma\frac{\sqrt{\h}}{\p}\left(g-\nabla_h u_{hp}\cdot\n+iku_{hp}\right)\|_{0,\partial\omega_z\cap\partial\Omega}\\
 &\quad +  \osc(f^z) + \osc(g^z).
\end{align*}
\end{lemma}
\begin{proof}
Since the right hand side $f^z$  and the boundary function $g^z$ 
are constructed such that the compatibility condition \eqref{eqn:comp_cond} is satisfied on $\omega_z$ it is, therefore,
also satisfied for their $L^2$-projections.
This, together with the Lax-Milgram lemma, implies that
\eqref{eq:residual} is well posed.
We have that
\begin{align}\label{eq:efficiency:proof}
 \|\nabla r^z\|_{0,\omega_z} = \sup_{v\in H^1_*(\omega_z), \|\nabla v\|_{0,\omega_z}=1} (\nabla r^z,\nabla v)_{\omega_z};
\end{align}
moreover, for $v\in H^1_*(\omega_z)$, $\|\nabla v\|_{0,\omega_z}=1$, we can write
\begin{align*}
(\nabla r^z,\nabla v)_{\omega_z} 
&= -(\psi_z\G(u_{hp}),\nabla v)_{\omega_z}  + \sum_{T\in\T(z)}(\Pi^{p_z}_T f^z,v)_T
    - \sum_{E\in\E(z)\cap\E(\partial\Omega)}(\Pi^{p_z}_E g^z,v)_E\\
&= -(\psi_z\G(u_{hp}),\nabla v)_{\omega_z}  + (f^z,v)_{\omega_z} -(g^z,v)_{\partial\omega_z\cap\partial\Omega}\\
&\quad+ \sum_{T\in\T(z)}(\Pi^{p_z}_T f^z - f^z ,v -\Pi_T^0 v)_T
    - \!\!\!\!\!\!\sum_{E\in\E(z)\cap\E(\partial\Omega)}(\Pi^{p_z}_E  g^z - g^z,v -\Pi_T^0 v)_E.
\end{align*}
The last two terms on the right-hand side are bounded by $\osc(f^z)$ and $\osc(g^z)$, respectively, by applying the Cauchy-Schwarz inequality, Lemmas~\ref{lem:tracebound} \&~\ref{Poincare}, and the fact that $\|\nabla v\|_{0,\omega_z}=1$. For the first three terms on the right-hand side, by application of integration by parts, the Cauchy-Schwarz inequality, and the definitions of $f^z$, $g^z$, and $w$, we obtain
\begin{align*}
-(&\psi_z\G(u_{hp}),\nabla v)_{\omega_z}  + (f^z,v)_{\omega_z} - (g^z,v)_{\partial\omega_z\cap\partial\Omega}\\
&\quad= -(\psi_z\G(u_{hp}),\nabla v)_{\omega_z}  +
( (f+k^2u_{hp})\psi_z,v)_{\omega_z} -( \G(u_{hp}) \cdot\nabla \psi_z,v)_{\omega_z}\\
&\qquad - (g^z,v)_{\partial\omega_z\cap\partial\Omega}\\
&\quad= (\nabla w-\G(u_{hp}),\nabla_h (\psi_zv))_{\omega_z}\\
&\qquad - (i\gamma\frac{\h}{\p}\left(g-\nabla_h u_{hp}\cdot\n+iku_{hp}\right)(\nabla_h\psi_z\cdot\n),v)_{\partial\omega_z\cap\partial\Omega}\\
&\quad\leq \|\nabla w - \G(u_{hp})\|_{0,\omega_z}\|\nabla_h(\psi_z v)\|_{0,\omega_z}\\
&\qquad+ \| i\gamma\frac{\sqrt{\h}}{\p}\left(g-\nabla_h u_{hp}\cdot\n+iku_{hp}\right)\|_{0,\partial\omega_z\cap\partial\Omega}
\|\sqrt{\h}(\nabla_h\psi_z\cdot\n) v \|_{0,\partial\omega_z\cap\partial\Omega}. 
\end{align*}
By the triangle inequality, the scaling of the hat-functions, shape regularity, and the Poincar\'e inequality, we have that
\[
\|\nabla_h(\psi_z v)\|_{0,\omega_z}\leq \|v\nabla_h \psi_z\|_{0,\omega_z} + \|\psi_z\nabla v\|_{0,\omega_z}
\lesssim \|\h v\|_{0,\omega_z} + \|\nabla v\|_{0,\omega_z} \lesssim \|\nabla v\|_{0,\omega_z} = 1;
\]
a similar bound holds for $\|\sqrt{\h}(\nabla_h\psi_z\cdot\n) v \|_{0,\partial\omega_z\cap\partial\Omega}$ after application of shape regularity and the trace estimates. Therefore,
we conclude from the previous estimates that
\begin{align*}
(\nabla r^z,\nabla v)_{\omega_z}  &\lesssim \|\nabla w - \G(u_{hp})\|_{0,\omega_z} \\
&\qquad+ \| i\gamma\frac{\sqrt{\h}}{\p}\left(g-\nabla_h u_{hp}\cdot\n+iku_{hp}\right)\|_{0,\partial\omega_z\cap\partial\Omega} 
+ \osc(f^z) + \osc(g^z),
\end{align*}
for all $v\in H^1_*(\omega_z)$, such that $\|\nabla v\|_{0,\omega_z}=1$. Inserting this result into \eqref{eq:efficiency:proof} completes the proof.
\end{proof}
\par
In the following lemma, we essentially report \cite[Theorem 7]{BPS2009}, which is a key result in the proof of $p$-robustness.
\begin{lemma}\label{lemma:stability}
Let $u_{hp}\in V_{hp}$ be the hp-DG approximation given by \eqref{discreteproblem}; furthermore, for $z\in\N$,
let $\bm{\zeta}_{hp}^z\in \Sigma_{g^z, hp}^z$ be the solution to the local nodal mixed problem
\eqref{local_mixed}, $\psi_z\in H^1(\Omega)$ be the nodal hat function associated with the node $z$, and $r^z\in H^1_*(\omega_z)$ be
defined as in Lemma~\ref{lemma:efficiency}. Then,
the stability result
\begin{align*}
\|\psi_z\G(u_{hp})+\bm{\zeta}_{hp}^z\|_{0,\omega_z} \leq C \|\nabla r^z\|_{0,\omega_z}
\end{align*}
holds, with a constant $C>0$ that is independent of the polynomial degree,
mesh size, and wave number, but depends on the shape regularity of the mesh.
\end{lemma}
\begin{proof}
As in \cite[Corollary 3.16]{EV2015}, the proof is essentially \cite[Theorem 7]{BPS2009}. 
Note that,
\[
\| \nabla r^z \|_{0,\omega_z}
= \sup_{v\in H^1_*(\omega_z),\, \|\nabla v\|_{0,\omega_z}=1} (\nabla r^z, \nabla v)_{\omega_z}.
\]
In fact, from \eqref{eq:residual} we have that
\begin{align*}
&(\nabla r^z, \nabla v)_{\omega_z} 
= -(\psi_z\G(u_{hp}),\nabla v)_{\omega_z} 
+\sum_{T\in\T(z)}(\Pi^{p_z}_T f^z,v)_T - \sum_{E\in\E(z)\cap\E(\partial\Omega)}(\Pi^{p_z}_E g^z,v)_E\\
&\quad=  
\sum_{T\in\T(z)} \int_T(\ddiv(\psi_z\G(u_{hp}))+\Pi^{p_z}_T f^z)v\,dx
+\sum_{E\in \E(z)\cap\E(\Omega)}\int_E \jump{-\psi_z\G(u_{hp})}v\, ds \\
&\quad\qquad-\sum_{E\in\E(z)\cap\E(\partial\Omega)}\int_E (\Pi^{p_z}_E g^z + \psi_z\G(u_{hp}))v\,ds
\end{align*}
for all $v\in H^1_*(\omega_z)$ such that $\|\nabla v\|_{0,\omega_z}=1$.

Defining $r_T:= \ddiv(\psi_z\G(u_{hp}))+\sum_{T\in\T(z)}\Pi^{p^z}_T f^z$, $r_E:= \jump{-\psi_z\G(u_{hp})}$ for interior edges,
and $r_E:=-\sum_{E\in\E(z)\cap\E(\partial\Omega)}(\Pi^{p_z}_E g^z + \psi_z\G(u_{hp}))$ for edges on the boundary,
we have that $\| \nabla r^z \|_{0,\omega_z}$ in our notation is $\| r \|_{[H^1(\omega)\backslash\mathbb{R}]}$ in the notation
of \cite[Lemma 7]{BPS2009}.
Moreover,
\begin{align*}
\| \psi_z\G(u_{hp}) + \bm{\zeta}_{hp}^z\|_{0,\omega_z} 
= \inf_{\substack{\bm{\tau}_{hp}\in\Sigma_{g^z,hp}^z\\\ddiv(\bm{\tau}_{hp})|_T=\Pi^{p_z}_T f^z \;\forall T\in\T(z)}}\| \psi_z\G(u_{hp}) + \bm{\tau}_{hp}\|_{0,\omega_z},
\end{align*}
which, in the notation of \cite{BPS2009}, reads as
\[
\inf_{\sigma\in \text{RT}^p_{-1,0},\, \ddiv\sigma=r} \|\sigma\|_0,\]
formulated in the broken Raviart-Thomas finite element space with imposed jumps
$\jump{-\psi_z\G(u_{hp})}$.
\end{proof}
\par
Finally, by using Lemmas \ref{lemma:efficiency} and \ref{lemma:stability},
we prove a $p$-robust efficiency bound of the first term in the error estimator \eqref{eq:errorestimator}.
\begin{theorem}[Flux reconstruction efficiency]\label{thm:flux_efficiency}
Let $u\in H^1(\Omega)$ be the weak solution of the Helmholtz problem \eqref{weak:formulation}, 
$u_{hp}\in V_{hp}$ be the discrete solution of \eqref{discreteproblem}, and
$\bm{\sigma}_{hp}\in H(\ddiv;\Omega)$ be the equilibrated flux reconstruction of $u_{hp}$ defined
in \eqref{eq:equilibrated_flux}; then,
\begin{align*}
 \|\G(u_{hp})+\bm{\sigma}_{hp}\|_{0,\Omega} &\lesssim \|\nabla u - \G(u_{hp})\|_{0,\Omega} +  k^2\|u-u_{hp}\|_{0,\Omega}
 + k\|u-u_{hp}\|_{0,\partial\Omega}\\
 &\quad +\osc(f) +\osc(g)
 + \|\gamma k \frac{\h}{\p}(g - \nabla_h u_{hp}\cdot\n + iku_{hp})\|_{0,\partial\Omega} \\
 &\quad
 + \| i\gamma\frac{\sqrt{\h}}{\p}\left(g-\nabla_h u_{hp}\cdot\n+iku_{hp}\right)\|_{0,\partial\Omega}.
\end{align*}
\end{theorem}

\begin{proof}
The uniform stability of the local mixed problems from Lemma~\ref{lemma:stability}, and the partition of unity property
prove that
\begin{align*}
 \|\G(u_{hp})+\bm{\sigma}_{hp}\|_{0,\Omega}
 \leq  \sum_{z\in\N}\|\psi_z\G(u_{hp})+\bm{\zeta}_{hp}^z\|_{0,\omega_z}
 \leq C \sum_{z\in\N}\|\nabla r^z\|_{0,\omega_z}.
\end{align*}
Applying Lemma~\ref{lemma:efficiency}, noting the finite overlap of the patches $\omega_z$, bounds this term
by $\osc(f)$, $\osc(g)$, $\|\nabla w - \G(u_{hp})\|_{0,\Omega}$, and the boundary terms appearing in the right-hand side of
the required bound; therefore, all that remains is to bound $\|\nabla w - \G(u_{hp})\|_{0,\Omega}$. By the triangle inequality,
we have
\begin{align}
  \|\nabla w - \G(u_{hp})\|_{0,\Omega} &\leq \|\nabla u - \G(u_{hp})\|_{0,\Omega} + \|\nabla (w - u)\|_{0,\Omega} \nonumber\\
  &= \|\nabla u - \G(u_{hp})\|_{0,\Omega} + \sup_{v\in H^1_*(\Omega),\, \|\nabla v\|_{0,\Omega}=1} (\nabla(w-u),\nabla v).
  \label{eq:flux_efficiency:proof}
\end{align}
Applying integration by parts, the definition of $w$ from \eqref{shifted_problem}, \eqref{problem}, and Cauchy-Schwarz, we get that
\begin{align*}
(\nabla(w-u),\nabla v)
&= -(\Delta (w-u), v) + (\nabla (w-u)\cdot \n, v)_{\partial\Omega} \\
&= (f+k^2u_{hp} - (f + k^2 u),v) \\
&\quad+ (g+iku_{hp} -\gamma k\frac{\h}{\p}(g-\nabla_h u_{hp}\cdot\n+iku_{hp})-(g+iku),v)_{\partial\Omega}\\
&\leq k^2\|u-u_{hp}\|_{0,\Omega}\|v\|_{0,\Omega} + k\|u-u_{hp}\|_{0,\partial\Omega}\|v\|_{0,\partial\Omega}\\
&\quad+ \|\gamma k \frac{\h}{\p}(g - \nabla_h u_{hp}\cdot\n + iku_{hp})\|_{0,\partial\Omega}\|v\|_{0,\partial\Omega},
\end{align*}
for all $v\in H^1_*(\Omega)$ such that $\|\nabla v\|_{0,\Omega}=1$. From the Poincar\'e inequality we get that
$\|v\|_{0,\Omega}\leq C\|\nabla v\|_{0,\Omega}=C$, where the constant $C$ depends only on the domain $\Omega$, and similarly
by applying a trace estimate $\|v\|_{0,\partial\Omega}\leq C'\|\nabla v\|_{0,\Omega}=C'$; therefore, inserting this result into
\eqref{eq:flux_efficiency:proof} completes the proof.
\end{proof}

\subsection{Localized potential reconstruction}\label{sec:potential}
In this section, we define the potential reconstruction such that the error
estimator \eqref{eq:errorestimator} is efficient.

In order to define a localized polynomial space on
patches we need to distinguish between boundary and interior nodes.
For a given boundary node $z\in\N(\partial\Omega)$, with associated integer $p_z\geq 1$ as defined in Section~\ref{sec:equilibrated_flux_reconstruction}, we define the localized polynomial space
\[
    V_{hp}^z:=\{ v_{hp}\in C^0(\overline{\omega_z})\,:\, v_{hp}|_T\in\Pp{p_z+1}(T)\quad\forall T\in\T(z), v_{hp} = 0 \text{ on } \partial\omega_z\setminus\partial\Omega \};
\]
for an internal node $z\in\N\setminus\N(\partial\Omega)$, with integer $p_z$, we define the localized polynomial space as
\[
    V_{hp}^z:=\{ v_{hp}\in C^0(\overline{\omega_z})\,:\, v_{hp}|_T\in\Pp{p_z+1}(T)\quad\forall T\in\T(z), v_{hp} = 0 \text{ on } \partial\omega_z \}.
\]
We then choose $\widetilde{s}_{hp}\in H^1(\Omega)$ as
\[
\widetilde{s}_{hp} := \sum_{z\in\N} s_{hp}^z,
\]
where
\begin{align}\label{eq:potentialminimization:nodal}
   s_{hp}^z := \argmin_{v_{hp}\in V_{hp}^z}\|\nabla_h(\psi_z u_{hp}) - \nabla v_{hp}\|_{0,\omega_z}
\end{align}
with extension by zero in $\Omega\setminus\omega_z$,
which is equivalent to finding $s^z_{hp}$ such that
\begin{align*}
 (\nabla s_{hp}^z,\nabla v_{hp})_{\omega_z} = (\nabla_h(\psi_zu_{hp}),\nabla v_{hp})_{\omega_z} \quad\text{for all } v_{hp}\in V_{hp}^z.
\end{align*}
Then, the potential reconstruction $s_{hp}\in H^1_*(\Omega)$ is defined as
\begin{equation}\label{eq:potentialminimization}
s_{hp} := \widetilde{s}_{hp} - \frac1{|\Omega|} \int_{\Omega} \widetilde{s}_{hp}\, dx,
\end{equation}
which clearly satisfies Definition~\ref{potential_recon}.
\par
It has been noted in \cite[Remark 3.10]{EV2015} that the local minimization in \eqref{eq:potentialminimization:nodal}
in primal form is equivalent to the following minimization in mixed form
\begin{align*}
	\bm{\zeta}_{hp}^z := \argmin_{\bm{\tau}_{hp}\in \Sigma_{0,hp}^z, \, \ddiv(\bm{\tau}_{hp})=0}\| \rot_h(\psi_z u_{hp})+\bm{\tau}_{hp}\|_{\omega_z},
\end{align*}
which is equivalent to solving the following (local) mixed problem:
Find $(\bm\zeta_{hp}^z,r_{hp}^z)\in\Sigma_{0,hp}^z\times Q_{hp}^z$ such that
\begin{align*}
(\bm{\zeta}_{hp}^z, \bm{\tau}_{hp})_{\omega_z} - (r_{hp}^z, \ddiv \bm{\tau}_{hp})_{\omega_z} &= -(\rot_h(\psi_z u_{hp}),\bm{\tau}_{hp})_{\omega_z} &&\textrm{for all }\bm{\tau}_{hp}\in \Sigma_{0,hp}^z,\\
(\ddiv \bm{\zeta}_{hp}^z, q_{hp})_{\omega_z}  &= 0 &&\textrm{for all }q_{hp}\in Q_{hp}^z.
\end{align*}
For the underlying continuous problem we have the primal formulation: Find $r^z\in H^1_*(\omega_z)$ such that
\begin{align*}
(\nabla r^z,\nabla v)_{\omega_z} = -(\rot_h(\psi_z u_{hp}),\nabla v)_{\omega_z}  \quad\text{for all } v\in H^1(\omega_z).
\end{align*}
Proceeding as in \cite[Section 4.3.2]{EV2015} leads to the analogue of Lemma~\ref{lemma:efficiency},
\begin{align*}
\|\nabla r^z\|_{0,\omega_z}^2 
\lesssim \|\nabla_h(u-u_{hp})\|_{0,\omega_z}^2 + \sum_{E\in\E(z)\cap\E(\Omega)} h_E^{-1}\|\Pi_E^0\jump{u-u_{hp}}\|_{0,E}^2.
\end{align*}
Following the lines of proof of the local efficiency in \cite[Theorem 3.17]{EV2015}, yields
\begin{align*}
\|\nabla_h(u_{hp}-s_{hp})\|_{0,T}
\lesssim \sum_{z\in \N(T)}\|\rot_h(\psi_z u_{hp}) +\bm{\zeta}_{hp}^z \|_{0,\omega_z} 
\lesssim \sum_{z\in \N(T)}\|\nabla r^z\|_{0,\omega_z}.
\end{align*}
Therefore,
\begin{align*}
\|\nabla_h(u_{hp}-s_{hp})\|_{0,T}^2 &\lesssim \sum_{z\in \N(T)} \|\nabla_h(u-u_{hp})\|_{0,\omega_z}^2\\
&\quad +\sum_{z\in \N(T)} \sum_{E\in\E(z)\cap\E(\Omega)} h_E^{-1}\|\Pi_E^0\jump{u-u_{hp}}\|_{0,E}^2.
\end{align*}
Hence, due to the stability of the lifting operators in Lemma~\ref{lemma:stabilitylifting}
we can derive the following efficiency estimate for the final term in the error estimator \eqref{eq:errorestimator}.
\begin{theorem}[Potential reconstruction efficiency]\label{thm:potential_efficiency}
Let $u\in H^1(\Omega)$ be the weak solution of the Helmholtz problem \eqref{weak:formulation}, 
$u_{hp}\in V_{hp}$ be the discrete solution of \eqref{discreteproblem}, and
$s_{hp}\in H^1_*(\Omega)$ be the potential reconstruction defined as in \eqref{eq:potentialminimization}; then,
\begin{align*}
\|\G(u_{hp}) - \nabla s_{hp}\|_{0,\Omega}^2 &\lesssim \|\nabla_h (u - u_{hp})\|_{0,\Omega}^2 + \sum_{E\in\E(\Omega)} h_E^{-1}\|\Pi_E^0\jump{u_{hp}}\|_{0,E}^2 \\
&\qquad+ \sum_{E\in\E(\Omega)} \beta^2 h_E\|\p^{-1}\Pi_E^0\jump{\nabla u_{hp}}\|_{0,E}^2.
\end{align*}
\end{theorem}

\subsection{Efficiency result}\label{sec:efficiency_summary}
We can now combine Theorems~\ref{thm:flux_efficiency} and~\ref{thm:potential_efficiency} to show that the complete error
indicator \eqref{eq:errorestimator} is efficient.
\begin{theorem}[Error estimator efficiency]
Let $u\in H^1(\Omega)$ be the weak solution of the Helmholtz problem \eqref{weak:formulation}, 
$u_{hp}\in V_{hp}$ be the discrete solution of \eqref{discreteproblem}, $\eta_{hp}$ be the error estimator
\eqref{eq:errorestimator}, 
$\bm{\sigma}_{hp}\in H(\ddiv;\Omega)$ be the equilibrated flux reconstruction of $u_{hp}$ defined
in \eqref{eq:equilibrated_flux}, and
$s_{hp}\in H^1_*(\Omega)$ be the potential reconstruction defined as in \eqref{eq:potentialminimization}; then,
\begin{align*}
\eta_{hp} &\lesssim  \|\nabla_h (u - u_{hp})\|_{0,\Omega} +  k^2\|u-u_{hp}\|_{0,\Omega}
 + k\|u-u_{hp}\|_{0,\partial\Omega}+\osc(f) +\osc(g)\\
 &\quad+ \|\gamma k \frac{\h}{\p}(g - \nabla_h u_{hp}\cdot\n + iku_{hp})\|_{0,\partial\Omega} 
 + \| i\gamma\frac{\sqrt{\h}}{\p}\left(g-\nabla_h u_{hp}\cdot\n+iku_{hp}\right)\|_{0,\partial\Omega} \\
  &\quad+\left(\sum_{E\in\E(\Omega)} h_E^{-1}\|\Pi_E^0\jump{u_{hp}}\|_{0,E}^2\right)^{1/2}
\!\!+ \left(\sum_{E\in\E(\Omega)} \beta^2 h_E\|\p^{-1}\Pi_E^0\jump{\nabla u_{hp}}\|_{0,E}^2\right)^{1/2}.
\end{align*}
\end{theorem}
\begin{proof}
The efficiency of the first and last terms of the error estimator \eqref{eq:errorestimator} are given by Theorems~\ref{thm:flux_efficiency}
and~\ref{thm:potential_efficiency}, respectively, noting that
\begin{align*}
\|\nabla u-\G(u_{hp})\|_{0,\Omega} &\leq \|\nabla_h (u - u_{hp})\|_{0,\Omega} + \left(\sum_{E\in\E(\Omega)} h_E^{-1}\|\Pi_E^0\jump{u_{hp}}\|_{0,E}^2\right)^{1/2} \\
&\quad+ \left(\sum_{E\in\E(\Omega)} \beta^2 h_E\|\p^{-1}\Pi_E^0\jump{\nabla u_{hp}}\|_{0,E}^2\right)^{1/2},
\end{align*}
due to the triangle inequality and Lemma~\ref{lemma:stabilitylifting}. Therefore, to complete the proof, we need to derive
efficiency estimates for the two remaining terms of the error
estimator \eqref{eq:errorestimator} that contain $f$ and $g$.
From the partition of unity property of the hat functions $\psi_z$, 
the definition~\eqref{eq:nodal_f},
and the property~\eqref{eq:divzeta_f_equal}, we deduce
\begin{align*}
\|f + k^2 u_{hp} -\ddiv\bm{\sigma}_{hp}\|_{0,T} 
&= \|\sum_{z\in\N(T)} (f^z-\ddiv\bm{\xi}^z_{hp})\|_{0,T}
\leq \sum_{z\in\N(T)}\|  f^z-\Pi_T^{p_z}f^z\|_{0,T}.
\end{align*}
Hence, summing over all $T\in\T$ and rearranging the summations, we arrive at
\begin{align*}
\sum_{T\in\T} \frac{h_T^2}{j_{1,1}^2}\|f + k^2 u_{hp} -\ddiv\bm{\sigma}_{hp}\|^2_{0,T} 
\leq \osc^2(f).
\end{align*}
Similarly, for any boundary edge $E\in\E(\partial\Omega)$ with associated element $T_E\in\T$, we get that
\begin{align*}
&\|\bm{\sigma}_{hp}\cdot\bm{n} + g + iku_{hp} - \gamma k \frac{\h}{\p}(g - \nabla_h u_{hp} \cdot \n + iku_{hp})\|_{0,E}\\
&\quad= \|\sum_{z\in\N(T_E)} (\bm{\xi}^z_{hp}\cdot\bm{n}-g^z)\|_{0,E}
\leq \sum_{z\in\N(T_E)}\| g^z-\Pi_E^{p_z}g^z\|_{0,E},
\end{align*}
which proves
\begin{align*}
\sum_{E\in\E(\partial\Omega)} C_{tr}^2h_E\|\bm{\sigma}_{hp}\cdot\bm{n} + g + iku_{hp} - \gamma k \frac{\h}{\p}(g - \nabla_h u_{hp} \cdot \n + iku_{hp})\|^2_{0,E}
\leq \osc^2(g).
\end{align*}
Combining these results completes the proof.
\end{proof}

\begin{remark}\label{remark_choice_pz}
In the numerical experiments we will be interested in the efficiency
index $\eta_{hp}/\|\nabla u - \G(u_{hp})\|_{0,\Omega}$ to be close to
one. While we have proven asymptotic reliability of $\eta_{hp}$ 
for any $p_z\geq1$, the efficiency of the flux reconstruction in Theorem~\ref{thm:flux_efficiency} also depends on the efficiency of the data terms.
Hence, for $\osc(f)$, and $\osc(g)$ to be comparably small, one has to at least match the polynomial degree $p_z$ of the Raviart-Thomas finite element space to that of $u_{hp}\psi_z$, i.e.,  $p_z = \max_{T\in\T(z)} p_T+1$, for the flux reconstruction.
On the contrary, the efficiency of the potential reconstruction in Theorem~\ref{thm:potential_efficiency} does not depend on the
data oscillations; therefore, the choice $p_z = \max_{T\in\T(z)} p_T$ is sufficient in order to match the (local) polynomial degree of $u_{hp}\psi_z$ in the local minimization problems \eqref{eq:potentialminimization:nodal}.
\end{remark}

\section{Numerical results}\label{sec:numerics}

In this section we present numerical results for four different benchmark
problems.
\par
For efficiency of the error estimator, we approximate the local mixed problems with Raviart-Thomas
finite elements of (varying) order $p_z = \max_{T\in\T(z)} p_T+1$ for the flux reconstruction
and $p_z = \max_{T\in\T(z)} p_T$ for the potential reconstruction, cf. Remark~\ref{remark_choice_pz}. 
To reduce the ill-conditioning of the basis for the local problems with high $p_z$,
we use the hierarchical basis functions for the Raviart-Thomas finite element space presented in \cite{BPZ2012}.
In comparison to a standard residual a posteriori error estimator, the computation of the flux and potential reconstructions 
is more costly; however, since all local problems are independent, they can be solved in parallel.
In addition, the number of local problems is independent of the polynomial degree.
\par
We compare $p$- and adaptive $h$-refinement to an adaptive $hp$-refinement strategy.
For the $hp$-refinement we use the decision mechanism
of Melenk \& Wohlmuth \cite[Algorithm 4.4]{MW2001} outlined in Algorithm~\ref{algo:refinement}, which determines for $h$- or $p$-refinement on the refinement level $\ell$
based on verifying the decay of the local error indicators.
We choose the constants $\gamma_h=4$, $\gamma_p=0.4$, $\gamma_n=1$, and the initial values $\eta_{T,0}^{pred}=\infty$, for all $T\in\T$.
Hence, the algorithm prefers $p$- over $h$-refinement in the first step.
For the mesh refinement, we use the newest vertex bisection algorithm, and 
mark elements based on the maximum marking strategy with parameter $0.75$.
\begin{algorithm}[t]
\begin{algorithmic}[1] 
\caption{$hp$-refinement algorithm.}\label{algo:refinement}
    	\If{$T$ is marked for refinement}
        	\If{$\eta_{T,\ell} > \eta_{T,\ell}^{\mathrm{pred}}$}
            	\State Perform $h$--refinement: Subdivide $T$ into $2$ children $T_\pm$, and set
            	\State $(\eta_{T_\pm,\ell+1}^{\mathrm{pred}})^2\gets \frac{1}{2}\gamma_h \left(\frac{1}{2}\right)^{p_T} \eta_{T,\ell}^2$.  	      	\Else
    	        \State Perform $p$--refinement: $p_T \gets p_T+1$
        	    \State $(\eta_{T,\ell+1}^{\mathrm{pred}})^2\gets \gamma_p \eta_{T,\ell}^2$        
        	\EndIf
    	\Else
        	\State $(\eta_{T,\ell+1}^{\mathrm{pred}})^2\gets \gamma_n (\eta_{T,\ell}^{\mathrm{pred}})^2$
    	\EndIf
\end{algorithmic}
\end{algorithm}
\par

In order to shorten the pre-asymptotic region, we choose the initial mesh size $h$ and (uniform) polynomial degree $p$ as
\begin{align}\label{eq:choice}
  p = \lceil\ln(k)\rceil
  \qquad\text{and}\qquad
  \frac{kh}{p} \leq C_{\mathrm{res}},
\end{align}
where the resolution constant $C_{\mathrm{res}}$ depends on the problem under consideration; cf. \cite[Section 5]{SZ2015}. 
Additional numerical experiments, in which the initial conditions \eqref{eq:choice} are violated, are also
presented in the following in order to demonstrate that the method under consideration is able to escape the pre-asymptotic region
regardless of the initial mesh.

For the DG formulation we choose the parameters
$\alpha = 10$, $\beta = 1$, and $\gamma = 1/4$.

\subsection{Square domain}\label{example:square}
\begin{figure}[tbp]
\centering
\includegraphics[width=0.49\textwidth]{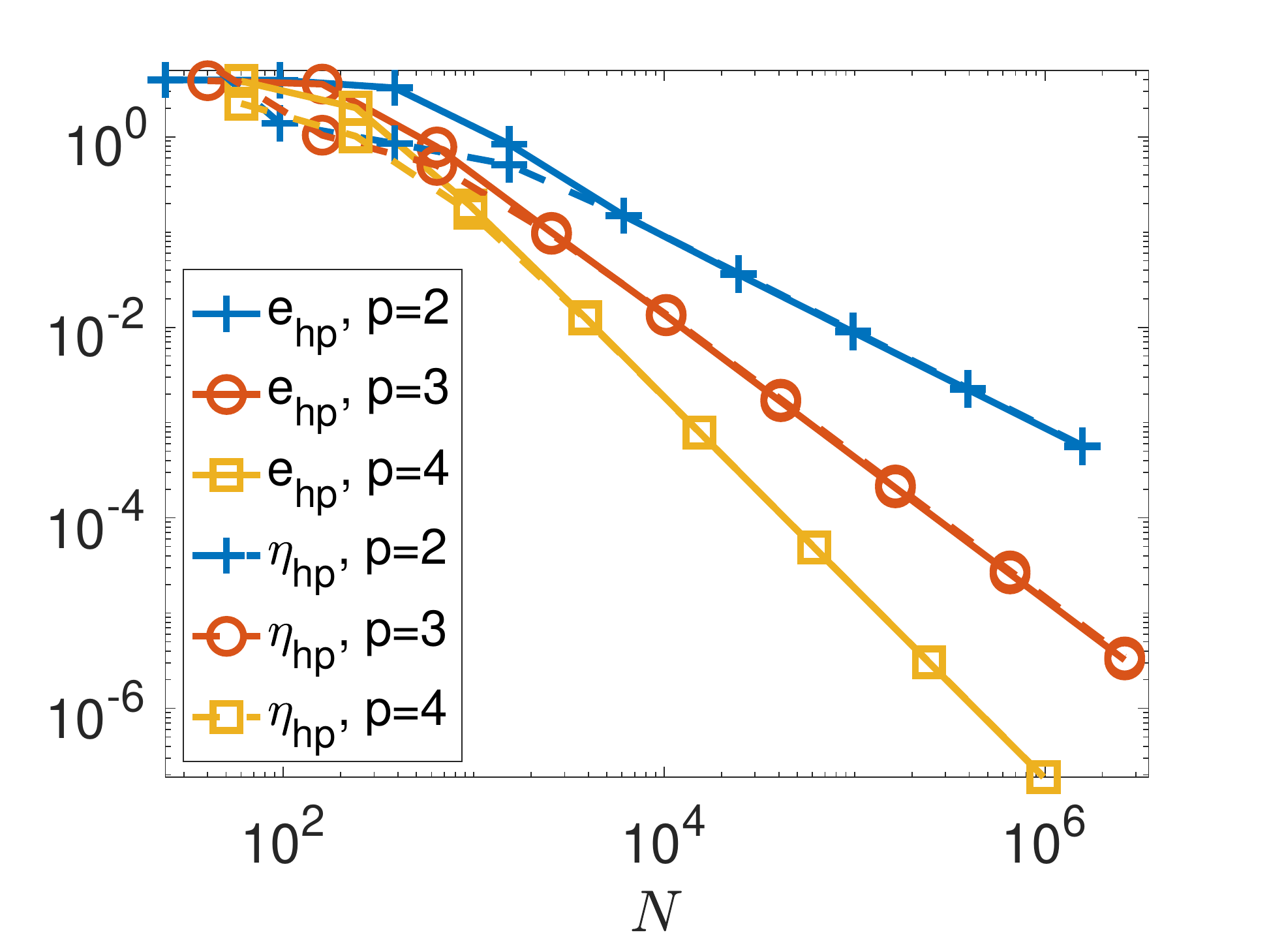}
\includegraphics[width=0.49\textwidth]{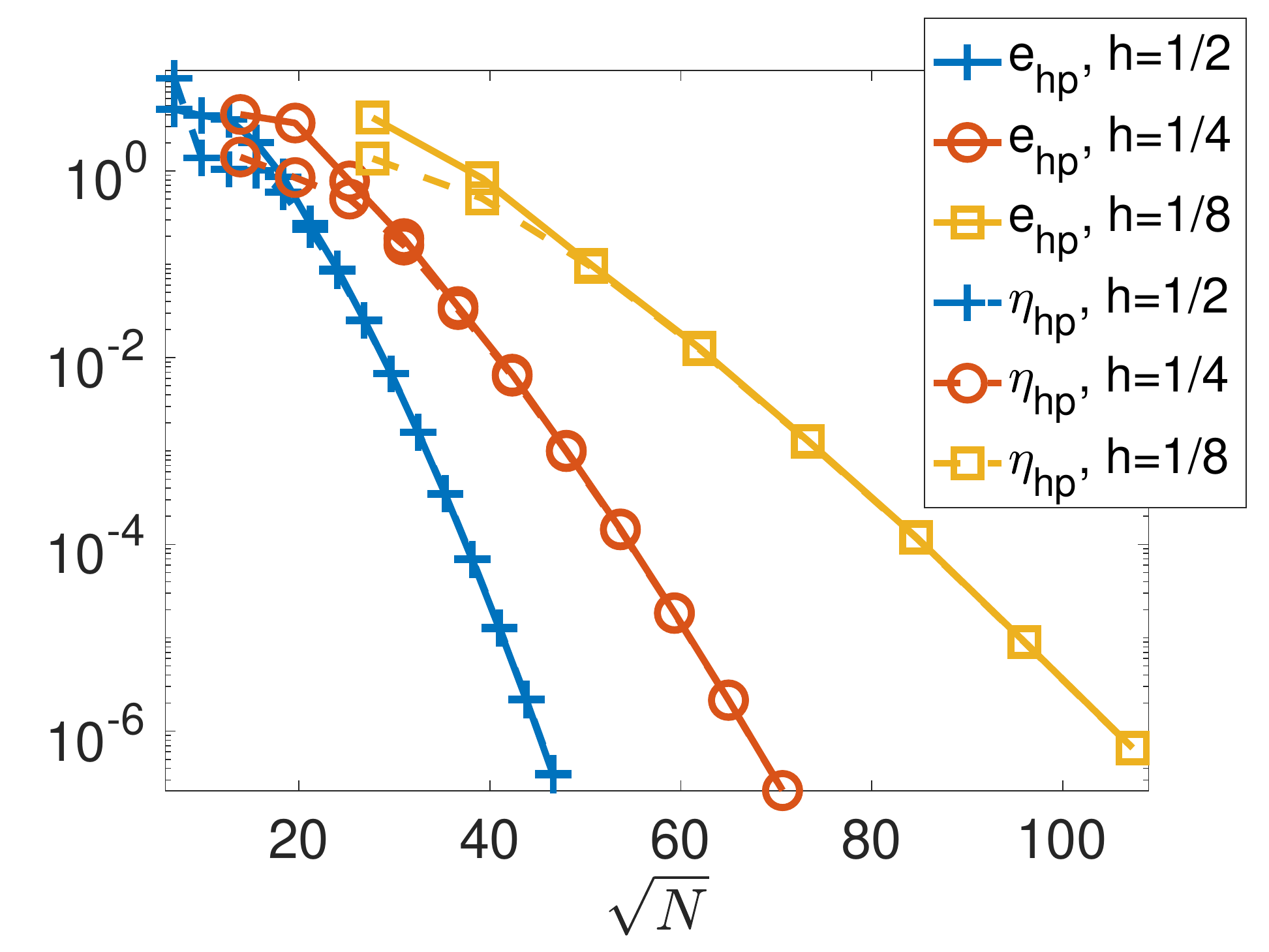}
\caption{Algebraic convergence  $\mathcal{O}(h^p)$ for the $h$-version (left) and exponential convergence of the $p$-version (right) of the error $e_{hp}=\|\nabla u - \G(u_{hp})\|_{0,\Omega}$ and the estimator $\eta_{hp}$ for the example in Section~\ref{example:square} with $k=20$.}
\label{fig:Hankel_uniform_H_P}
\end{figure}
\begin{figure}[tbp]
\centering
\includegraphics[width=0.49\textwidth]{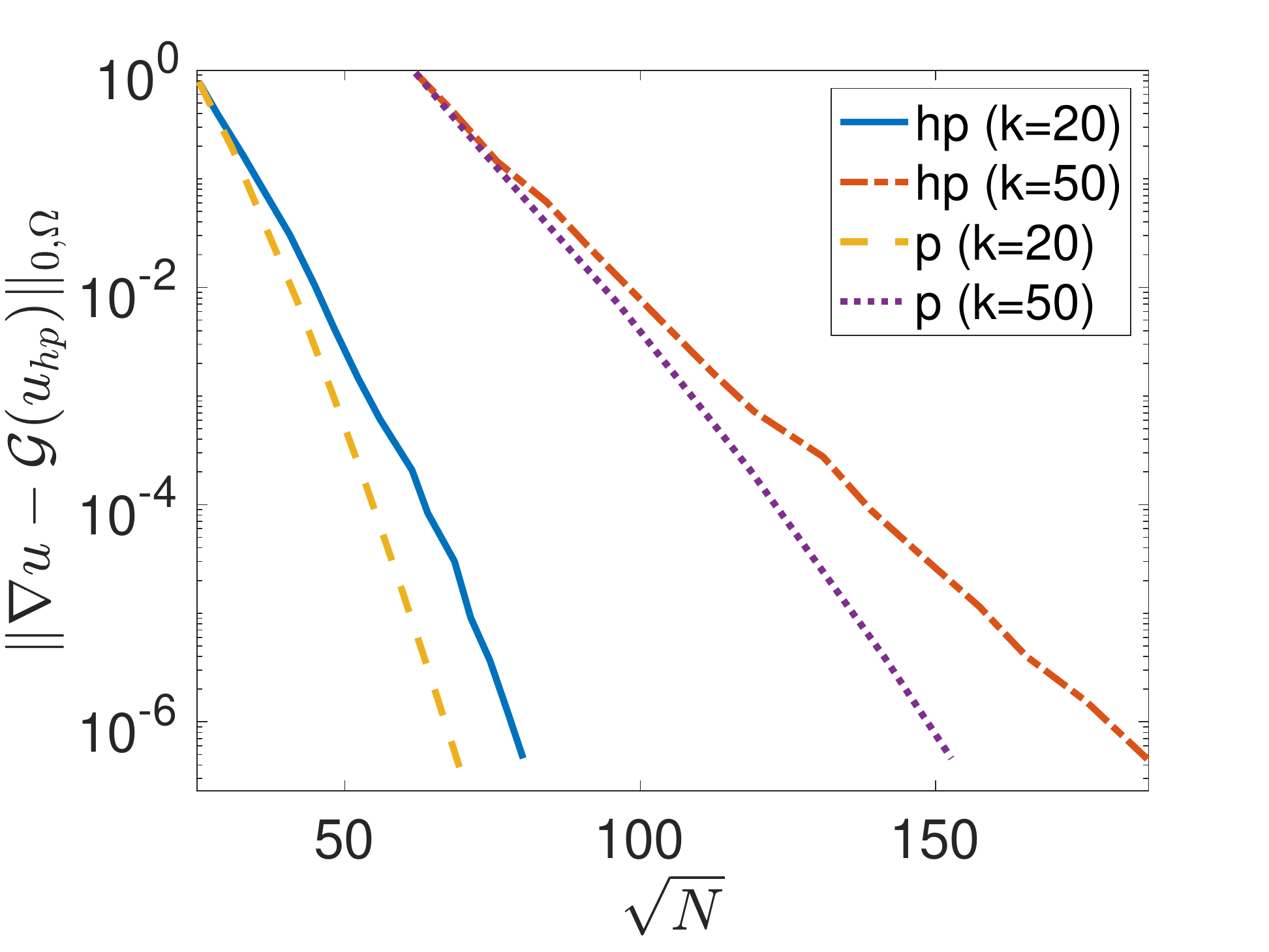}
\includegraphics[width=0.49\textwidth]{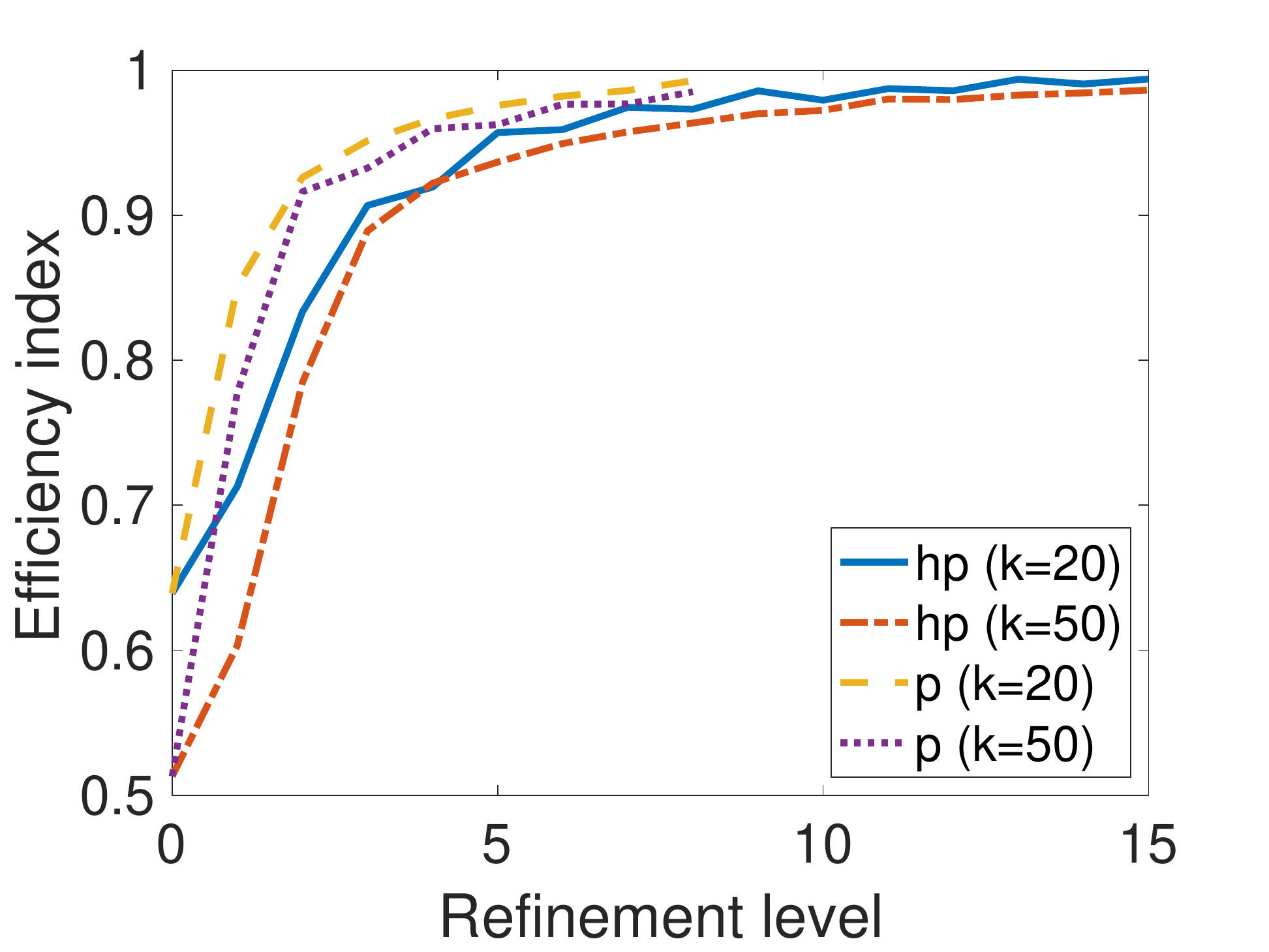}
\caption{Exponential convergence (left) and efficiency indices (right) of the $p$- and $hp$-version for the example in Section~\ref{example:square}.}
\label{fig:Hankel}
\end{figure}
\begin{figure}[tbp]
\centering
\includegraphics[width=0.49\textwidth]{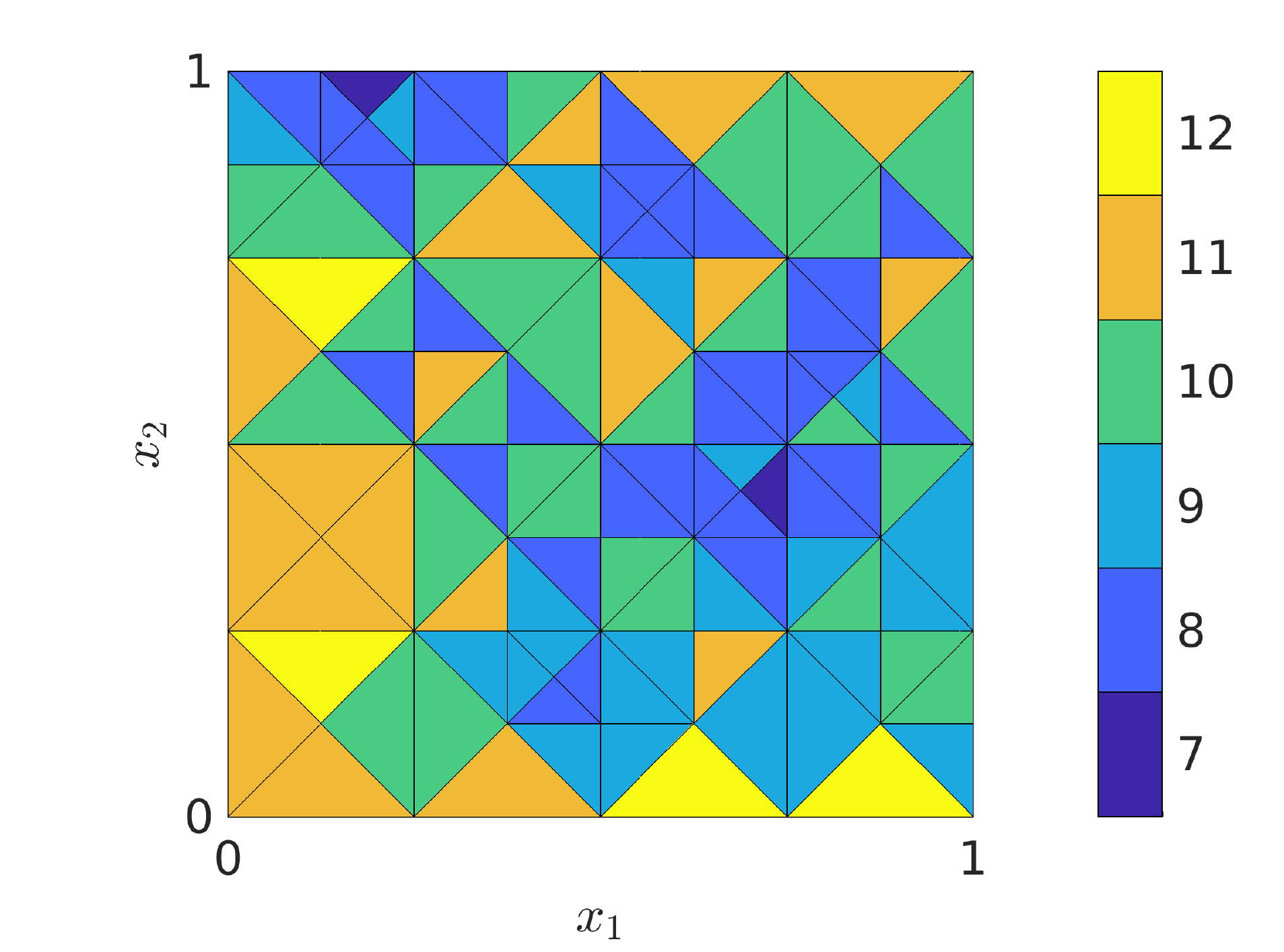}
\includegraphics[width=0.49\textwidth]{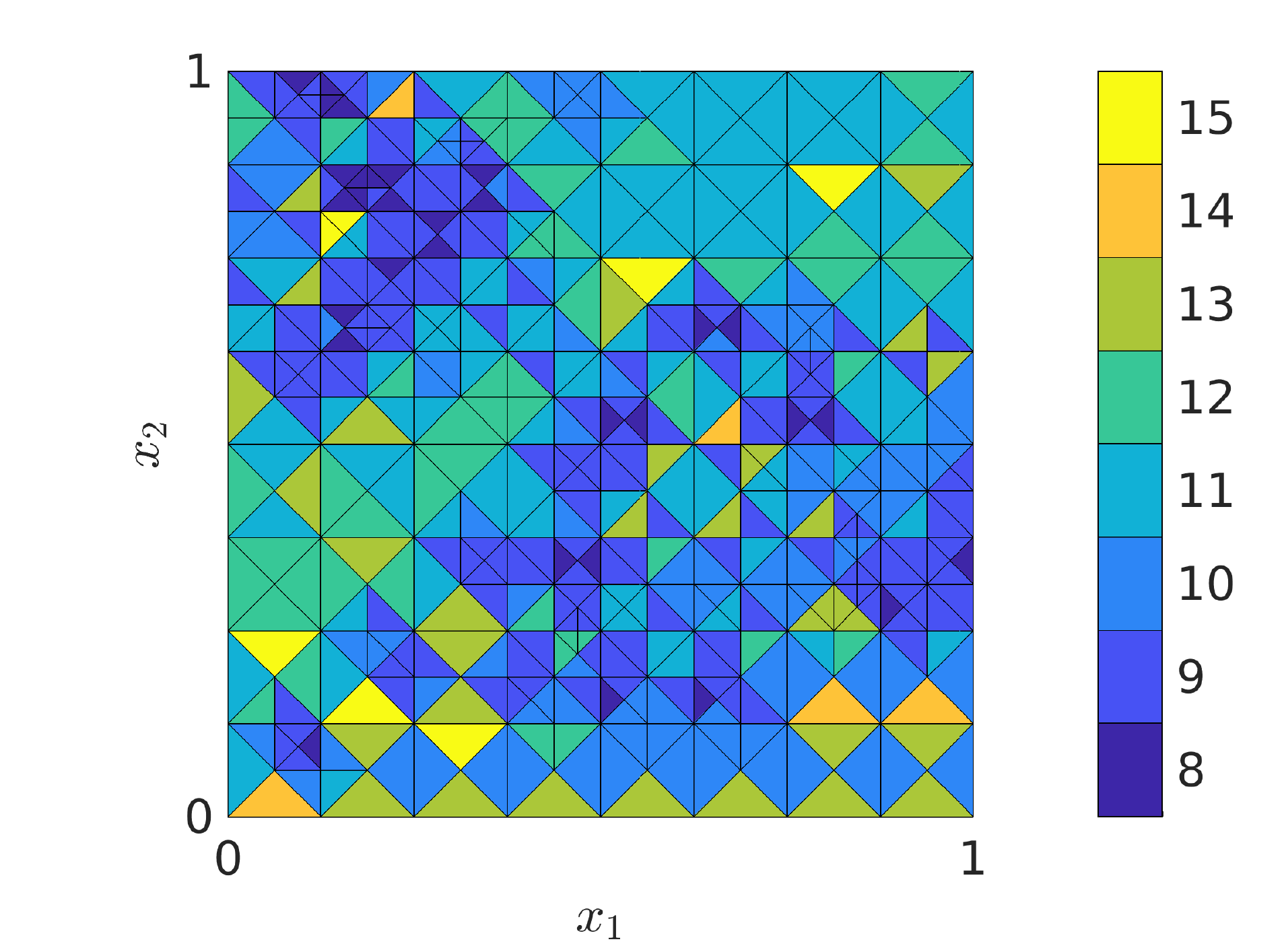}
\caption{$hp$-refined mesh for the example in Section~\ref{example:square} with $k=20$ (left) and $k=50$ (right),
where the polynomial degree is indicated with different shading.}
\label{fig:mesh:Hankel}
\end{figure}
\begin{figure}[tbp]
\centering
\includegraphics[width=0.49\textwidth]{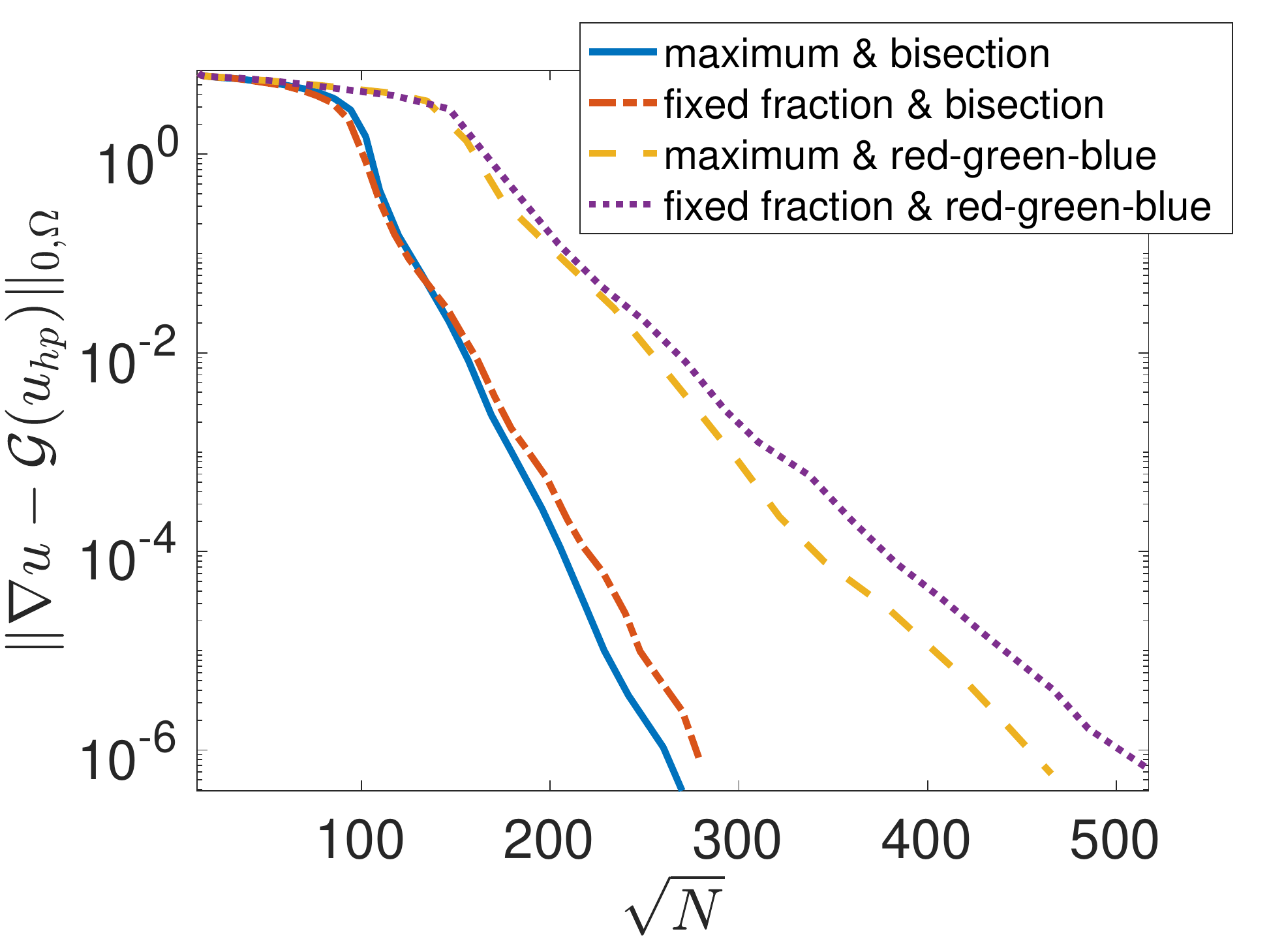}
\includegraphics[width=0.49\textwidth]{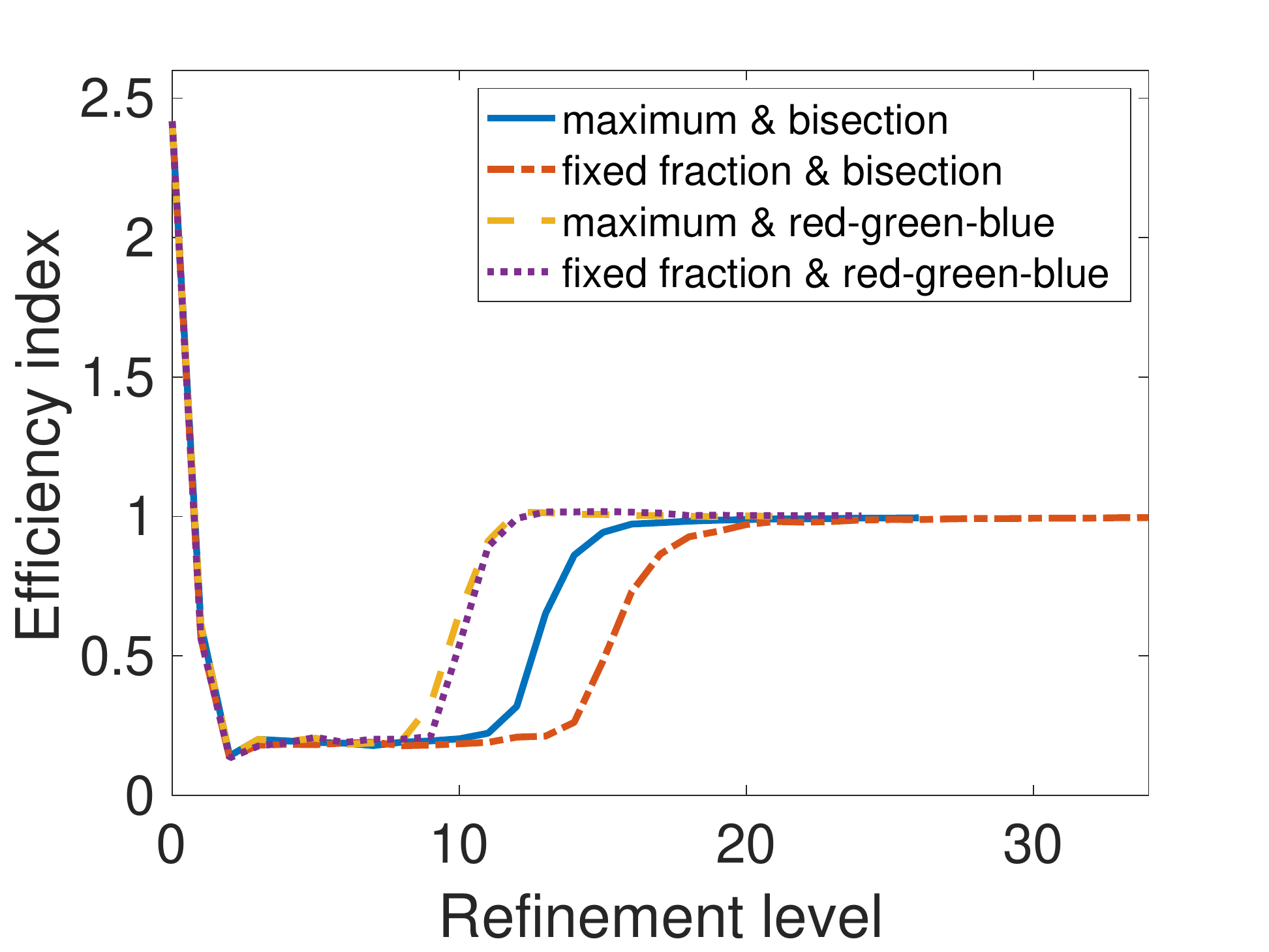}
\caption{Exponential convergence (left) and efficiency indices (right) for different $hp$-refinement 
strategies for the example in Section~\ref{example:square} with $k=50$.}
\label{fig:HankelStrategies}
\end{figure}
Let $\Omega=(0,1)^2$, $f=0$, and select $g$ such that the solution of \eqref{weak:formulation} is given by
\begin{align*}
u(x) = \mathcal{H}^{(1)}_0\left( k\sqrt{(x_1+1/4)^2+x_2^2}\right),
\end{align*}
where $\mathcal{H}^{(1)}_0$ denotes the zeroth order Hankel function of the first kind.
In Figure~\ref{fig:Hankel_uniform_H_P}, we observe convergence of $\mathcal{O}(h^p)$, $p=2,3,4$, for uniform $h$-refinement, and exponential convergence for uniform $p$-refinement.
Note that the error estimator is very close to the error and the corresponding lines are in fact overlapping.
In Figure~\ref{fig:Hankel} we observe exponential convergence of the error $\|\nabla u - \G(u_{hp})\|_{0,\Omega}$ for both $p$- and adaptive $hp$-refinement for wavenumbers $k=20,50$. For $hp$- and $p$-refinement we observe efficiency indices $\eta_{hp}/\|\nabla u - \G(u_{hp})\|_{0,\Omega}$ asymptotically close to $1$. 
For the construction of the initial mesh, we choose $C_{\mathrm{res}}=2$.
Note that for coarse $h$ and $p$ the pollution error is still dominant and, hence, $\eta_{hp}$
underestimates the error $\|\nabla u - \G(u_{hp})\|_{0,\Omega}$; therefore, the efficiency indices are initially
less than one.
The final $hp$-refined mesh for $k=20,50$ are displayed in Figure~\ref{fig:mesh:Hankel}.
\par
Next, we compare four variants of the $hp$-refinement strategy for $k=50$.
First, instead of bisecting a triangle into just two new triangles we divide it into four using the red-green-blue-refinement strategy \cite[Section 4.1]{RV}.
The effect of this is a more aggressive $h$-refinement. Note that our implementation of mesh refinement
is based on conforming refinements, which means that there also occurs additional refinement 
to remove hanging nodes. This overhead is significantly larger for red-green-blue-refinement than for refinement based on bisection.
The second variation is to use a fixed fraction marking strategy instead of the maximum marking strategy, where 25\% of the elements with the
largest indicators are refined. This leads to a more aggressive refinement between two consecutive
levels. 
In Figure~\ref{fig:HankelStrategies}, we observe that less $h$-refinement is more effective; hence, the errors corresponding to refinement by bisection
lead to less degrees of freedom than those corresponding to red-green-blue-refinement  for the same level of accuracy.
The two different marking strategies lead to comparable errors in this smooth example.
We draw the conclusion that even though in principle we are using the same $hp$-marking strategy, 
the actual performance of the method depends significantly on the concrete implementation.
Note that, for this experiment, the choice of the initial values $p=1$ and $h=1/4$ violates both conditions in \eqref{eq:choice}, and, neglecting the first mesh, 
the error is initially underestimated.
Nevertheless, all strategies are capable of eventually refining enough so that the pollution error becomes sufficiently small, 
such that the efficiency indices are asymptotically close to 1 for all four strategies.

\subsection{L-shaped domain}\label{example:lshape}
\begin{figure}[tbp]
\centering
\includegraphics[width=0.49\textwidth]{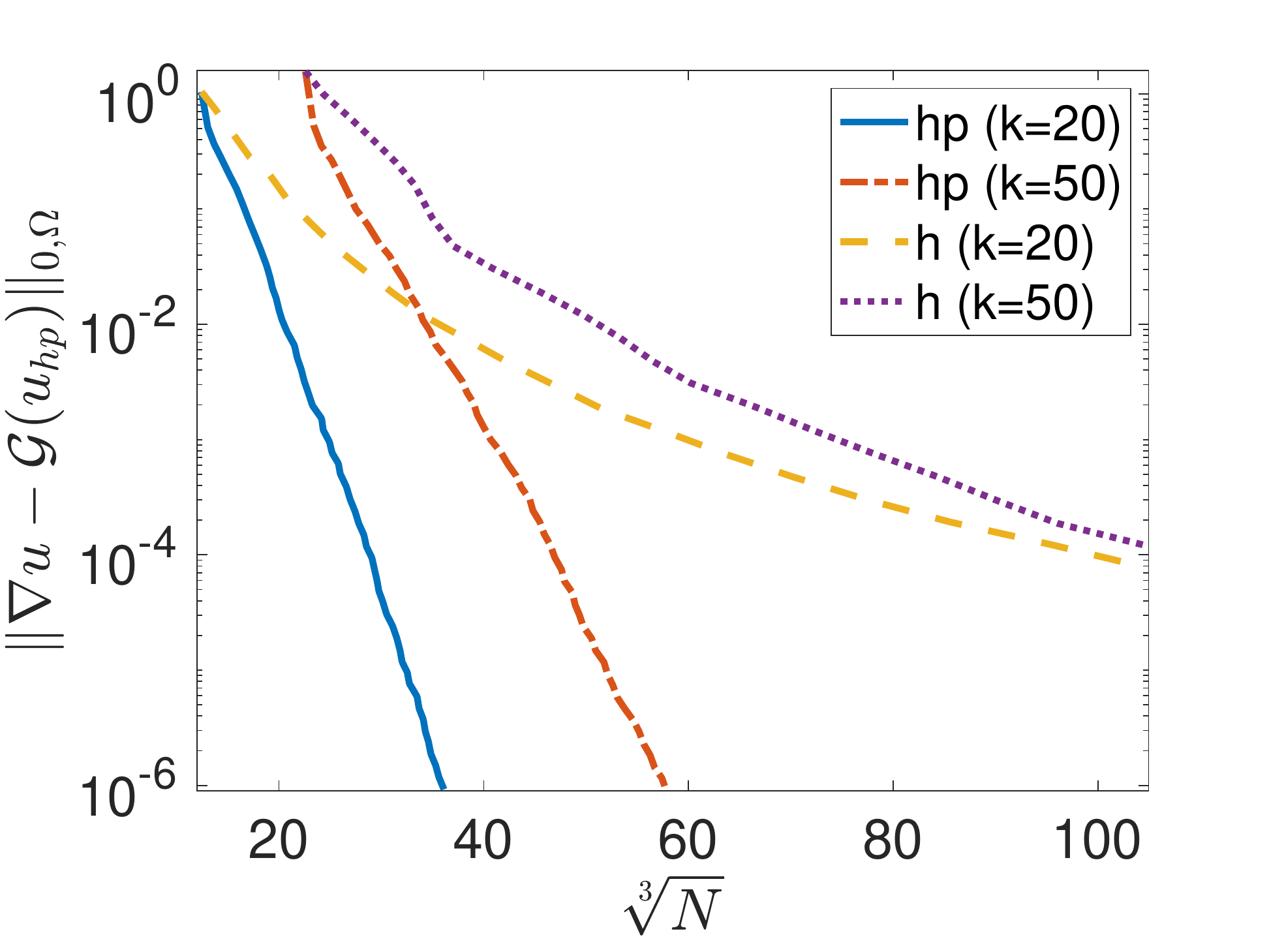}
\includegraphics[width=0.49\textwidth]{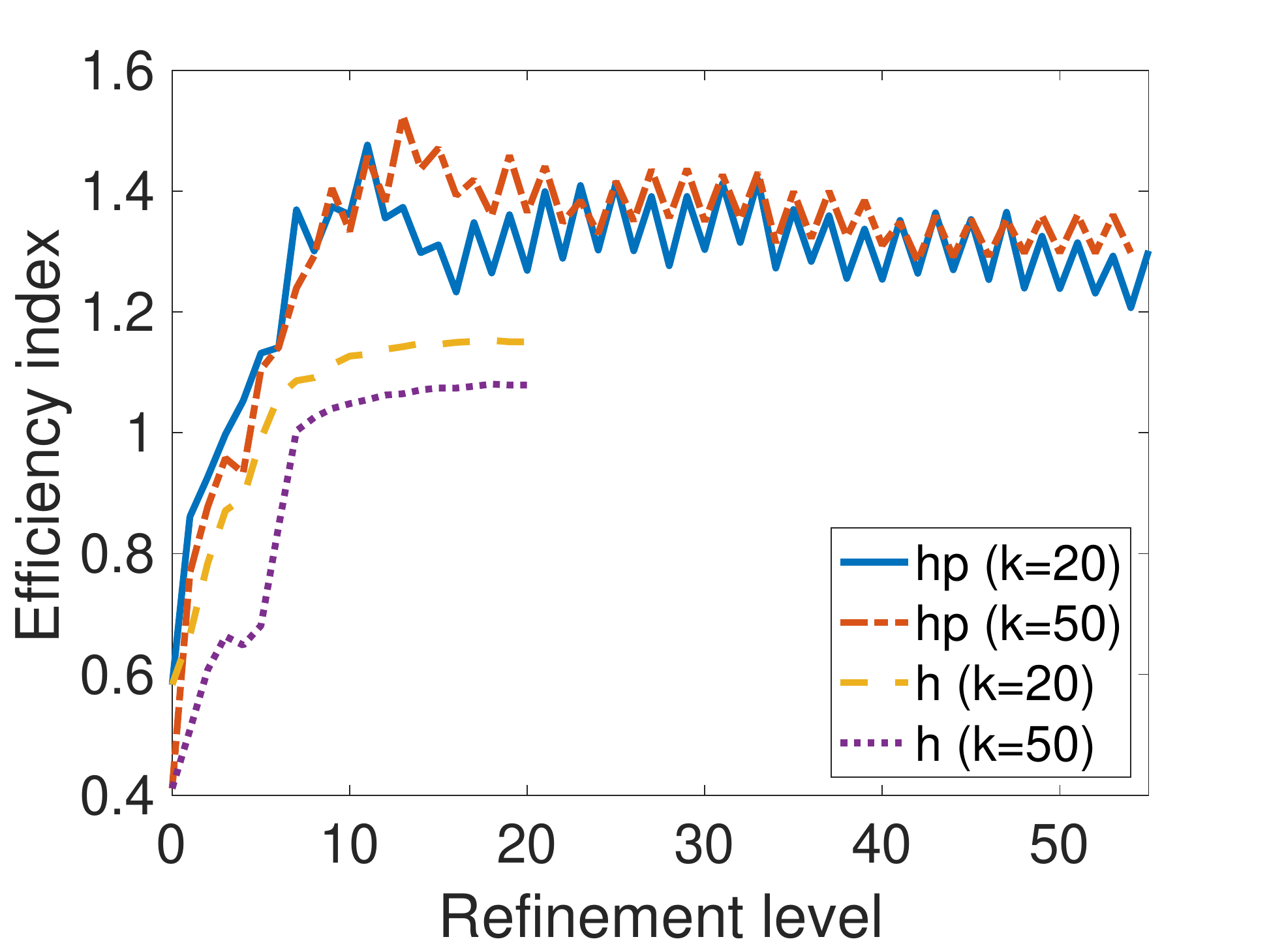}
\caption{Exponential convergence (left) and efficiency indices (right) of the $h$- and $hp$-version  for the example in Section~\ref{example:lshape}.}
\label{fig:Lshape}
\end{figure}
\begin{figure}[tbp]
\centering
\includegraphics[width=0.49\textwidth]{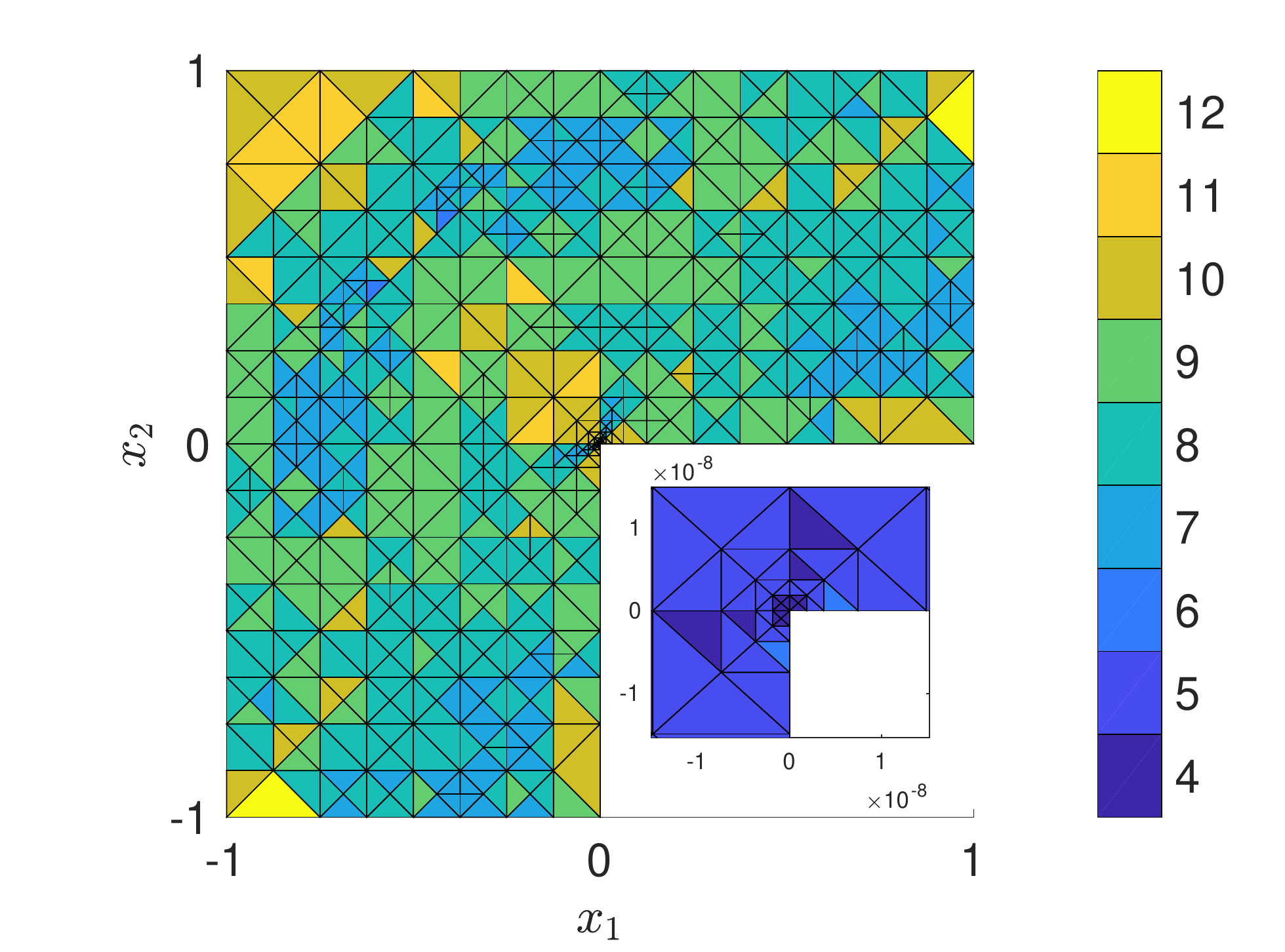}
\includegraphics[width=0.49\textwidth]{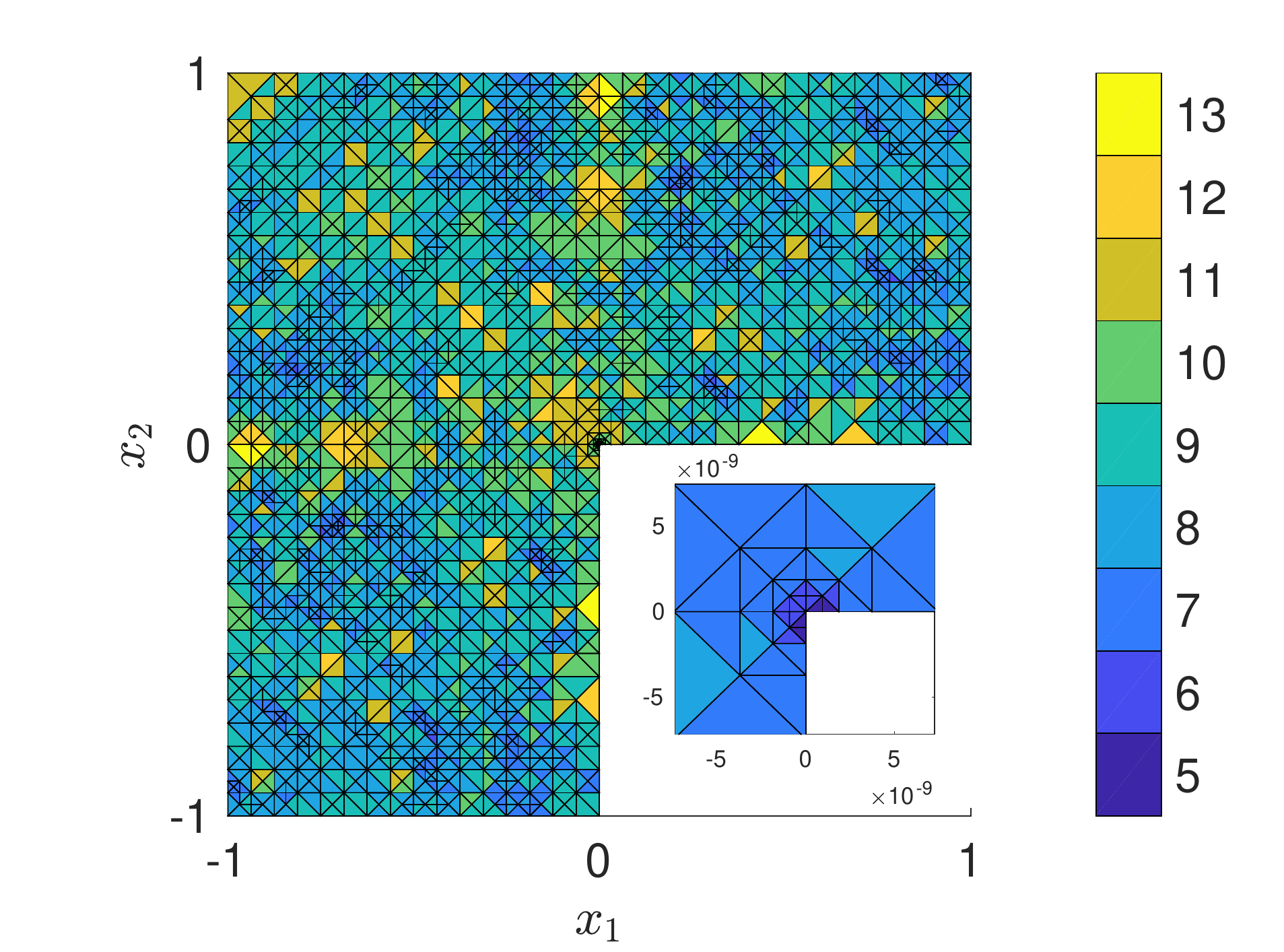}
\caption{$hp$-refined mesh for the example in Section~\ref{example:lshape} with $k=20$ (left) and $k=50$ (right),
where the polynomial degree is indicated with different shading.}
\label{fig:mesh:Lshape}
\end{figure}
\begin{figure}[tbp]
\centering
\includegraphics[width=0.49\textwidth]{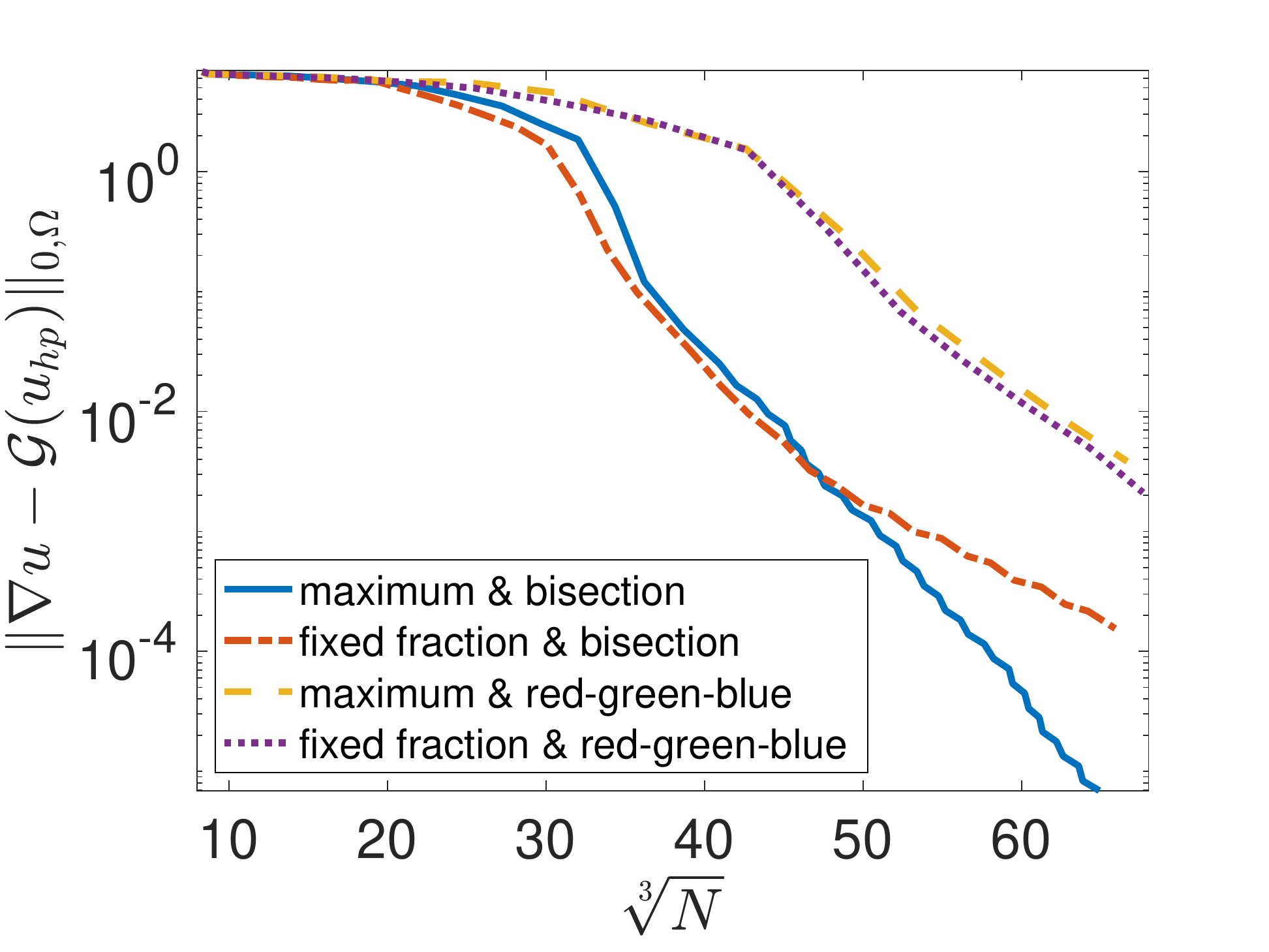}
\includegraphics[width=0.49\textwidth]{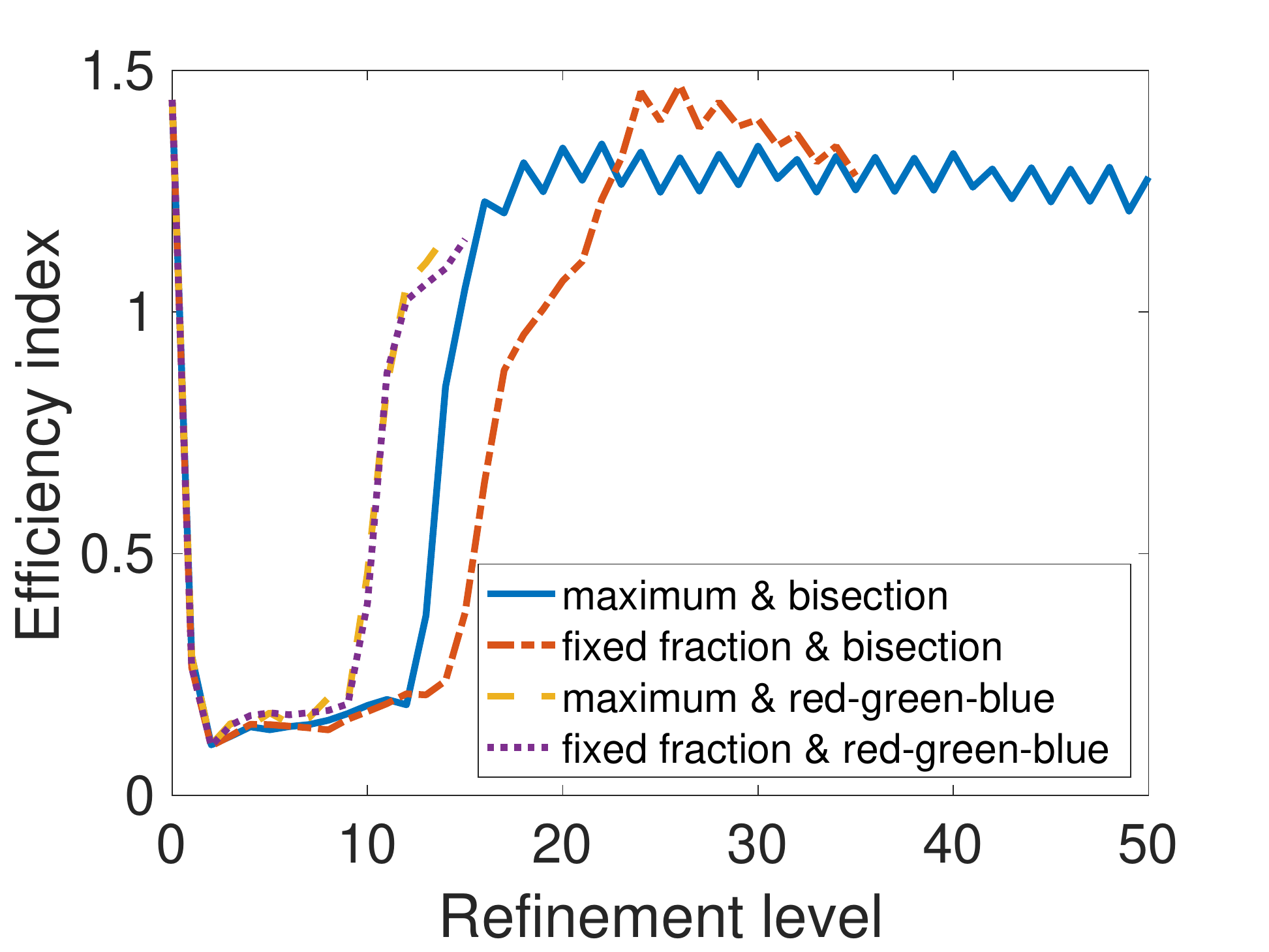}
\caption{Exponential convergence (left) and efficiency indices (right) of different $hp$-refinement 
strategies for the example in Section~\ref{example:lshape} with $k=50$.}
\label{fig:LshapeStrategies}
\end{figure}
\begin{figure}[tbp]
\centering
\includegraphics[width=0.49\textwidth]{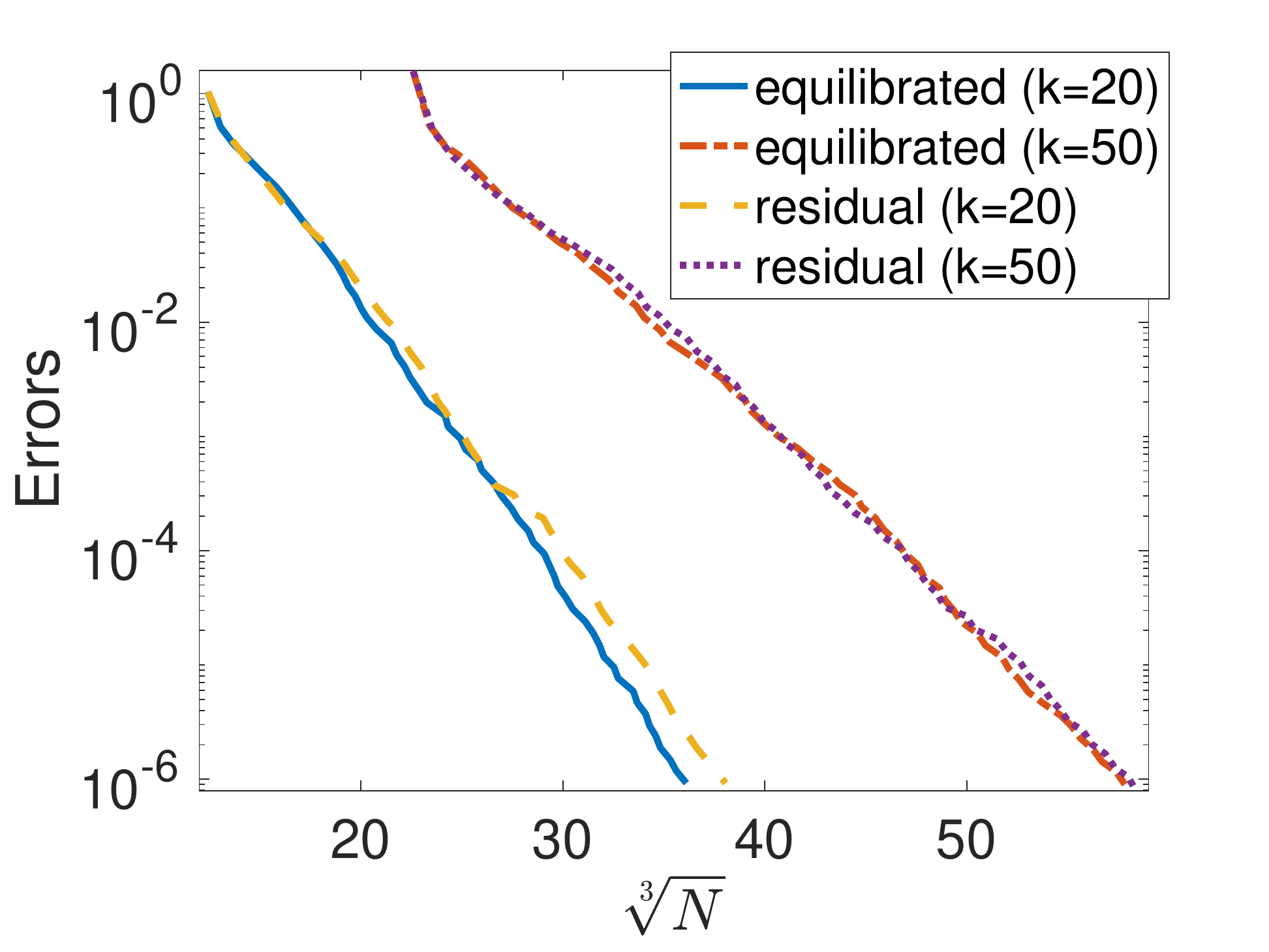}
\includegraphics[width=0.49\textwidth]{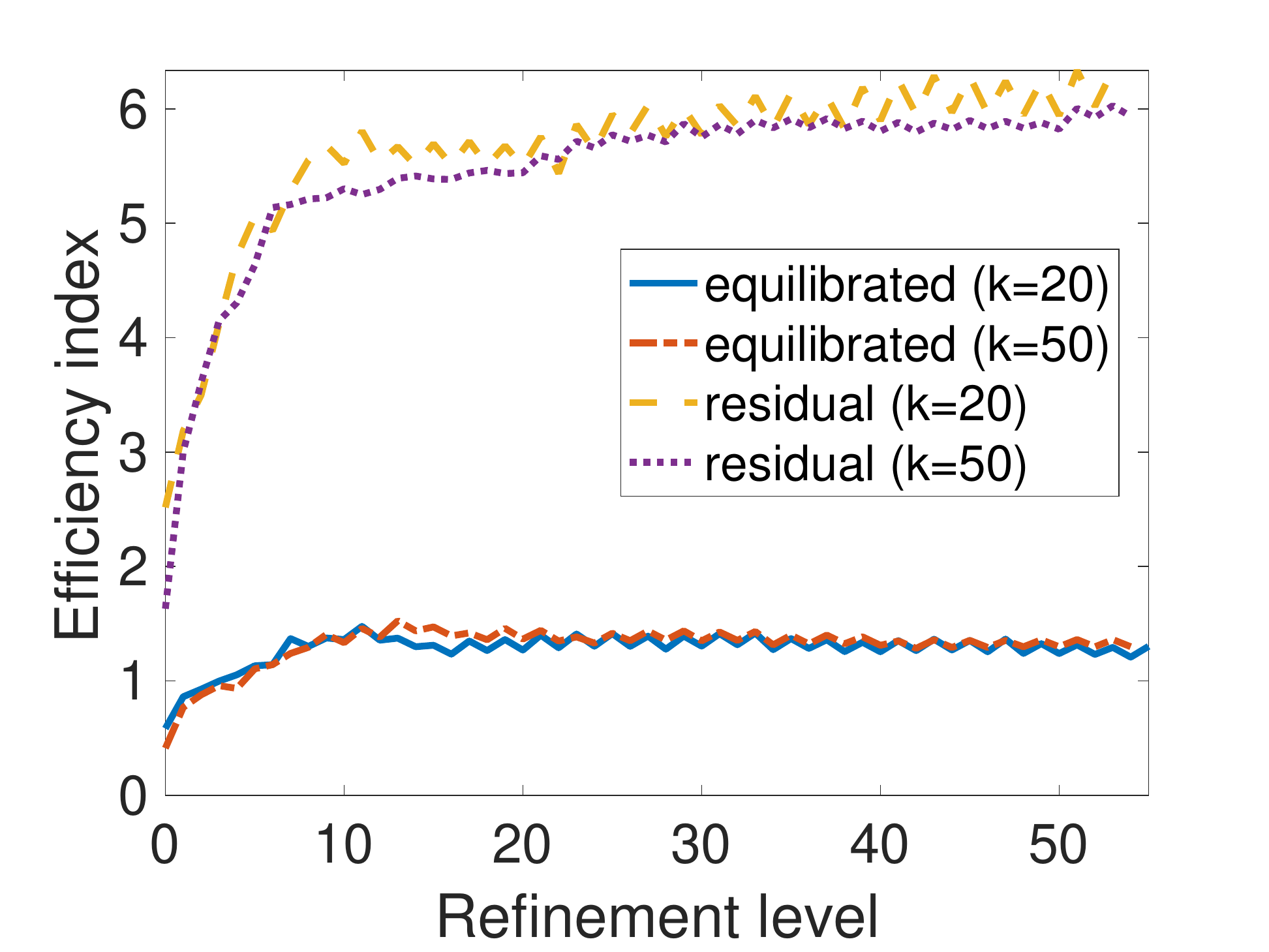}
\caption{Exponential convergence (left) and efficiency indices (right) of the $hp$-version for the equilibrated a posteriori error estimator in comparison to the residual a posteriori error estimator for the example in Section~\ref{example:lshape}.}
\label{fig:LshapeResidual}
\end{figure}
Let $\Omega = (-1,1)^2\backslash((0,1)\times(-1,0))$, $f=0$, and select $g$ such that the solution of \eqref{weak:formulation} is given 
in polar coordinates $(r,\varphi)$ by
\begin{align*}
 u(r,\varphi) = \mathcal{J}_{2/3}(kr)\sin(2\varphi/3),
\end{align*}
where $\mathcal{J}_{2/3}$ denotes the Bessel function of first kind. Note that the gradient of $u$ is singular at the origin; therefore,
adaptive mesh refinement towards the origin is needed.
Notice that we approximate~\eqref{problem} with
$g=\nabla u\cdot\n-iku$, and $g$ is singular at $(0,0)$. As $g$
enters~\eqref{local_mixed} as an essential boundary condition, we use
a high polynomial degree for the flux reconstruction in the patch
associated with $(0,0)$.
\par
For the initial mesh refinement we choose $C_{\mathrm{res}}=2$.
In Figure~\ref{fig:Lshape} we observe algebraic convergence of adaptive $h$-refinement and 
exponential convergence of adaptive $hp$-refinement for $k=20,50$. 
More precisely, the convergence for $hp$-refinement is of the form $\exp(-b\sqrt[3]{N})$, and we
observe that $b\approx 0.61$ for $k=20$ and $b\approx0.41$ for $k=50$.
This indicates that the rate of exponential convergence deteriorates for larger $k$.
We note that higher wavenumbers result in more $h$-refinement,
cf. Figure~\ref{fig:mesh:Lshape}, 
and 
smaller elements generally result in a slower rate of exponential convergence, cf. Figure~\ref{fig:Hankel_uniform_H_P}. Therefore, we cannot
conclude whether the rate deterioration is a direct result of the
increased wavenumber, or is due to the increased $h$-refinement, which 
results from the higher wavenumber.
In all cases, the efficiency indices are asymptotically close to 1.
\par
Figure~\ref{fig:mesh:Lshape} displays the final $hp$-refined mesh for $k=20,50$. Note that the displayed zoom at the
re-entrant corner shows low polynomial degrees close to the origin.
\par
For comparison, in Figure~\ref{fig:LshapeStrategies} we compare the four variants of $hp$-refinement
described in the previous example. We observe again that the fewer $h$-refinements performed by bisection
is advantageous over the larger $h$-refinements performed by red-green-blue-refinement. For higher accuracy, the maximum marking strategy appears to be more effective than the fixed fraction marking strategy. In all four cases, the efficiency indices are asymptotically close to 1 and the four $hp$-strategies are all able to overcome the pre-asymptotic region even when starting from $p=1$ and $h=1/4$.
\par
In Figure~\ref{fig:LshapeResidual}, we compare the equilibrated a posteriori error estimator to the residual a posteriori error
estimator \cite{SZ2015}
\begin{align*}
\eta_{hp,residual}^2
&=  \sum_{T\in\T} \frac{h_T^2}{p_T^2}\|\Delta u_{hp} + k^2 u_{hp} + f\|_{0,T}^2
+ \sum_{E\in\E(\Omega)} \beta\frac{\h}{\p}\|\jump{\nabla_h u_{hp}}\|_{0,E}^2\\
&\quad+ \sum_{E\in\E(\Omega)} \alpha\frac{\p^2}{\h}\|\jump{u_{hp}}\|_{0,E}^2
+\sum_{E\in\E(\partial\Omega)} \h \| g-\nabla_h u_{hp}\cdot \bm{n} + iku_{hp}\|_{0,E}^2.
\end{align*}
We observe that the error $\|\nabla_h(u-u_{hp})\|_{0,\Omega}$ for the residual a posteriori error estimator 
is very close to the error $\|\nabla u-\G(u_{hp})\|_{0,\Omega}$ for the equilibrated a posteriori error estimator.
In fact, in the case of $k=50$ both errors overlap. The difference, however, is in the efficiency. 
The efficiency indices for the equilibrated a posteriori error estimator are asymptotically close
to 1 and robust in $p$; by contrast, the efficiency indices for the residual a posteriori error estimator 
are close to 5 and show a small but persistent growth in $p$.

\subsection{Internal reflection/refraction}\label{example:reflect}
\begin{figure}[tbp]
\centering
\includegraphics[width=0.49\textwidth]{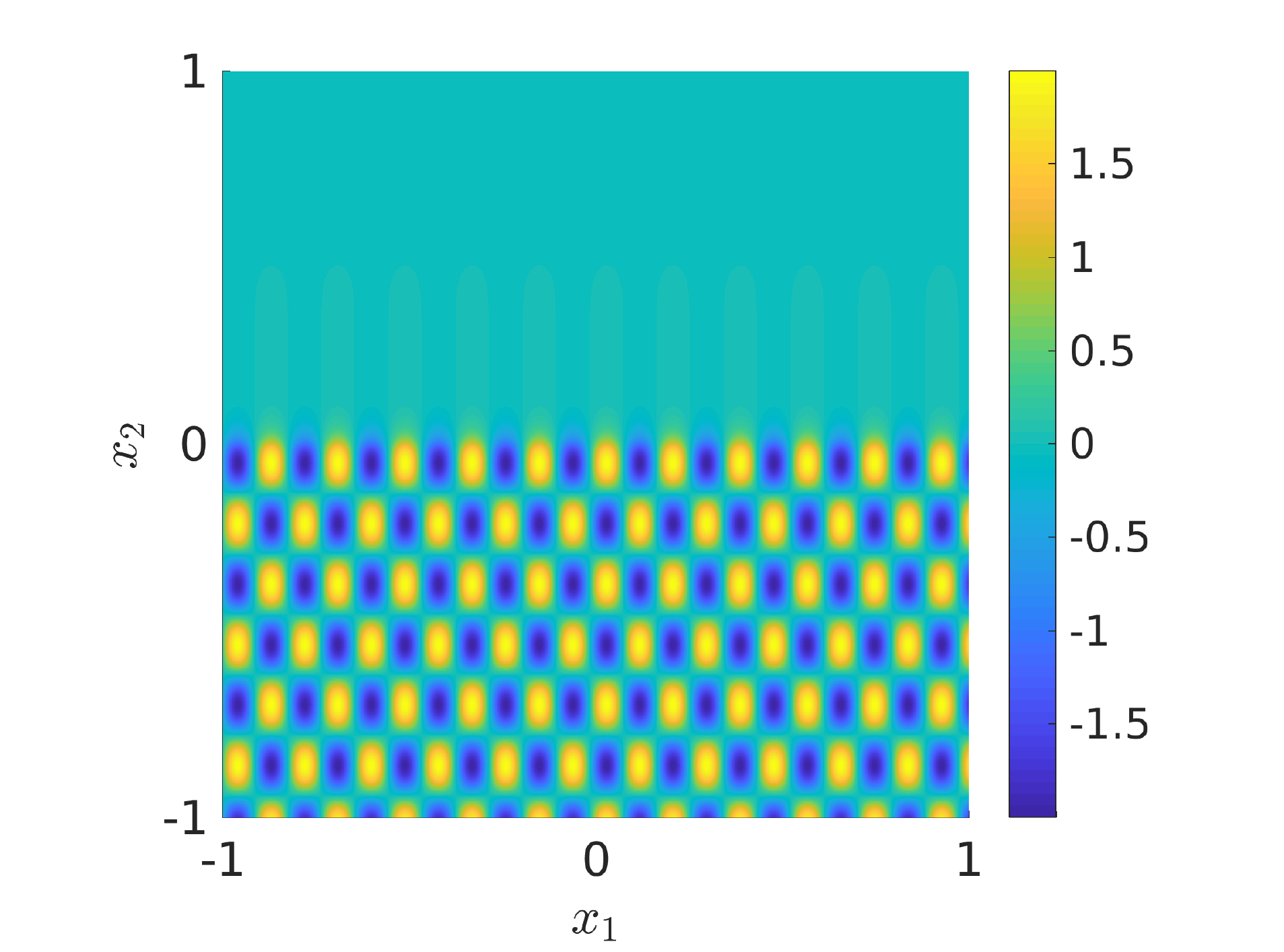}
\includegraphics[width=0.49\textwidth]{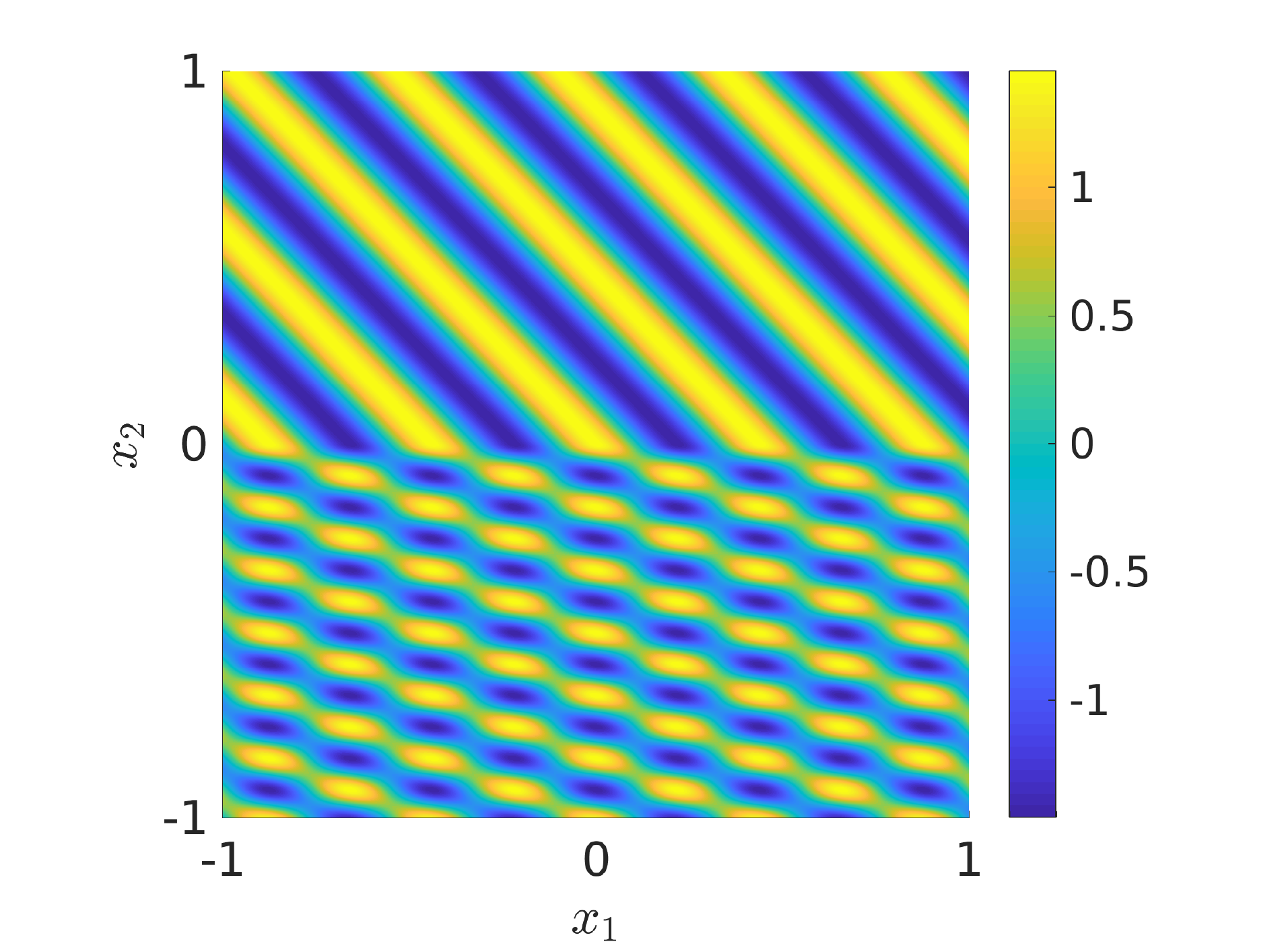}
\caption{Real parts of solutions with $29^\circ$ reflection (left) and $69^\circ$  refraction (right) for $k=20$.}
\label{fig:solution}
\end{figure}
\begin{figure}[tbp]
\centering
\includegraphics[width=0.49\textwidth]{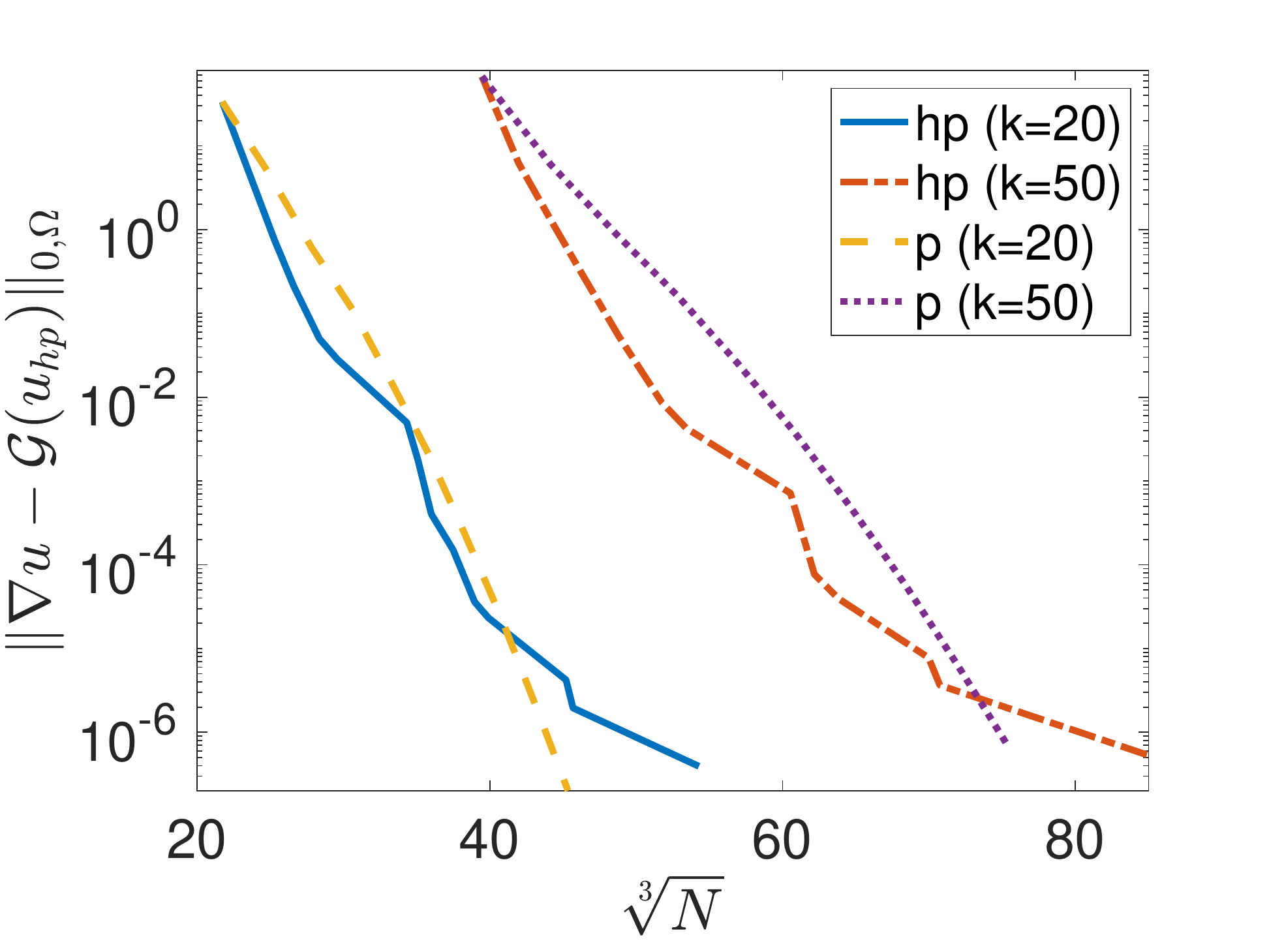}
\includegraphics[width=0.49\textwidth]{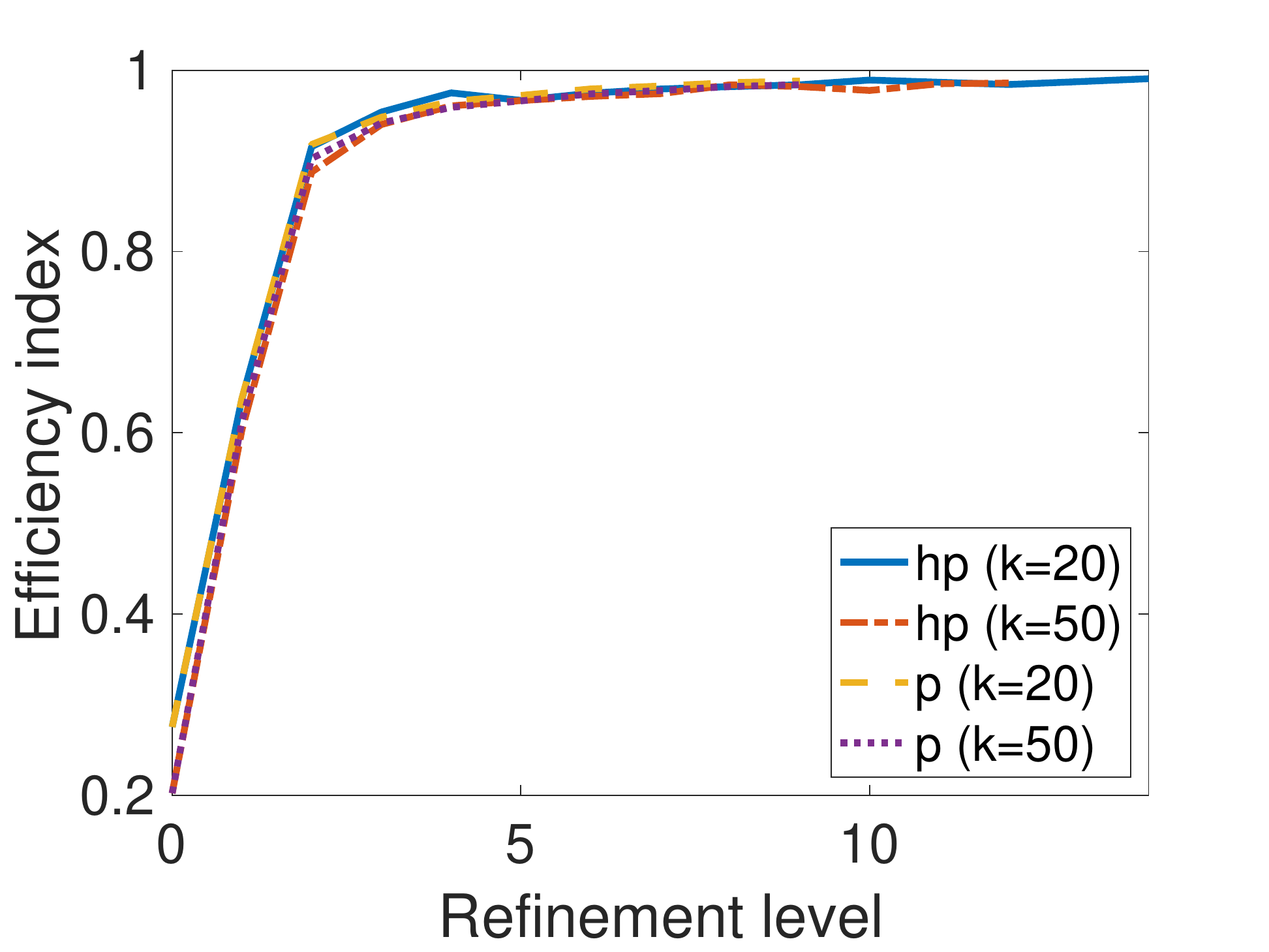}
\caption{Exponential convergence (left) and efficiency indices (right) of the $p$- and $hp$-version  for the example in Section~\ref{example:reflect} with $29^\circ$ reflection.}
\label{fig:29}
\end{figure}
\begin{figure}[tbp]
\centering
\includegraphics[width=0.49\textwidth]{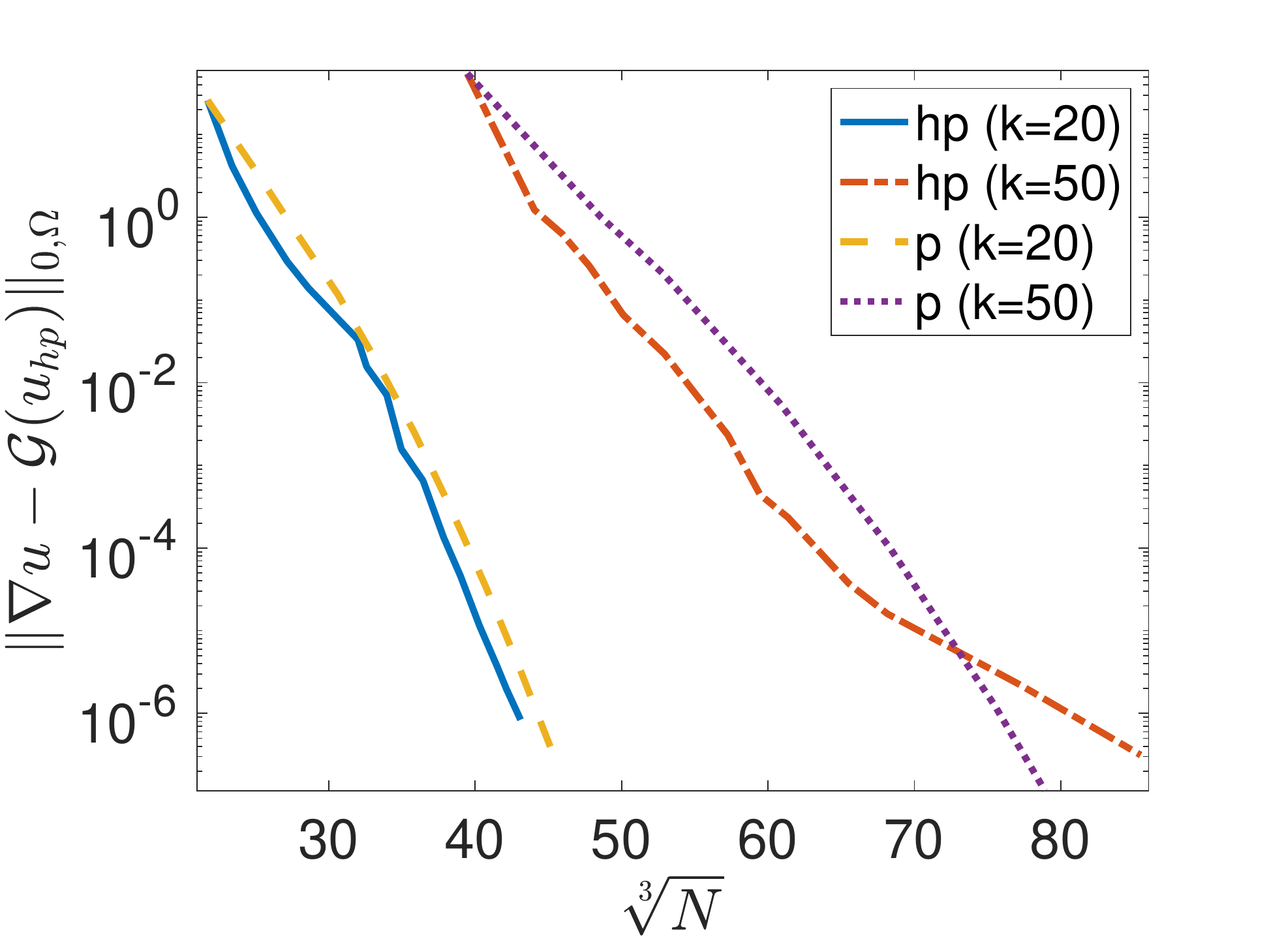}
\includegraphics[width=0.49\textwidth]{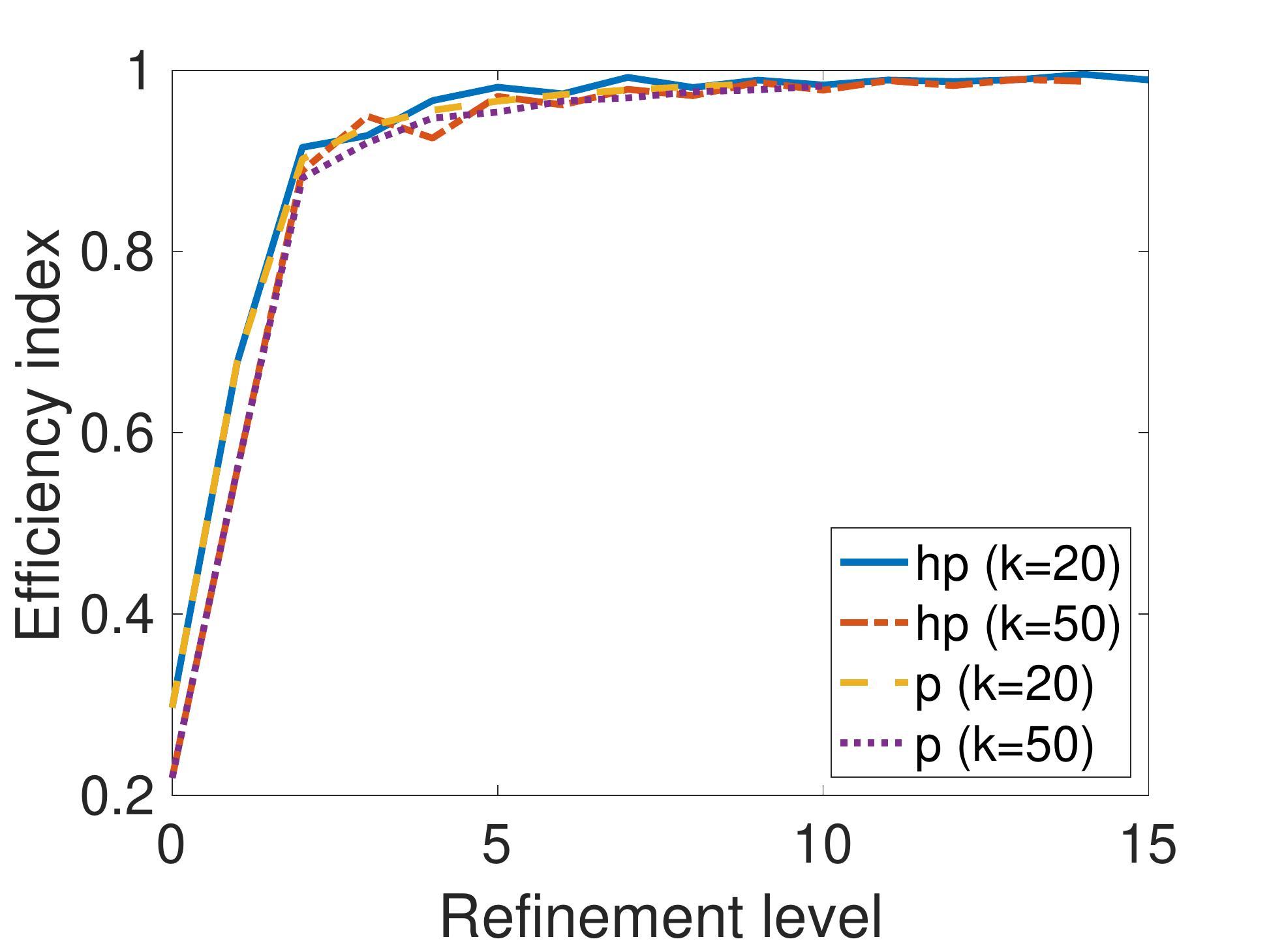}
\caption{Exponential convergence (left) and efficiency indices (right) of the $p$- and $hp$-version  for the example in Section~\ref{example:reflect} with $69^\circ$ refraction.}
\label{fig:69}
\end{figure}
\begin{figure}[tbp]
\centering
\includegraphics[width=0.49\textwidth]{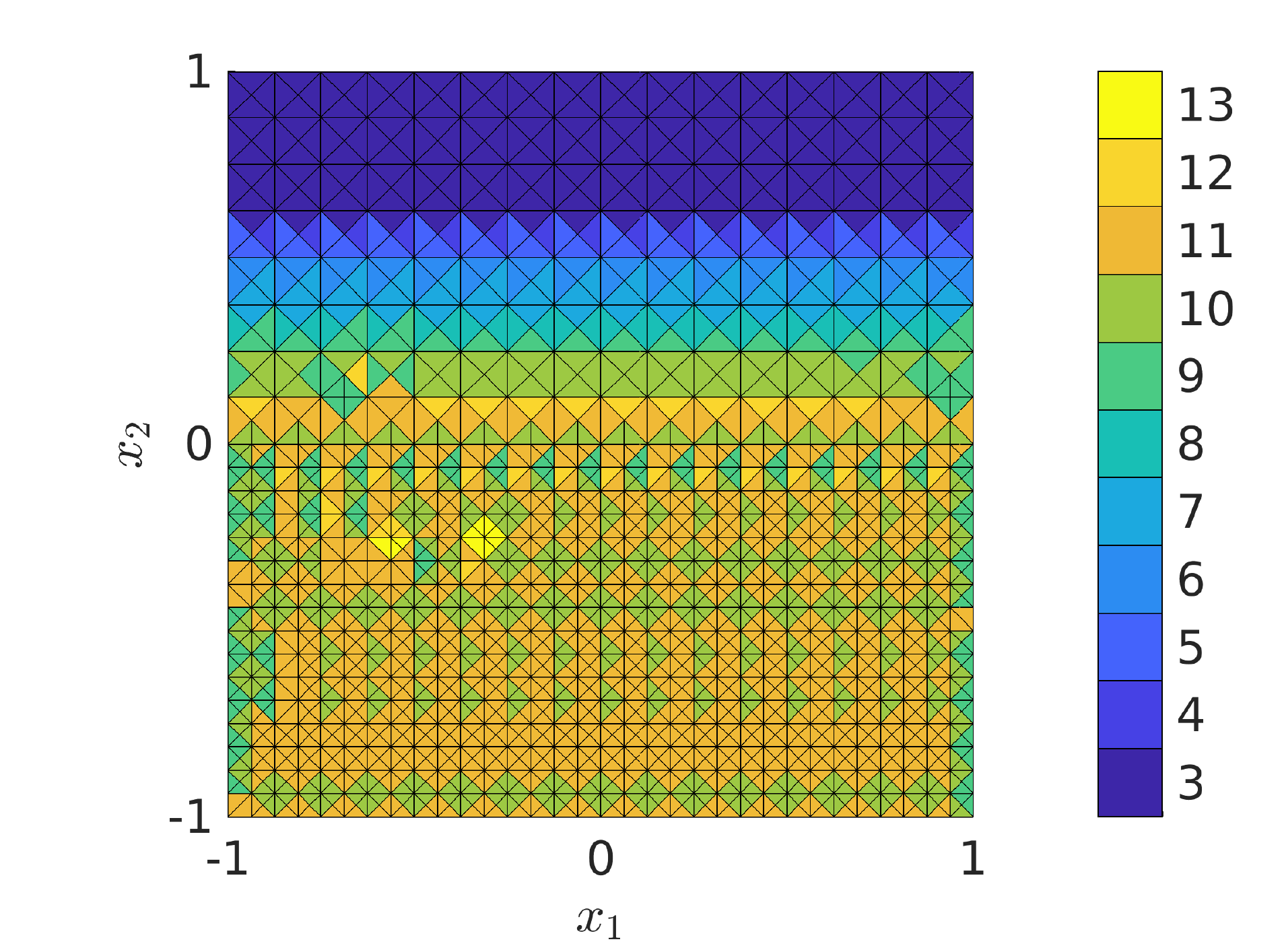}
\includegraphics[width=0.49\textwidth]{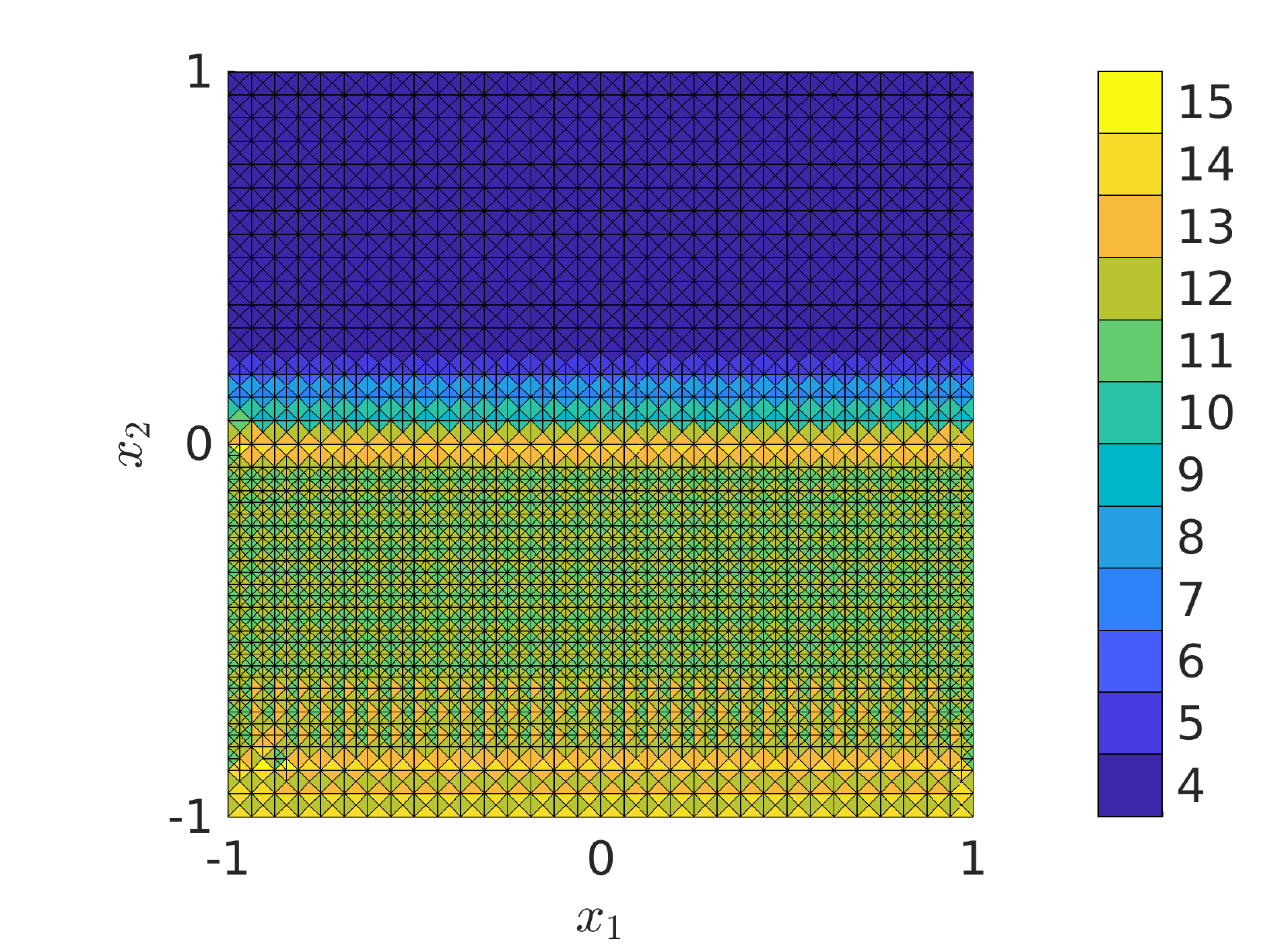}
\caption{$hp$-refined mesh for the example in Section~\ref{example:reflect} with $29^\circ$ reflection using $k=20$ (left) and $k=50$ (right),
where the polynomial degree is indicated with different shading.}
\label{fig:mesh:reflection29}
\end{figure}
\begin{figure}[tbp]
\centering
\includegraphics[width=0.49\textwidth]{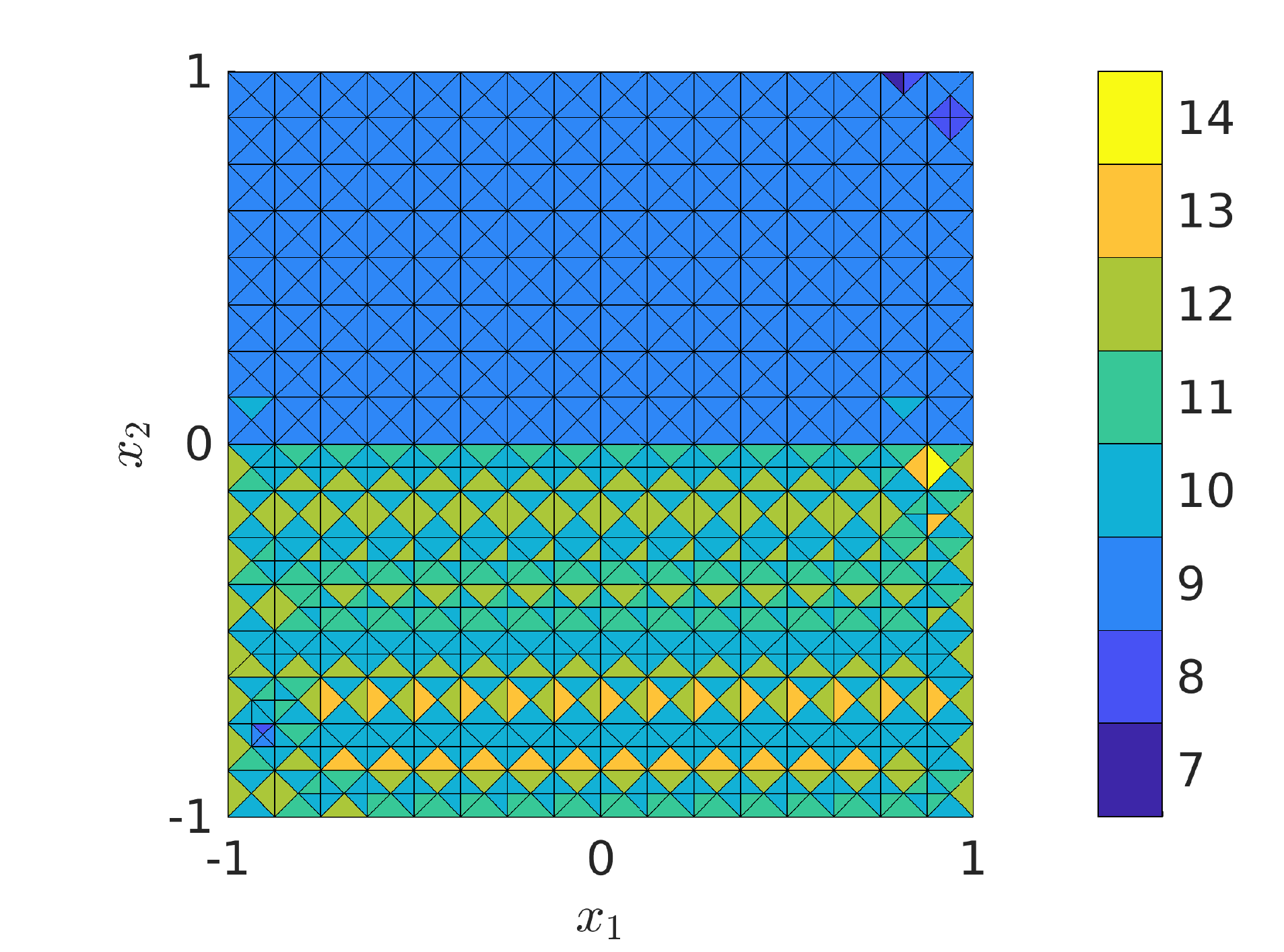}
\includegraphics[width=0.49\textwidth]{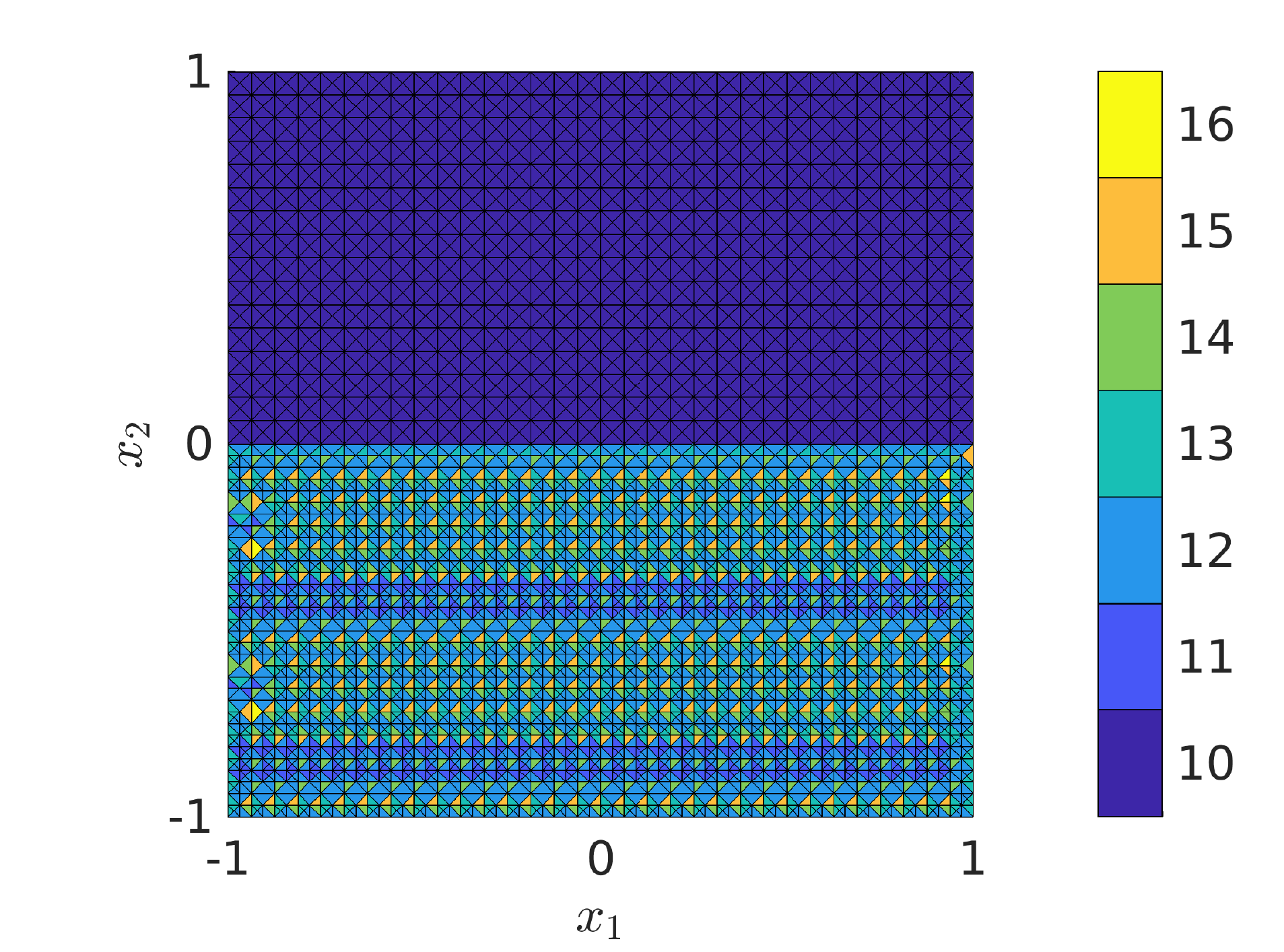}
\caption{$hp$-refined mesh for the example in Section~\ref{example:reflect} with $69^\circ$ refraction using $k=20$ (left) and $k=50$ (right),
where the polynomial degree is indicated with different shading.}
\label{fig:mesh:reflection69}
\end{figure}
Although not covered in the theoretical part, we now consider the benchmark from \cite[Section 6.3]{KMW2015}
with non-constant refractive index $\epsilon_r$; hence,
we consider the following problem
\begin{align*}
-\Delta u - k^2\epsilon_r u &= 0\quad\textrm{in }\Omega,\\
\nabla u\cdot\n - ik\sqrt{\epsilon_r}u &= g\quad\textrm{on }\partial\Omega,
\end{align*}
where
\begin{align*}
   \epsilon_r(x) = \left\{
   \begin{array}{l}
   n_1^2\quad \text{if } x_2 < 0,\\
   n_2^2\quad \text{if } x_2 \geq 0.
   \end{array}
   \right.
\end{align*}
For $\Omega=(-1,1)^2$, $n_1=2$, $n_2=1$, and $0\leq \theta < \pi/2$, one can
show that this problem admits the following solution
\begin{align*}
u(x)= \left\{
   \begin{array}{cr}
   (1+R)\exp\left(i(K_1x_1+K_3x_2)\right)& \text{if } x_2 \geq 0,\\
    \exp\left(i(K_1x_1+K_2x_2))\right) + R\exp\left(i(K_1x_1-K_2x_2)\right)& \text{if } x_2 < 0,
      \end{array}
   \right.
\end{align*}
where $K_1 = kn_1\cos(\theta)$, $K_2 = kn_1\sin(\theta)$, $K_3=k\sqrt{n_2^2-n_1^2\cos^2(\theta)}$,  and 
\[ R = -(K_3-K_2)/(K_3+K_2).\]
There exists a critical angle $\theta^*$ such that for $\theta>\theta^*$ the wave is refracted, and for $\theta<\theta^*$
the wave is internally reflected;
therefore, we compute two examples with $\theta_1=29^\circ, 69^\circ$, in order to demonstrate internal reflection and refraction, respectively.
\par

The solutions for $k=20$ and $\theta_1,\theta_2$ are displayed in Figure~\ref{fig:solution}.
Figures~\ref{fig:29} and \ref{fig:69} show exponential convergence for both $p$- and adaptive $hp$-refinement,
and the efficiency indices are again asymptotically close to 1.
Note that the initial mesh is chosen such that \eqref{eq:choice} is fulfilled with $C_{\mathrm{res}}=1/2$ and the jump of the refractive index is resolved by the mesh,
otherwise strong anisotropic mesh refinement towards the interface would be needed for fast convergence.
Interestingly, we observe in Figure~\ref{fig:29} that initially $hp$-refinement outperforms $p$-refinement; however, we generally expect $p$-refinement to
perform better, and indeed this occurs towards the end of the refinement.
Figures~\ref{fig:mesh:reflection29} and \ref{fig:mesh:reflection69} display the final $hp$-refined meshes for $k=20,50$, and $\theta_1$ and $\theta_2$ respectively.

\subsection{Gaussian beam simulation}\label{example:beam}
\begin{figure}[tbp]
\centering
\includegraphics[width=0.49\textwidth]{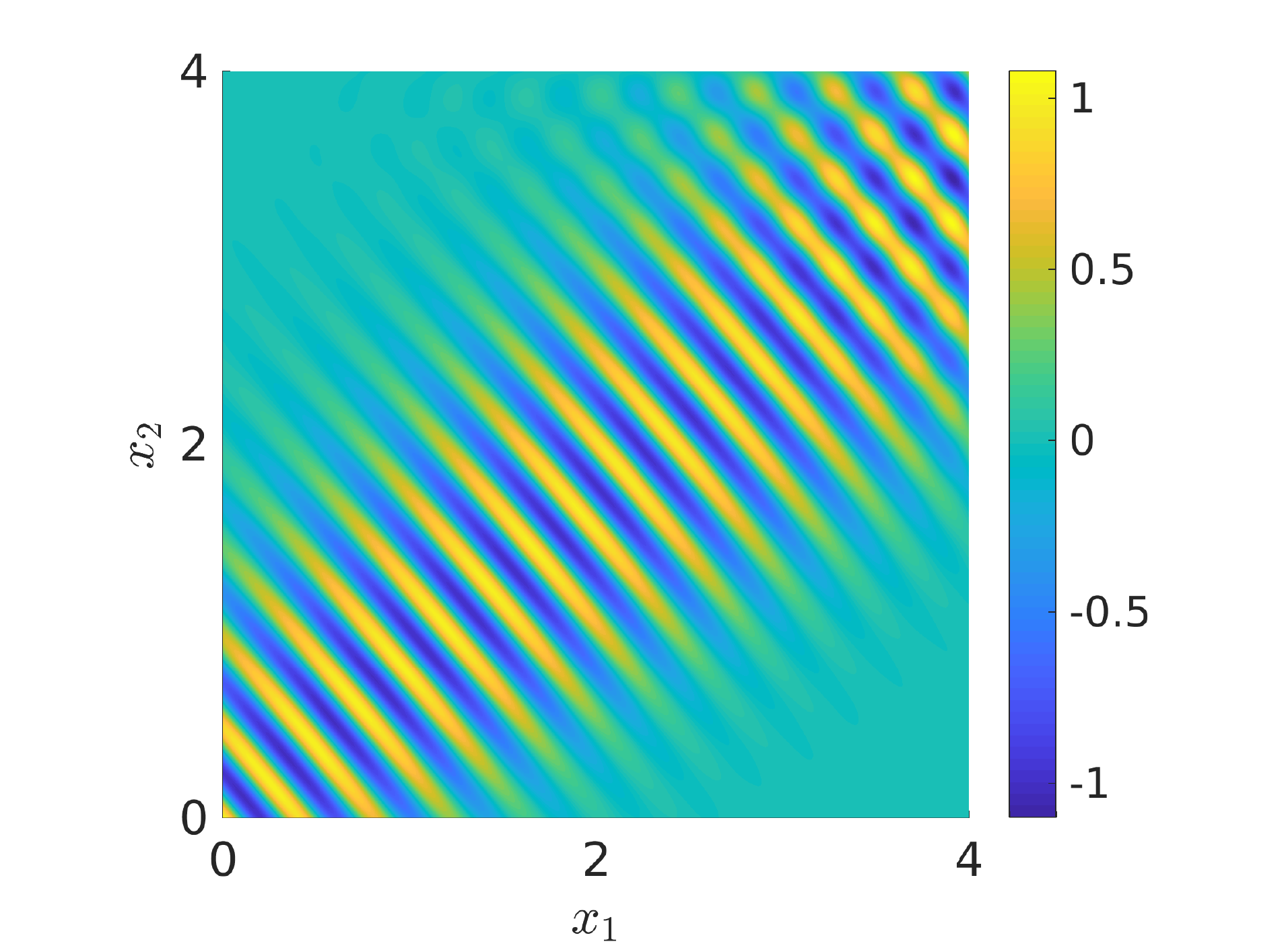}
\includegraphics[width=0.49\textwidth]{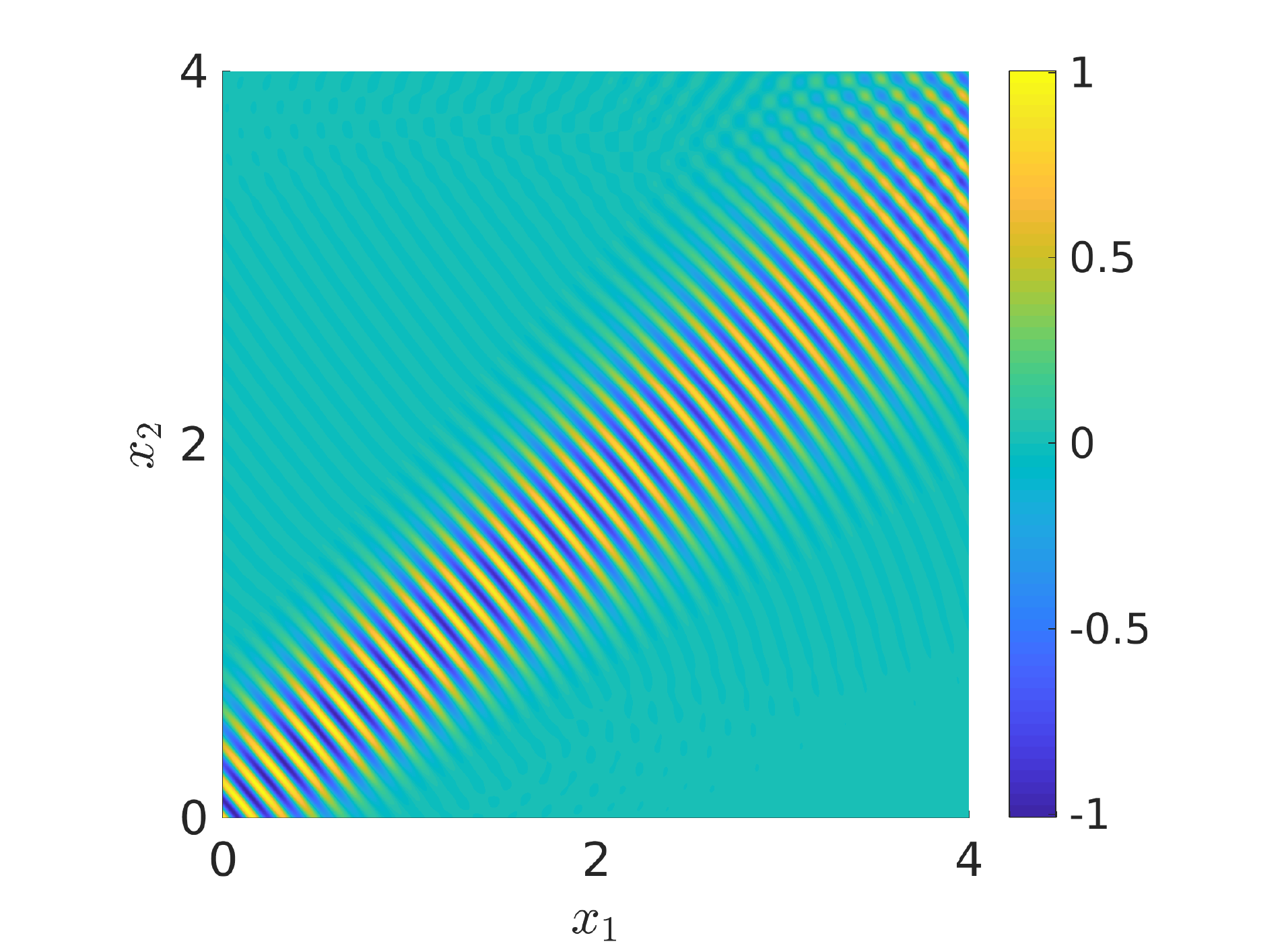}
\caption{Real parts of Gaussian beam approximations for $k=20$ (left) and $k=50$ (right).}
\label{fig:BeamSolution}
\end{figure}
\begin{figure}[tbp]
\centering
\includegraphics[width=0.49\textwidth]{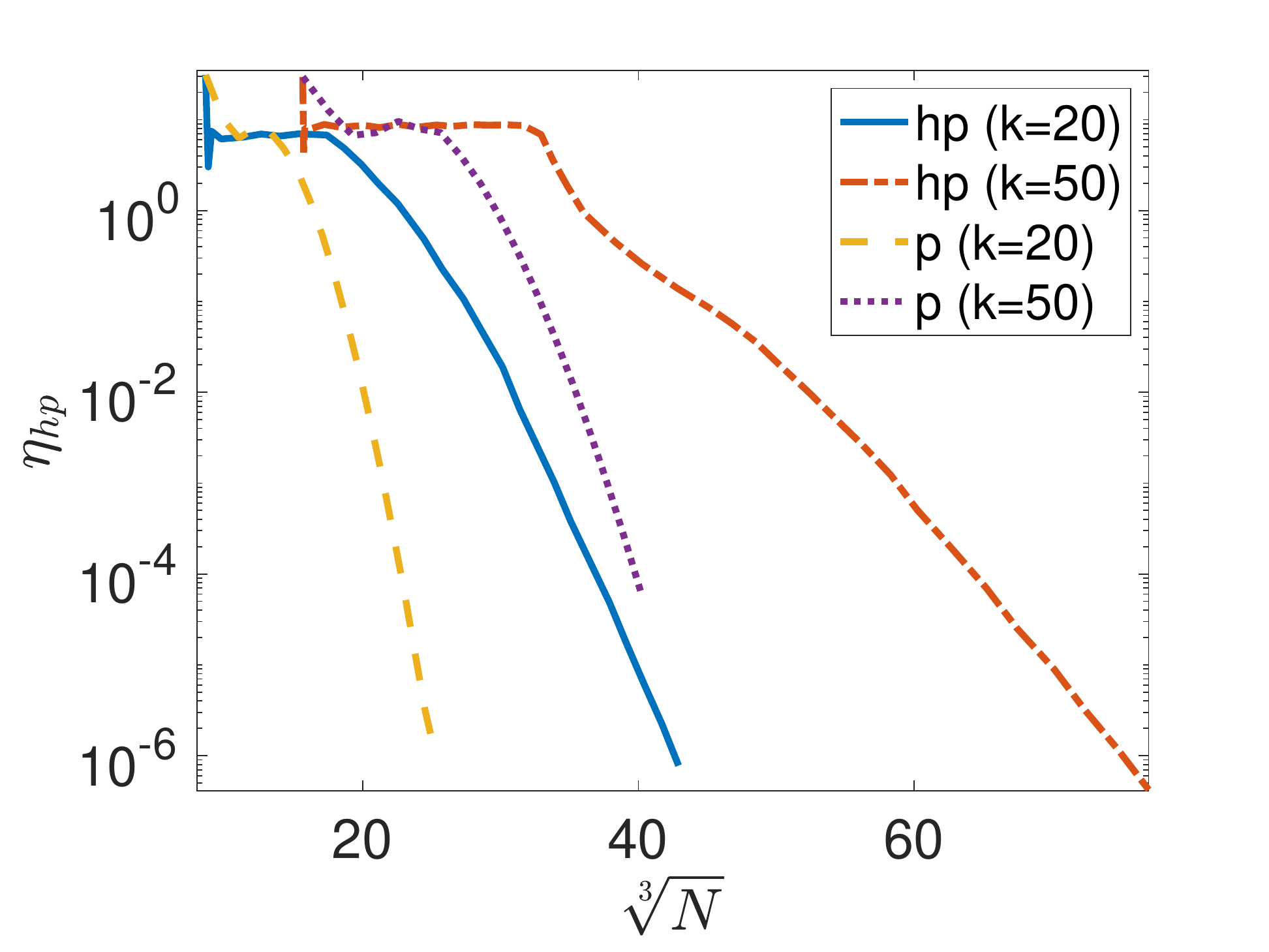}
\caption{Exponential convergence of the $p$- and $hp$-version for the Gaussian beam simulation in Section~\ref{example:beam}.}
\label{fig:BeamConvergence}
\end{figure}
\begin{figure}[tbp]
\centering
\includegraphics[width=0.49\textwidth]{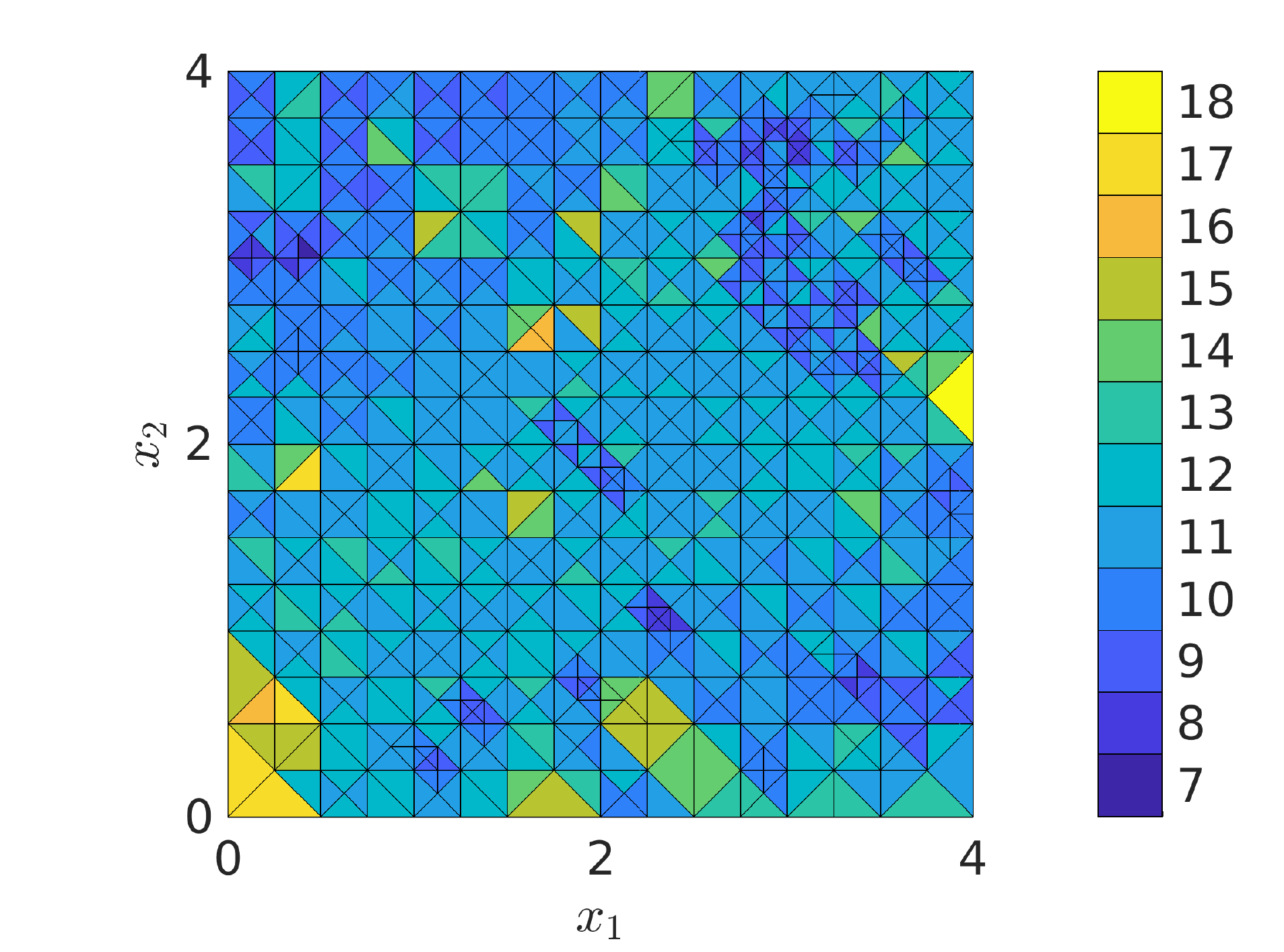}
\includegraphics[width=0.49\textwidth]{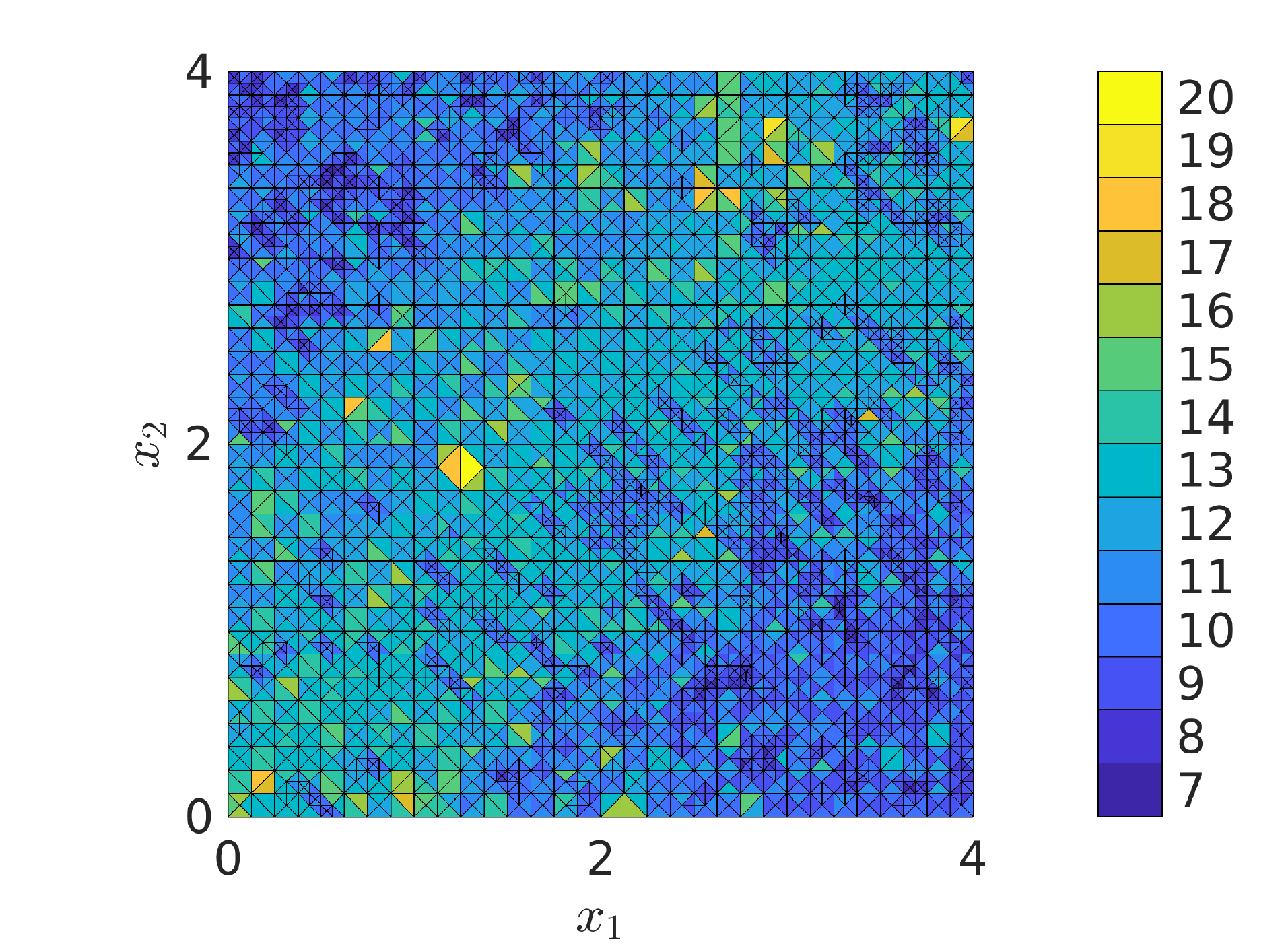}
\caption{$hp$-refined mesh for the Gaussian beam simulation in Section~\ref{example:beam} with $k=20$ (left) and $k=50$ (right),
where the polynomial degree is indicated with different shading.}
\label{fig:BeamMesh}
\end{figure}
In the last example, we consider a Gaussian beam simulation similar to the one in \cite[Section 3.7]{PD2017}.
We choose the domain $\Omega=(0,4)^2$, $f=0$,
and the inhomogeneous impedance boundary condition $g$ corresponding to the 
fundamental Gaussian beam mode that satisfies the paraxial wave equation, which reads in polar coordinates as
\begin{align*}
	v(r,\varphi) = \frac{w_0}{w} \exp\left( \frac{-r^2}{w^2} - ikz - \frac{i\pi r^2}{\lambda R} + i\theta_0\right),
\end{align*}
where $z(r,\varphi)$ is the radius of the orthogonal projection of $(r,\varphi)$ onto the direction of propagation,
$w_0$ is the beam waist radius, $R(z)$ is the radius of curvature, $w(z)$ is the beam radius, and $\phi_0(z)$ is the Gaussian beam phase shift.
We choose a $40^{\circ}$ angle for the direction of the beam, and the beam waist radius $w_0=8\pi/k$.
For the other variables we have that $\lambda=2\pi/k$,
\begin{align*}
R(z) = z + \frac{1}{z}\left( \frac{\pi w_0^2}{\lambda} \right)^2,\quad
w(z) = w_0 \left( 1 + \left( \frac{\lambda z}{\pi w_0^2} \right)^2\right)^{1/2},\quad\text{and}\quad
\tan \phi_0(z) = \frac{\lambda z}{\pi w_0^2}.
\end{align*}
Two Gaussian beam approximations for $k=20,50$ are displayed
in Figure~\ref{fig:BeamSolution}.
\par
Since the exact solution is not known in this particular example,
we only plot the values for the equilibrated a posteriori error estimator for $k=20,50$
in Figure~\ref{fig:BeamConvergence}, whose values we have demonstrated
in the previous experiments should match well with those of the true error.
For the initial mesh construction we take the rather large value $C_{\mathrm{res}}=2$; hence, we observe a pre-asymptotic
region for the convergence,
which in case of $hp$-refinement is longer than for $p$-refinement.
Due to this, $hp$-refinement leads to a higher number of degrees of freedom for the same accuracy than $p$-refinement.
However, both $p$- and $hp$-refinement lead to exponential convergence of the a posteriori error estimator.
The two final $hp$-refined meshes are displayed in Figure~\ref{fig:BeamMesh}.
In particular, for $k=50$ we observe that the polynomial degree is higher closer to the beam
than further away from the beam towards the upper left  and lower right corners of the domain.

\section{Conclusion}
We have presented an equilibrated a posteriori error estimator for the indefinite Helmholtz problem based on
a non trivial extension of the unified theory for the elliptic problem using a shifted Poisson problem.
We have shown that the presented error estimator is both reliable and efficient, providing that
the equilibrated flux and potential reconstructions are suitably chosen. We have provided several numerical experiments which verify
that, after escaping the pollution regime, the a posteriori error estimator is efficient and reliable. In contrast to a
residual based a posteriori error estimator, we demonstrated that the presented error estimator is robust in the polynomial degree.

Note that the analysis for the potential reconstruction in Section~\ref{sec:potential} is a purely 2D argument.
A different analysis approach for the 3D case has recently been proposed in \cite{EV2017}, together with the extension of the 2D stability result of \cite{BPS2009}. Therefore, a potential extension of this current work would be to consider the three dimensional case.

\section*{Acknowledgments}
The authors thank the anonymous referees for their valuable comments and suggestions that lead to an improvement of the presentation of the paper.


\end{document}